\documentclass[12pt,reqno,twoside]{amsart}
\usepackage{amsmath}
\usepackage{amssymb}
\usepackage[all]{xy}
\usepackage{epsfig}
\usepackage{color}
\usepackage{mathabx}
\usepackage{tikz}
\usepackage[english]{babel}
\usepackage{verbatim}
\usepackage{enumitem}
\usepackage{soul}

\makeatletter
\let\@wraptoccontribs\wraptoccontribs
\makeatother

\numberwithin{equation}{section}

\title{Rel leaves of the Arnoux-Yoccoz surfaces}  
\author[Hooper]{W. Patrick Hooper}
\address{City College of New York and CUNY Graduate Center,
{\tt whooper@ccny.cuny.edu}}
\author[Weiss]{Barak Weiss}
\address{Tel Aviv University, 
{\tt barakw@post.tau.ac.il}}

\contrib[with an appendix by]{Lior Bary-Soroker}
\address{Tel Aviv University, 
{\tt barylior@post.tau.ac.il}}
\contrib[]{Mark Shusterman}
\address{Tel Aviv University, 
{\tt markshus@mail.tau.ac.il}}
\contrib[]{Umberto Zannier}
\address{Scuola Normale Superiore, 
{\tt u.zannier@sns.it}}

\font\sn = cmssi8 scaled \magstep0
\font\si = cmti8 scaled \magstep0
\long\def\compat#1{\ifdraft{{\color{blue}\si Pat says ``#1'' }}\else\ignorespaces\fi}
\long\def\combarak#1{\ifdraft{\color{red}\sn Barak says ``#1'' }\else\ignorespaces\fi}

\newcommand\hol{\mathrm{hol}}

\newif\ifdraft\drafttrue
\draftfalse

\newcommand\name[1]{\label{#1}{\ifdraft{\sn [#1]}\else\ignorespaces\fi}}

\newcommand\eq[2]{{\ifdraft{\ \tt [#1]}\else\ignorespaces\fi}\begin{equation}\label{#1}{#2}\end{equation}}
\newcommand {\equ}[1]{\eqref{#1}}

\newcommand{\MM}{{\mathcal{M}}}
\newcommand{\HH}{{\mathcal{H}}}

\newcommand{\Q}{{\mathbb {Q}}}

\newcommand{\GG}{{\mathcal{G}}}

\newcommand{\R}{{\mathbb{R}}}

\newcommand{\TTT}{{\mathbb{T}}}
\newcommand{\Res}{{\mathrm{Res}}}
\newcommand{\Z}{{\mathbb{Z}}}

\newcommand{\C}{{\mathbb{C}}}

\newcommand{\IE}{{\mathcal{IE}}}

\newcommand{\E}{{\mathbf{e}}}

\newcommand{\A}{{\mathbf{a}}}

\newcommand{\B}{{\mathbf{b}}}

\newcommand{\GL}{\operatorname{GL}}

\newcommand{\Mod}{\operatorname{Mod}}

\newcommand{\SL}{\operatorname{SL}}

\newcommand{\diag}{{\rm diag}}

\newcommand {\ignore}[1]  {}

\newcommand{\spa}{{\rm span}}

\newcommand{\EE}{{\mathcal{E}}}

\newcommand{\bq}{{\mathbf{q}}}

\newcommand{\LL}{{\mathcal L}}

\newcommand{\FF}{{\mathcal{F}}}

\newcommand{\x}{{\bf x}}

\newcommand{\til}{\widetilde}

\newcommand{\supp}{{\rm supp}}

\newcommand{\rel}{{\mathrm{Rel}}}

\newcommand{\sm}{\smallsetminus}

\newcommand{\vre}{\varepsilon}
\newcommand{\genus}{\mathbf{g}}

\newcommand{\q}{\mathbf{q}}

\newtheorem{thm}{Theorem}[section]

\newtheorem{lem}[thm]{Lemma}

\newtheorem{prop}[thm]{Proposition}

\newtheorem{cor}[thm]{Corollary}

\newtheorem{remark}[thm]{Remark}

\newtheorem{dfn}[thm]{Definition}

\usepackage[utf8]{inputenc}
\usepackage{amsfonts}
\usepackage{mathtools}
\usepackage{graphicx}
\usepackage[colorinlistoftodos]{todonotes}
\usepackage{hyperref}

\newcommand{\defeq}{\vcentcolon=}

\newtheorem{theorem}{Theorem}[section]

\newtheorem{proposition}[theorem]{Proposition}
\newtheorem{corollary}[theorem]{Corollary}
\usepackage{stmaryrd}
\usepackage{appendix}

\newcommand{\lemref}[1]{\hyperref[#1]{Lemma \ref*{#1}}}
\newcommand{\thmref}[1]{\hyperref[#1]{Theorem \ref*{#1}}}
\newcommand{\propref}[1]{\hyperref[#1]{Proposition \ref*{#1}}}
\newcommand{\corref}[1]{\hyperref[#1]{Corollary \ref*{#1}}}

\begin{document}
\date{\today}

\begin{abstract}
We analyze the rel leaves of the Arnoux-Yoccoz translation surfaces. We
show that for any genus $\genus \geq 3$, the leaf is dense in the
connected component of the stratum
$\HH(\genus -1 ,\genus -1)$ to which it belongs, and the one-sided
imaginary-rel trajectory of the 
surface is divergent. For one surface on this trajectory, namely the
Arnoux-Yoccoz surface itself, the horizontal
foliation is
invariant under a pseudo-Anosov map (and in particular is uniquely
ergodic), but for all other surfaces, the horizontal foliation is 
completely periodic. 
The appendix proves a field theoretic result needed for denseness of the leaf:
for any $n \geq 3$,
the field extension of $\Q$ obtained by adjoining a root of
$X^n-X^{n-1}-\ldots-X-1$ has no totally real subfields other than
$\Q$. 
\end{abstract}
\maketitle

\section{Introduction}
A translation surface is a compact oriented surface equipped with a
geometric structure 
which is Euclidean everywhere except at finitely
many cone singularities and which has trivial rotational holonomy around loops. 
A stratum is a moduli space of translation
surfaces of the same genus whose singularities share the same
combinatorial characteristics. In recent years, intensive study has
been 
devoted to the study of dynamics of group actions and foliations on
strata of translation surfaces. See \S \ref{sec: basics} for precise
definitions, and see \cite{MT, zorich survey} for surveys.

Let $x$ be a translation surface with $k>1$ singularities. There
is a local deformation of $x$ obtained by moving its singularities
with respect to each other while keeping the translational holonomies of closed
curves on $x$ fixed. This local deformation gives rise to a foliation
of the stratum $\HH$ containing $x$, with leaves of real dimension $2(k-1)$. In the literature this foliation
has appeared under various names (see \cite[\S 9.6]{zorich
  survey}, \cite{schmoll}, \cite{McMullen-twists} and references therein), and we refer to it as the {\em rel
  foliation}, since nearby surfaces in the same leaf differ only in their relative
periods. A sub-foliation of this foliation, which we will refer to as the
{\em real-rel foliation}, is obtained by only
varying the horizontal holonomies of vectors, keeping all vertical
holonomies fixed. The {\em imaginary-rel} foliation is defined by
switching the roles of horizontal and vertical.  Although none of
these foliations are given by a group 
action, the obstructions to flowing along the real-rel and imaginary-rel foliations are
completely understood (see \cite{MW}). Specializing to
$k=2$, it makes sense to discuss the real-rel trajectory
$\{\rel_r^{(h)}x: 
r \in \R \}$ 
(respectively, the imaginary-rel trajectory $\{\rel_s^{(v)}x:
s \in \R \}$) of any surface $x$ without horizontal (resp. vertical)
saddle connections
joining distinct labeled singularities. 

\ignore{
In genus 2, results of McMullen \cite{McMullen-isoperiodic,
  McMullen-cascades} give a detailed understanding of
the closure of rel leaves in the eigenform locus. These results should
be viewed as a
companion to McMullen's celebrated result classifying closed
sets invariant under the action of $G=\SL_2(\R)$ in genus 2 \cite{McMullen-Annals}. 
A recent breakthrough result of Eskin, Mirzakhani and Mohammadi
\cite{EMM2} has shed light on the $G$-invariant closed
subsets for arbitrary strata, and the purpose of this paper is to
contribute to the 
study of  the topology of closures of rel leaves. Here we focus on the
rel leaf of the Arnoux-Yoccoz surface $x_0$ of genus $\genus \geq 3$ 
(see Figure \ref{fig:AY} for a picture, and see \S \ref{sec: AY surface} for the definition), which lies in  the 
stratum $\HH(\genus-1, \genus-1)$.
 Note that the surface depends on $\genus$ but we suppress $\genus$ from
the notation. We also discuss a related rel leaf
in the genus 2 stratum $\HH(1,1)$. These  
surfaces, introduced by
Arnoux and Yoccoz in \cite{AY}, are a source of interesting
examples for the theory of translation surfaces. See in particular
\cite{HL, Bowman}, and the work of Hubert, Lanneau and M\"oller \cite{HLM} in
which the orbit-closure $\overline{Gx_0}$ was determined in case
$\genus =3$.  
}

A primary purpose of this paper is to
contribute to the 
study of  the topology of closures of rel leaves. We focus on the
rel leaf of the Arnoux-Yoccoz surface $x_0$ of genus $\genus \geq 3$ 
(see Figure \ref{fig:AY} for a picture when $\genus=3$, and see \S \ref{sec: AY surface} for the definition), which lies in  the 
stratum $\HH(\genus-1, \genus-1)$.
 Note that the surface depends on $\genus$ but we suppress $\genus$ from
the notation. We also discuss a related rel leaf
in the genus 2 stratum $\HH(1,1)$. These  
surfaces, introduced by
Arnoux and Yoccoz in \cite{AY}, are a source of interesting
examples for the theory of translation surfaces as we discuss in \S \ref{sss: AY}.

\begin{figure}
  \hfill
  \begin{minipage}[c]{0.65\textwidth}
   \vspace{0pt}\raggedright
   \includegraphics[scale=0.60]{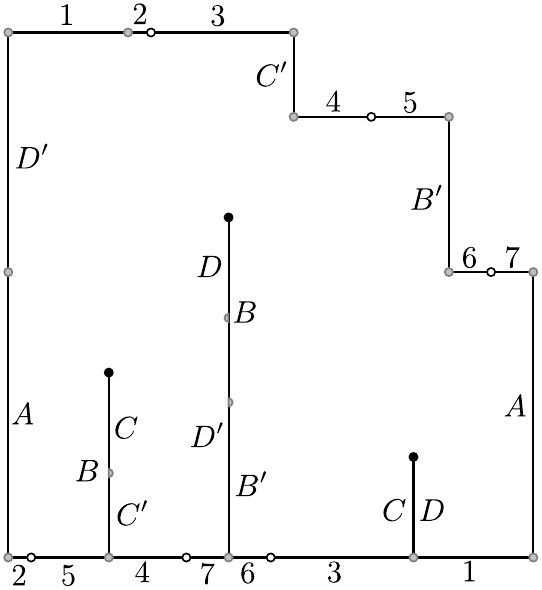}
  \end{minipage}
   \hfill
  \begin{minipage}[c]{0.3\textwidth}
\vspace{0pt}\raggedright
\begin{tabular}{ll}
Label & Edge length \\
\hline
\vspace{-1em} \\
$1$ & $\frac{1-\alpha}{2}$ \\
$2$ & $\alpha-\frac{1}{2}$ \\
$3$ & $\frac{\alpha}{2}$ \\
$4$, $5$ & $\frac{\alpha^2}{2}$ \\
$6$, $7$ & $\frac{\alpha^3}{2}$ \\
$A$ & $\alpha$ \\
$B$ &  $\frac{\alpha+\alpha^3}{2}$ \\
$B'$ & $\alpha^2$ \\
$C$, $D$ & $\frac{\alpha^2+\alpha^4}{2}$ \\
$C'$ & $\alpha^3$ \\
$D'$ & $\alpha^2+\alpha^3$
\end{tabular}  
  \end{minipage}
  \caption{The Arnoux-Yoccoz surface of genus $\genus=3$ with distinguished
    singularities and the presentation from \cite[pp. 496-498]{Arnoux88}.
    Edges with the same label are identified 
  by translation, and their lengths are provided by the chart. The points $\bullet$ and $\circ$
  are the two singularities of the surface. Grey dots denote regular points.}
  \name{fig:AY}
  \end{figure}

Our main results are stated below. Refer to \S \ref{sec: basics} for
detailed definitions, to \S \ref{sect:related} for connections to other mathematical ideas,
and to \S \ref{sect: outline} for an outline of the proofs. 

\begin{thm}\name{thm: real rel}
Suppose $\genus \geq 3$ and $x_0$ is the Arnoux-Yoccoz surface in
genus $\genus$. Then the imaginary-rel trajectory
of $x_0$ is divergent in $\HH$. Moreover, for any 
$s
>
0$, there is a horizontal cylinder decomposition of
$\rel_s^{(v)} x_0$, and the circumferences of the horizontal cylinders
tend 
uniformly to zero as $s \to 
+\infty$. 
\end{thm}

The pseudo-Anosov map of \cite{AY} preserves both the horizontal and
vertical directions on $x_0$. Nevertheless, the vertical and
horizontal directions play a different role in Theorem \ref{thm: real
  rel}.  
In case $\genus=3$ the surface $x_0$ has additional symmetries (see
\cite{Bowman10}), and exploiting them we obtain the following stronger
statement: 
\begin{cor}\name{cor: real rel}
Suppose $\genus =3$. Then for any $s \neq 0$, the surface
$\rel^{(v)}_s x_0$ has a horizontal cylinder decomposition and the
surface $\rel^{(h)}_s x_0$ has a vertical cylinder decomposition. As
$s \to +\infty$ and as $s \to -\infty$, the circumferences of the
cylinders in these decompositions tend uniformly to zero.  
\end{cor}

\begin{remark}
In genus $\genus=4$, experimental evidence (attained with the {\em sage-flatsurf} software package \cite{flatsurf}) suggests that $\rel^{(v)}_s x_0$ is divergent in $\HH$ in negative time as well, though not all $\rel^{(v)}_s x_0$ with $s<0$ admit horizontal cylinder decompositions. It would be interesting to understand
the behavior of $\rel^{(v)}_s x_0$ as $s \to -\infty$ and the behavior of $\rel^{(h)}_s x_0$ as $s \to \pm \infty$. The present article sheds little light on this.
\end{remark}

The behavior of the full rel leaf is very different. We will denote by 
$\HH_0$ the subset of the connected component of $\HH(\genus-1, \genus-1)$
containing $x_0$, consisting of surfaces whose area is equal that of $x_0$. (In the notation of \cite{KZ}, when
$\genus \geq 3$ is odd, the connected component of $x_0$ is $\HH(\genus-1, \genus-1)^{\mathrm{odd}}$
and when $\genus \geq 4$ is even, it is $\HH(\genus-1, \genus-1)^{\mathrm{non-hyp}}$. 
This follows from Proposition \ref{prop:stratum}.)
\begin{thm} \name{thm: rel}
For $\genus \geq 3$, the rel leaf of $x_0$ is dense in 
$\HH_0$. 
\end{thm}

\ignore{
The assumption of Theorem \ref{thm: rel} holds for all prime $\genus$,
since the trace field of $x_0$ is not itself totally real. We have
also verified it by direct computation in the cases $\genus = 3, \ldots,
1000$, and Mark Shusterman has proved that it holds for all even
$\genus$. We are not aware of any $\genus$ for which it is not
satisfied. 
\combarak{Hopefully Shusterman and Bary-Soroker will prove that it
  holds for all $\genus$ in which case we need to cite them and remove
  the condition from the statement.}
}
Theorem \ref{thm: rel} gives the first explicit examples of  
dense rel leaves in these strata.

In case $\genus=2$ the analogue of the Arnoux-Yoccoz surface was
defined by Bowman in the appendix of \cite{Bowman}. It is a
disconnected surface (more accurately, two tori attached at a node)
in $\HH(0) \times \HH(0)$, which can be thought of as part of a 
boundary of the genus two stratum $\HH(1,1)$  (see
\cite{Bainbridge}). In fact, in the terminology of McMullen \cite{McMullen-Annals}, it is an
eigenform surface for discriminant $D=5$.  Although it does not make sense to discuss the rel
leaf of $x_0$, one can view it as a terminal point of an imaginary rel
leaf in the $D=5$ eigenform locus $\EE_5(1,1) \subset \HH(1,1)$. Using this point of
view both of our main results extend to $\genus=2$ as follows: 
\begin{thm}\name{thm: g=2}
When $\genus =2$ there is a surface $x_1 \in \EE_5(1,1)$, such that
surfaces $x_s= \rel^{(v)}_{s-1} x_1$ have the
following properties: 
\begin{itemize}
\item[(i)]
The surfaces $x_s$ are defined for all $s >0$.
\item[(ii)]
The surfaces $x_s$ are horizontally periodic, and as $s \to +\infty$,
the circumferences of the horizontal cylinders tend uniformly to $0$.  
\item[(iii)]
In a suitable compactification of $\HH(1,1)$, the limit $x_0 = \lim_{s
  \to 0+}x_s$ exists and is the ``genus 2 Arnoux-Yoccoz surface''
of \cite{Bowman}.
\item[(iv)]
The rel-leaf of $x_1$ is dense in $\EE_5(1,1)$. 
\end{itemize}
\end{thm}

\subsection{Related ideas and works}
\name{sect:related}

\subsubsection{Linear deformations of IETs}
Theorem \ref{thm: real rel} has implications for the study of unique
ergodicity of interval exchange transformations (IETs), which we now describe. 
Let $\R^d_+$ denote the vectors in $\R^d$ with positive
entries. 
For a fixed permutation $\sigma$ on $d$ symbols, and 
$\mathbf{a} = (a_1, \ldots, a_d) \in \R^d_+$, let $\mathcal{IE} =
\mathcal{IE}_\sigma(\mathbf{a}): I \to I$ be the interval exchange transformation
obtained by partitioning the interval $I=\left[0, \sum a_i \right)$ into
subintervals of lengths $a_1, \ldots, a_d$ and permuting them
according to $\sigma$. Assuming irreducibility of $\sigma$, it is known that for almost all choices of
$\mathbf{a}$, $\IE_\sigma(\mathbf{a})$ is uniquely ergodic, but
nevertheless that
the set of non-uniquely ergodic interval exchanges is not very
small (see \cite{MT} for definitions, and a survey of this intensively studied
topic). It is natural to expect that all line segments in $\R^d_+$,
other than some obvious 
counterexamples, should inherit the prevalence of uniquely ergodic
interval exchanges. A conjecture in this regard was made by the
second-named author in \cite[Conj. 2.2]{cambridge}, and partial
positive results supporting the conjecture were obtained in 
\cite[Thm. 6.1]{MW}. 
Namely, a special case of \cite[Thm. 6.1]{MW} asserts that 
for any uniquely ergodic $\IE_0 =
  \IE_\sigma(\mathbf{a}_0)$, 
there is an explicitly given hyperplane $\mathcal{L} \subset \R^d_+$
containing $\mathbf{a}_0$, 
such that for any line segment $\ell  = \{\mathbf{a}_s : s \in I\}
\subset \R^d_+$ with $\ell \not \subset \mathcal{L}$, there is an
interval $I_0 \subset I$ containing $0$, such that for almost every $s
\in I_0$, $\IE_\sigma(\mathbf{a}_s)$ is uniquely ergodic. 

\ignore{
\combarak{Instead of stating the result with Yair as you suggested,
  which is technically involved and maybe slows things down, I put in
  an informal discussion in the above paragraph, what do you think?
  Note there is a brief discussion of the paper with Yair in the last
  paragraph of \S \ref{subsection: iets}. There is a version of the
results with Yair (relatively painless special case, but still
slightly intimidating) commented out in
the tex file, if you want to look at it. } \compat{I don't think this
theorem is actually that technical as stated below, and it makes this
work more surprising, so I vote to include it. Maybe less of a
discussion is needed in \S \ref{subsection: iets} if we do this.}

\begin{thm} For any uniquely ergodic $\IE_0 =
  \IE_\sigma(\mathbf{a}_0)$, 
there is an explicit hyperplane $\mathcal{L} \subset \R^d_+$
containing $\mathbf{a}_0$, 
such that for any line segment $\ell  = \{\mathbf{a}_s : s \in I\}
\subset \R^d_+$ with $\ell \not \subset \mathcal{L}$, there is an
interval $I_0 \subset I$ containing $0$, such that for almost every $s
\in I_0$, $\IE_\sigma(s)$ is uniquely ergodic. 
\end{thm}

}
This expectation is false. In fact counterexamples to 
\cite[Conj. 2.2]{cambridge} were exhibited by Dynnikov\footnote{Dynnikov's
  counterexamples appeared (in a 
  different context and 
  using different terminology) three
  years before \cite[Conj. 2.2]{cambridge} was formulated. However the authors only became
  aware of this after a preliminary version of this 
  paper was circulated.} \cite{Dynnikov}, and further related results were
obtained in the subsequent work 
\cite{DD, Dynnikov-Skripchenko, Hubert-Skripchenko}. 
In particular it was shown that there are  linear subspaces $\mathcal{L} \subset \R^d$ such that
almost every point in $\mathcal{L}$ is non-minimal, and yet
$\mathcal{L}$ contains an uncountable closed set of uniquely ergodic
points. Our results can be used to give a stronger refutation
of  \cite[Conj. 2.2]{cambridge}. Recall that a standard
construction of an interval exchange transformation, is to fix a
translation surface $q$ with a segment $\gamma$ transverse to horizontal
lines, and define $\IE(q, \gamma)$ to be the first return map to
$\gamma$ along horizontal leaves on $q$ (where we parameterize $\gamma$
using the transverse measure $dy$). If $L = \{x(s) : s \in I\}$ is a
sufficiently small 
straight line segment in a stratum of
translations surfaces, $\gamma$ can be chosen uniformly for all $s \in
I$ and the interval exchanges $\IE(x(s), \gamma)$ can be written as
$\IE_{\sigma} (\mathbf{a}(s))$ for some fixed permutation $\sigma$ and some
line $\ell = \{\mathbf{a}(s) : s \in I\} \subset \R^d_+$ (see
\cite{MW} for more details). The number of intervals one obtains for
surfaces in $\HH(\genus-1, \genus-1)$ is $2\genus +1$. Thus, taking $x(s) = \rel^{(v)}_s x_0$,
Corollary \ref{cor: real rel} implies:

\begin{cor}\name{cor: Boshernitzan}
For $d=7$, there is a uniquely ergodic self-similar interval exchange
transformation $\IE_0$ on $d$ intervals, 
and a line segment $\ell \subset \R^d_+$ such that $\IE_0 =
\IE_\sigma(\mathbf{a}_0)$ with $\mathbf{a}_0$ in the interior of $\ell$,
and such that for all $\mathbf{a} \in \ell \sm \{\mathbf{a}_0\}$,
$\IE_\sigma(\mathbf{a})$ is periodic.  
\end{cor}

\subsubsection{Rel leaves}
It is well-known that some imaginary-rel trajectories exit their stratum in
finite time due to the collapse of vertical saddle
connections. In this case the imaginary-rel trajectory is not defined for
all time. Theorem \ref{thm: real rel} shows that an imaginary-rel trajectory
may be defined for all time and still be divergent in its
stratum. 
This should be contrasted with the horocycle flow for which
there are no divergent trajectories (see \cite{Veech - fractals}).

Although we only discuss the rel
leaves of the Arnoux-Yoccoz surfaces in this paper, it is quite likely
that the rel foliation is ergodic in any stratum with more than one singularity, and thus almost all rel
leaves are dense in all such strata. Recently 
 Calsamiglia, Deroin and Francaviglia \cite{Deroin} and Hamenst\"adt
 \cite{Hamenstadt} have proved ergodicity in principal strata,
 i.e. strata where all singularities are simple. The
 methods of proof used in these papers are very different from the one
 used in this paper.

Individual rel leaves have a natural translation structure (see
\cite{McMullen-isoperiodic, MW}). In
\cite[Thm. 1.1]{McMullen-isoperiodic} it was shown that in a suitable
bordification of the stratum, this geometric structure is complete
(note that the conventions used in \cite{McMullen-isoperiodic} differ from those used in
this paper by taking the quotient by a finite permutation group,
e.g. if $k=2$ then with our conventions rel leaves are translation
surfaces and with McMullen's, they are half-translation
surfaces). Furthermore, in genus 2, results of McMullen \cite{McMullen-foliations, McMullen-isoperiodic} give a detailed understanding of
the closure of rel leaves in the eigenform locus. These results should
be viewed as a
companion to McMullen's celebrated result classifying closed
sets invariant under the action of $G=\SL_2(\R)$ in genus 2 \cite{McMullen-Annals}. 

\subsubsection{The Arnoux-Yoccoz IET and the SAF invariant}
\label{sss: AY}
This work began with computer experiments in the case $\genus=3$
which demonstrated the complete periodicity of the vertical foliation
on real-rel deformations of $x_0$. These computer experiments were
suggested by Michael Boshernitzan, and 
the observed results are now described by Corollary \ref{cor: real rel}.

These experiments were in turn motivated by several related
observations. First, the Arnoux-Yoccoz IETs have vanishing SAF
invariant. This is evident from their construction in \cite{AY}, and
follows from the various functorial properties of the SAF invariant;
see \cite[\S IV]{Arnoux88} or \cite[\S 2.1]{HL-handbook}. These
properties of the SAF invariant also imply that the invariant is
well-defined for the vertical foliation on any translation surface
suspension 
of the IET, and 
is rel-invariant. Thus every rel deformation of $x_0$ has vanishing
SAF invariant. A vanishing SAF invariant is often associated with a
completely periodic IET, so it was natural to wonder if rel
deformations of $x_0$ are often vertically completely periodic; and a
precursor  
to this work \cite{case g=3} investigated
this question.

\subsubsection{Algebraic dynamics}
For any interval exchange map $T_0:[0,1] \to [0,1]$, observe that the
displacements $T_0(x)-x$ lie in a finitely generated subgroup of
$\R/\Z$. One can then consider the Cayley graph of this group where
the generators are taken to be the possible values of $T_0(x)-x$. Any
$T_0$-orbit then gives rise to a parameterized path $n \mapsto
T^n_0(x)-x$ in this Cayley graph. The SAF-invariant may be interpreted
as the average displacement in this subgroup, so vanishing
SAF-invariant is suggestive of prevalent closed paths or paths which
diverge sublinearly (though other behavior is possible). This makes
these paths particularly interesting when the SAF-invariant vanishes. 

The discontinuity points of an IET $T_0$ have a natural partition into
subsets corresponding to singular points of any suspension of the IET,
and a perturbation $T_t$ of the IET in which all discontinuity points
corresponding to one singular point are moved by the same $t \in \R$
(where $t$ is small enough so that no interval collapses) corresponds to a rel deformation. 
Such perturbations change the IET
without affecting the group generated by the displacements, so the Cayley graph above remains invariant.
Thus it is
interesting to consider how the paths change as we
vary $t$. Note that the paths that arise vary continuously in $t$
where we use the topology of pointwise convergence on paths viewed as
functions from $\Z$ to the Cayley graph.  

In \cite{LPV}, a variant these paths were studied for the
Arnoux-Yoccoz IET $\IE_0$. In this case the group generated by
possible displacements (modulo $1$) is isomorphic to $\Z^2$ and the
Cayley graph is isomorphic to the hexagonal tiling of the plane. The
article \cite{LPV} found combinatorial space filling curves for
$\IE_0$. The paths for $\IE_0$ were also discussed in \cite[\S
5]{McMullen-cascades} as well as paths for other IETs. A closed path was found in \cite[end of \S
5]{McMullen-cascades} associated to an IET arising from one small rel-deformation of the Arnoux-Yoccoz surface of genus $3$. We find that
under rel deformations, the closed combinatorial orbits associated to
the closed orbits of $\IE_t$ fill larger and larger regions in the
group as $t \to 0$; see Figure \ref{fig: sequence} and the further
discussion in \S \ref{sect: genus three}. This phenomenon may be
viewed as a consequence of continuity and the space filling results in
\cite{LPV}. 

This sort of algebraic dynamics does not seem to have been used
heavily in the study of IETs except in the case when the subgroup of
$\R/\Z$ mentioned above is isomorphic a finite extension of $\Z$. This
case was considered by  
Arnoux \cite[\S IV.8]{Arnoux88}, Boshernitzan  \cite{Boshernitzan} and
McMullen \cite{McMullen-Acta} and plays a role in the result that
ergodic IETs in this case have non-vanishing SAF-invariant.

\begin{figure}
   \includegraphics[width=5in]{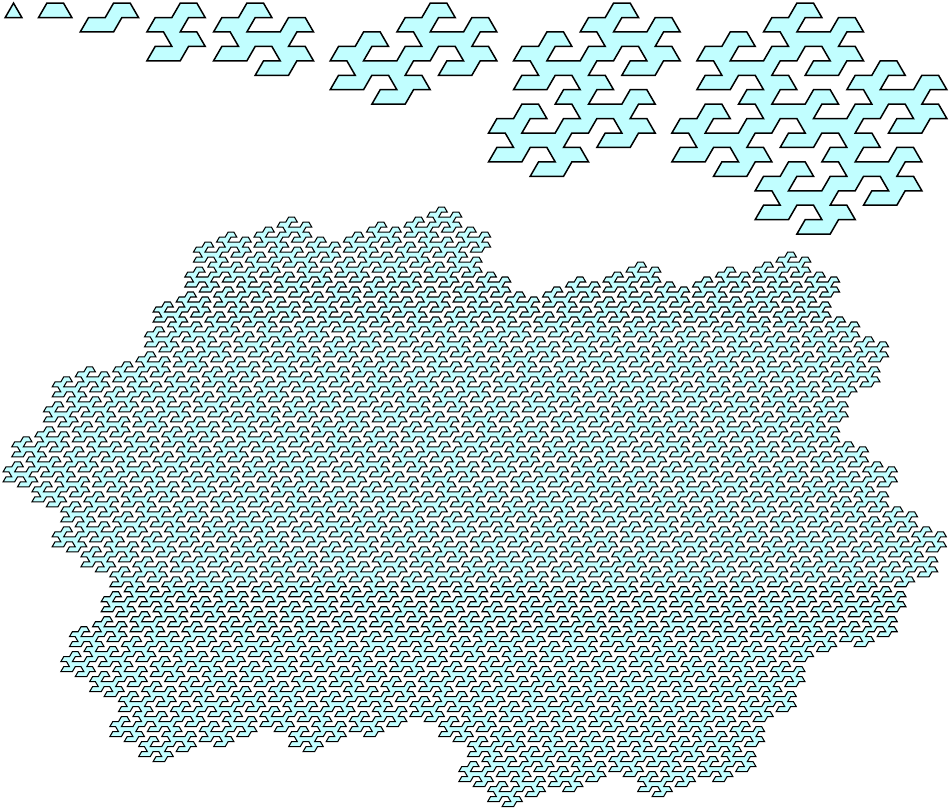}
  \caption{Depictions of the algebraic dynamics of periodic orbits of
    the rel deformed Arnoux-Yoccoz IET in genus three. The 1st through
    8th and 15th shortest combinatorial orbits are shown. See \S
    \ref{sect: genus three} for more detail.
    } 
  \name{fig: sequence}
  \end{figure}

\subsubsection{The work of Dynnikov}
We remark that the Arnoux-Yoccoz surface in genus $\genus=3$ is a
suspension of the Arnoux-Yoccoz IET, which  
belongs to a family of IETs parameterized by three real numbers whose
interval lengths vary linearly in the parameters and can be understood
through methods of Dynnikov \cite{Dynnikov talk} which differ greatly
from the approach of this article. 
Lebesgue almost every IET in the family is non-minimal, and there is a
renormalization scheme which can be used to understand these IETs. In
particular, a phenomenon similar to Corollary \ref{cor: real rel} 
is exhibited in \cite[\S 7]{Dynnikov}, where complete periodicity
holds in a parameter neighborhood of a minimal IET.

\subsubsection{$\SL_2(\R)$-orbit closures}
Let $G=\SL_2(\R)$ and let $U,V \subset G$ be the unipotent subgroups
which shear in the horizontal and
vertical directions respectively. 
A recent breakthrough of Eskin, Mirzakhani and Mohammadi
\cite{EMM2} describes the $G$-invariant closed
subsets for arbitrary strata. The
following result, which we deduce from their work, is important for
our proof of Theorem \ref{thm: rel}, 
but should also be of independent interest: 

\begin{prop}\name{prop: V suffices}
Suppose that $x$ is a translation surface and $\{g_t x : t \in
\R\}$ is a periodic trajectory for the geodesic flow. Then the orbit closures
$\overline{Ux}$, $\overline{Vx}$ and $\overline{Gx}$ coincide.
\end{prop}

The proof makes use of results of \cite{EMM2} on the orbit closures
under the upper-triangular subgroup of $G$.

In case $\genus =3$, Hubert,
Lanneau and M\"oller
\cite{HLM} determined $\overline{Gx_0}$, proving that it coincides
with the set of surfaces in $\HH_0$ admitting a hyperelliptic
involution. In case $\genus \geq 4$, 
we use an idea of Alex Wright to show that
$\overline{Gx_0}$ is an affine manifold of full rank (see \S
\ref{subsec: full rank}). The argument requires the following algebraic fact, which was proved
in response to our question (see Appendix
\ref{appendix: BSZ},  Corollary \ref{NoTrCor}): 

\begin{thm}[Bary-Soroker, Shusterman and Zannier]
\name{thm: no real subfields}
Let $\genus \geq 3$ and let $\alpha$ be the unique real number in
$[0,1]$ satisfying $\alpha + 
\cdots + \alpha^\genus =1$. Then $\Q(\alpha)$ contains no totally real
subfields other than $\Q$.  
\end{thm}

The following result is of independent interest and follows from
Theorem \ref{thm: no real subfields} and a recent preprint of
Mirzakhani and Wright \cite{Mirzakhani-Wright}, via an argument of Alex Wright. It
could be used to simplify some of the steps in the proof of Theorem
\ref{thm: rel}, see Remark \ref{rem: Mirzakhani Wright}.


\begin{thm}
\name{thm: orbit closure}
If $\genus \geq 4$, then $\overline{G x_0}=\HH_0$.
\end{thm}

\subsubsection{Twist coordinates for horizontally periodic surfaces}
Let $x$ be a horizontally completely periodic translation surface in the stratum $\HH$.
Observe that surfaces in the horocycle orbit $Ux$ all have a horizontal cylinder decomposition with the same
circumferences, same cylinder widths, and same lengths of horizontal saddle connections. 
By deforming $x$ (in any manner) while leaving these quantities constant, a subset of $\HH$ is obtained
which is naturally the image of a torus, and {\em twist coordinates} can be placed on this torus. 
In \cite{calanque}, it was shown that the horocycle dynamics on the orbit closure $\overline{U x}$ can be described as linear flow on this torus in twist coordinates, and this observation was used to describe horocycle orbit closures of horizontally completely periodic surfaces.

In this article we observe that twist coordinates can also be used to describe real rel orbits of $x$, and the dynamics are again linear flow on a torus. In these coordinates, the horocycle flow orbit-closure is determined
by the moduli of the cylinders while the real-rel orbit-closure is
determined by their circumferences. We use this here to show the horocycle orbit of the Arnoux-Yoccoz surface $x_0$ is contained in the closure of the rel leaf of $x_0$. 

\subsection{Outline of paper and proofs}
\name{sect: outline}
We now discuss the proofs of our results. 

Sections \ref{sect: rel ray} and \ref{sec: AY surface} contain the proof of Theorem \ref{thm: real rel} (as well as Theorem \ref{thm: g=2} and  Corollary \ref{cor: real rel}). We obtain a description of surfaces $x_t=\rel^{(v)}_t x_0$ indirectly. Section \ref{sect: rel ray}
concentrates on studying a mystery imaginary rel trajectory, and \S  \ref{sec: AY surface} shows it is the imaginary rel trajectory of the Arnoux-Yoccoz surface.
\begin{itemize}[leftmargin=3ex]
\item In \S \ref{sect:renormalization} we review the definition of
the Arnoux-Yoccoz interval exchange $\IE$ on $2 \genus +1$ intervals. Theorem \ref{thm:renormalization} describes a renormalizing
symmetry of $\IE$ known to Arnoux and Yoccoz. We prove this result for completeness.
\item In \S \ref{sect: suspensions} we discuss suspensions of
$\IE$, all of which lie in the stratum $\HH(\genus-1, \genus-1)$.
We find the connected component of the stratum containing suspensions of $\IE$ in Proposition \ref{prop:stratum}.
We associate to the renormalization of $\IE$ a canonical mapping class $[\phi]$ in Corollary \ref{cor:second suspension}
and Definition \ref{def: renormalization mapping class}.
\item In \S \ref{sect: rel rays}, we choose a special suspension
$q_0 \in \HH(\genus-1, \genus-1)$ of $\IE$ which is horizontally completely periodic.
The imaginary rel trajectory is denoted $q_s=\rel^{(v)}_s q$ whenever defined. Theorem \ref{thm:completely periodic ray}
shows that these surfaces exhibit a useful symmetry (which explains our choice of $q_0$):
After reparameterizing the trajectory as $x_t=q_{t-\beta}$ for $\beta=\frac{\alpha^2}{1-\alpha}$ we observe
the coincidence $\tilde g x_{\alpha^{-1} t}= x_t$ for all $t$, where $\tilde g \in \SL(2,\R)$ is a diagonal matrix expanding the vertical direction by a factor of $\alpha$. Furthermore, the mapping class determined by deforming $x_t$ into $x_{\alpha^{-1} t}$ by applying vertical rel and then applying $\tilde g$ to get back to $x_t$ is $[\phi]$ mentioned above. We show all this by explicit computations,
and along the way demonstrate that the surfaces $x_t$ for $t>0$ are all horizontally completely periodic. From the symmetry above,
it suffices to show $x_t$ is horizontally completely periodic for $t$
taken from the fundamental domain $[\beta, \beta/\alpha]$ for the $\tilde g$
action on the vertical rel ray. 
\item In \S \ref{sect:cylinders}, we prove a lemma about the horizontal cylinders of $x_t$ for $t>0$.
This will be used later to understand how twist coordinates change when we apply the real rel flow to $x_t$
for $t>0$, which is an ingredient in the proof that the full rel leaf is dense in the sense of Theorem \ref{thm: rel}.
\item We consider the case of $\genus \geq 3$ in \S \ref{subsec: AY surfaces general genus}. Lemma \ref{lem:mapping class 1} shows that in this case the mapping class $[\phi]$ is pseudo-Anosov. We argue that suspensions of $\IE$ have infinite vertical rel trajectories,
and thus there is a surface $x_0$ so that $\tilde g x_0=x_0$. Theorem \ref{thm: identifying AY}
uses Thurston's theory of pseudo-Anosov homeomorphisms of surfaces
to argue that $x_0$ must be the Arnoux-Yoccoz surface. This completes the proof of Theorem \ref{thm: real rel}.
\item In \S \ref{sect: genus two} we discuss the case of $\genus=2$ and prove Theorem \ref{thm: g=2}.
\item In \S \ref{sect: genus three} we describe additional symmetries that show up when $\genus=3$ and prove Corollary \ref{cor: real rel}.
We note that most of these additional symmetries do not persist when $\genus \geq 4$. We discuss
the arithmetic dynamics in the case of $\genus=3$ as mentioned above.
\end{itemize}

For the proof of Theorem \ref{thm: rel}, we successively show that the
closure of the rel leaf of the Arnoux-Yoccoz surface $x_0$ is larger
and larger. From this point of view, there are two steps, which we
undertake in reverse order. In \S \ref{sec: minimal sets} we show that
the  orbit $U x_0$ is in this closure. In \S \ref{subsec:
  breakthroughs}, we show that it follows that the closure of the
rel leaf is the stratum connected component $\HH_0$.  Here is a more
detailed description of these sections. 

Section \ref{subsec: breakthroughs} reviews the relevant results from
Eskin, Mirzakhani and Mohammadi, 
and then discusses consequences used in the proof of Theorem \ref{thm: rel}:
\begin{itemize}[leftmargin=3ex]
\item In \S \ref{sect:horocyle} we prove Proposition \ref{prop: V
    suffices}. 
\item After reviewing results of Filip and Wright, along with some
  ideas due to Wright, \S \ref{subsec: full 
    rank} culminates in the observation that when $\genus \geq 3$ the
  only closed subset of $\HH_0$ which is rel-invariant and contains $U
  x_0$ is $\HH_0$; 
see Corollary \ref{cor: corollary V}. This means that in order to show the
closure of the rel leaf contains $\HH_0$ it suffices to show that this
closure contains the horocycle orbit $U x_0$. 
\end{itemize}

In \S \ref{sec: minimal sets} we analyze the interaction of real-rel
and the horocycle flow on surfaces which have a decomposition into
parallel horizontal cylinders.
\begin{itemize}[leftmargin=3ex]
\item  It \S \ref{sect: twist coordinates} we consider deforming a horizontally completely periodic surface while preserving cylinder dimensions and horizontal saddle connection lengths. We describe how to place twist coordinates on this deformation space and observe that the subset of the stratum obtained in this way is an orbifold quotient of a torus.
\item In \S \ref{sect: rel twist} we consider the real rel-flow applied to a horizontally completely periodic translation surface with no horizontal saddle connections joining distinct
singularities. We find that in twist coordinates, the real rel-flow can be explicitly described as linear flow on a torus (or its quotient as above).
\item In \S \ref{subsec: deformations etc} we specialize the results
  of the preceding subsections 
to the case of the horizontally periodic surfaces $x_s=\rel^{(v)}_s x_0$ where $x_0$ is the Arnoux-Yoccoz surface. By taking closures of real real deformations of such surfaces $x_s$, we find a 
family of  tori 
$\{\mathcal{O}_s\}$ of dimension greater than 1, defined for all but a discrete set of $s>0$.
Our explicit parameterization
enables us to compute the behavior of the tangent planes $T_s$ to
$\mathcal{O}_s$, 
as $s \to 0+$. 
Recalling that twist coordinates can also be used to describe horizontal horocycle orbits (following \cite{calanque}), we show that the topological
limit of the $\mathcal{O}_s$ 
contains the entire horocycle orbit $Ux_0$, concluding the proof. 
\end{itemize}

The appendix by Bary-Soroker, Shusterman and Zannier
contains the proof of Theorem \ref{thm:
  no real subfields} (which is crucial to our arguments in \S \ref{subsec: full rank}.)

\subsection{Acknowledgments} 
This work was stimulated by insightful
comments of Michael Boshernitzan, who conjectured Corollary \ref{cor:
  Boshernitzan}. We thank Alex Wright for directing our 
attention to the case $\genus=2$ and for his proofs of Theorems
\ref{thm: orbit closure} and 
\ref{thm: full rank Wright idea}. Theorem \ref{thm: no real subfields}, which is a
crucial step in our proof of Theorem \ref{thm: rel}, was proved in
response to our queries by Lior Bary-Soroker, Mark Shusterman, and
Umberto Zannier. We thank them for agreeing to include their results
in Appendix \ref{appendix: BSZ} of this paper. 
We thank Ivan Dynnikov, Pascal Hubert and Sasha
Skripchenko for pointing out the connections to their prior work and
other insightful remarks. We are also grateful to
David Aulicino, Josh Bowman, Duc-Manh Nguyen and John Smillie for useful
discussions. We also are grateful to the anonymous referee for useful comments which helped to improve the paper.
This collaboration was supported by BSF grant 2010428.
The first author's work is supported by NSF grant DMS-1500965 as well as a PSC-CUNY
Award (funded by The Professional Staff Congress and The City University of New
York). The second author's work was supported by  ERC
starter grant DLGAPS 279893.

\section{Basics}\name{sec: basics} 
\subsection{Translation surfaces, strata, \texorpdfstring{$G$}{G}-action, cylinders}
\name{sec: translation surfaces}
In this section we define our objects of study and review their basic
properties. We refer to \cite{MT}, \cite{Wright survey} and \cite{zorich survey} for more information
on translation surfaces and related notions, and for references for
the statements given in this subsection.

Let $S$ be a compact oriented surface of genus $\genus \geq 2$, let $\Sigma= \{\xi_1, \ldots, 
\xi_k\} \subset S$
be a finite set of points
and let $\mathbf{r} = (r_1,
\ldots, r_k)$ be a list of non-negative integers such that $\sum r_i = 2\genus-2$. 
A {\em translation atlas} of type $\mathbf{r}$ on $(S, \Sigma)$
is an atlas of charts $(U_{\alpha}, \varphi_{\alpha})$,
where:
\begin{itemize}
\item For each $\alpha$, the set $U_\alpha \subset S\sm \Sigma$ is open, and the map
$$\varphi_\alpha:U_\alpha \to \R^2$$
is continuous and injective.
\item The union over $\alpha$ of $U_\alpha$ is $S \sm \Sigma$. 
\item Whenever the sets $U_\alpha$ and $U_\beta$ intersect, the
  transition functions are local translations, i.e., the maps
$$
\varphi_{\beta} \circ
  \varphi^{-1}_{\alpha}: \varphi_\alpha(U_\alpha \cap U_\beta) \to
\R^2
$$
are differentiable with derivative equal to the
identity. 
\item Around each $\xi_j \in
\Sigma$ the charts glue together to form a cone singularity with cone angle 
$2\pi(r_j+1)$. 
\end{itemize}
A {\em translation surface structure} on $(S,\Sigma)$ of type $\mathbf{r}$ is an equivalence class of such translation atlases, where $(U_{\alpha}, \varphi_{\alpha})$ and $(U'_{\beta},
\varphi'_{\beta})$ are equivalent if there is an orientation preserving  
homeomorphism $h: S\to S$, fixing all points of $\Sigma$, such that 
$(U_{\alpha}, \varphi_{\alpha})$  
is compatible with  $\left(h(U'_{\beta}), \varphi'_{\beta} \circ h^{-1}\right).$
A {\em marked translation surface structure} is an equivalence class of such atlases
subject to the finer equivalence relation where $(U_{\alpha}, \varphi_{\alpha})$ and $(U'_{\beta},
\varphi'_{\beta})$ are equivalent if $h$ can be taken to be isotopic to the identity via an isotopy fixing $\Sigma$.
Thus, a marked translation surface $\bq$ determines a translation surface
$q$ by {\em forgetting the marking}, and we write $q =\pi(\bq)$ to denote this operation.
Note that our convention is that all singularities are labeled.

Pulling back $dx$ and $dy$ from the coordinate charts we obtain two
well-defined closed 1-forms, which we can integrate along any path
$\gamma$ on $S$. If $\gamma$ is a cycle or has endpoints in $\Sigma$
(a relative cycle), then we define
$$\hol_{\mathrm x}(\gamma, \bq)=\int_\gamma dx \quad \text{and} \quad \hol_{\mathrm y}(\gamma, \bq)=\int_\gamma dy.$$
These integrals only depend on the homology class of $\gamma$ in
$H_1(S, \Sigma)$ and the pair of these integrals is the {\em holonomy}
of $\gamma$, 
\eq{eq: defn hol}{\hol(\gamma, \bq) = \left(\begin{matrix} \hol_{\mathrm x}(\gamma, \bq) \\
\hol_{\mathrm y}(\gamma, \bq) \end{matrix} \right) \in \R^2.
}
We let $\hol(\bq) = \hol(\cdot, \bq)$ be the corresponding element of $H^1(S, \Sigma;
\R^2)$, with coordinates ${\mathrm x}(\bq)$ and ${\mathrm y}(\bq)$ in $H^1(S, \Sigma; \R)$. 

 The set of all (marked) translation surfaces on $(S,
\Sigma)$ of type $\mathbf{r}$ is called the {\em stratum of
(marked) translation surface of type $\mathbf{r}$} and is denoted by
$\HH(\mathbf{r})$ (resp. $\HH_{\mathrm{m}}(\mathbf{r})$). 
The map $\hol: \HH_{\mathrm{m}} (\mathbf{r})\to H^1(S, \Sigma; \R^2)$ just defined gives
local charts for $\HH_{\mathrm{m}} (\mathbf{r})$, endowing it (resp. $\HH(\mathbf{r})$) with the
structure of an affine manifold (resp. orbifold).

Let $\Mod(S, \Sigma)$ denote the mapping class group, i.e. the
orientation preserving homeomorphisms of $S$ fixing $\Sigma$
pointwise, up to an isotopy fixing $\Sigma$. 
The map hol is $\Mod(S, \Sigma)$-equivariant. 
%
The 
$\Mod(S, \Sigma)$-action on $\HH_{\mathrm{m}}$ 
is properly discontinuous. Thus $\HH (\mathbf{r})= \HH_{\mathrm{m}} (\mathbf{r})/\Mod(S, \Sigma)$ is a
linear orbifold and $\pi: \HH_{\mathrm{m}} (\mathbf{r})\to \HH
(\mathbf{r})$ is an orbifold covering map. 
In particular, if 
$\{\mathbf{q}_s : s \in [0,1]\}$ is a continuous 
path in $\HH_{\mathrm{m}}$ with
$\pi(\mathbf{q}_0) = \pi(\mathbf{q}_1)$ then there is a corresponding
element of $\Mod(S, \Sigma)$ which maps $\mathbf{q}_0$ to
$\mathbf{q}_1$ and is independent of the choice of the lift. 
We have
\eq{eq: dimension of HH}{
\dim \HH (\mathbf{r})= \dim \HH_{\mathrm{m}} (\mathbf{r})= 
\dim H^1(S, \Sigma; \R^2)= 2(2g+k-1).
}

There is an action of $G = \SL_2(\R)$ on $\HH(\mathbf{r})$ and on
$\HH_{\mathrm{m}}(\mathbf{r})$ by post-composition on each 
chart in an atlas. The projection $\pi : \HH_{\mathrm{m}}(\mathbf{r})
\to \HH(\mathbf{r})$ is $G$-equivariant.  
The $G$-action is linear in the homology coordinates, namely, given a
marked translation surface structure $\bq$ and $\gamma \in 
H_1(S, \Sigma)$, and given $g \in G$, we have
\eq{eq: G action}{
\hol(\gamma, g\bq) = g \cdot \hol(\gamma, \bq),
}
where on the right hand side, $g$ acts on $\R^2$ by matrix
multiplication. 

We will write 
\begin{equation}
\name{eq: matrices}
u_s 
= \left(\begin{array}{cc} 1 & s \\ 0 & 1 
\end{array}
\right),
 \ \ \ \ \, \ 
g_t = \left(\begin{array}{cc} e^{t} & 0 \\ 0 & e^{-t} 
\end{array}
\right), \ \ \ \ \ 
v_{s}=
\left(\begin{array}{cc}
1 & 0 \\
s & 1 
\end{array}
\right).
\end{equation}
Also we will denote
$$
U = \{u_s : s\in \R\}, \ \ A = \{g_t: t \in \R \}, 
$$
$$
V = \{v_s: s
\in \R\}, \ \ P=AU =
\left(\begin{matrix} * & * \\ 0 & * \end{matrix} 
\right) \subset G. 
$$

The connected components of strata $\HH(\mathbf{r})$ have been
classified by Kontsevich and Zorich \cite{KZ}. We will be interested in the
particular connected component of
$\HH(\genus-1,\genus -1 )$ containing
$x_0$. For any $\mathbf{r}$, the area of surfaces in 
$\HH(\mathbf{r})$ is preserved by the action of $G$ and we  let $\HH$
be a fixed-area sub-locus of a connected component
of $\HH(\mathbf{r})$. 
A common  convention is to normalize area by setting
$\HH$ to be the locus of area-one surfaces, but it will be more
convenient for us to fix the area equal to some constant, e.g. the area of $x_0$.
There is a globally supported measure on $\HH$ which is defined using 
Lebesgue measure on $H^1(S, \Sigma; \R^2)$ and a `cone construction'. It was shown by Masur that the
$G$-action is ergodic with respect to this measure, and in particular,
almost every $G$-orbit is dense. 

\ignore{
\begin{remark}
Formally, we have defined $q$ to be a translation surface structure on $S$, i.e., an equivalence class of atlases.
We abuse this formalism by thinking of $q$ as a translation surface, i.e. a topological surface endowed with $dx$ and $dy$ and the notion of holonomy as above. 
\compat{Maybe it is too pedantic to be bothered by this?}
\end{remark}
}
Let $\vec{u} \in \R^2$ be a unit vector. The {\em straight-line flow} in direction $\vec{u}$
on a translation surface $q$ is the flow $F_t:q \to q$ defined so that holonomy of a trajectory of length $t$,
$$[0,t] \to q; \quad s \mapsto F_s(p),$$
is $t \vec{u}$. Because of the cone singularities, the straight-line
flow does not need to be defined for all time. The straight-line flow
in direction $(0,1)$ will be called the {\em vertical straight-line
  flow}. 

A {\em separatrix} on a translation surface is a straight-line trajectory leaving a singularity.
A {\em saddle connection} is a straight-line trajectory terminating at
singularities in forward and backward time. 

An {\em affine automorphism} of a translation surface $q$ is a
homeomorphism of $q$ which leaves invariant the set of singular points and which
is affine in charts. 
The derivative of an affine automorphism is a $2 \times
2$ real matrix of determinant $\pm 1$. If this matrix is hyperbolic (i.e. has distinct real
eigenvalues) then the affine automorphism is called a {\em
  pseudo-Anosov map}. If the matrix is parabolic (i.e. is nontrivial
and has both eigenvalues equal to 1) then the affine automorphism is
called {\em parabolic}. The group of derivatives of orientation
preserving affine automorphisms of $q$ is called the {\em Veech
  group} of $q$. 

The field $k$ generated
over $\Q$ by the traces of all derivatives of pseudo-Anosov maps of $q$ is called the
{\em trace field of $q$}. In fact, by a theorem of Kenyon and Smillie
\cite{KS}, $k$ is generated by the trace of any single 
pseudo-Anosov map of $q$. The {\em holonomy  field} of $q$ is the
smallest field $k$ such that there is $g \in \GL_2(\R)$ and a marked
surface $\mathbf{q}$ with $\pi(\q)=q$ such
that for all $\gamma \in H_1(S; \Z)$, $g \cdot \hol(\gamma, \mathbf{q}) \in
k^2$. In \cite{KS} it is shown
that if $q$ has a pseudo-Anosov map then the holonomy field coincides
with the trace field.

Let $I \subset \R$ be a closed interval with interior, let $c>0$ and
let $\R/c\Z$ be the circle of circumference $c$. 
A {\em cylinder} on a translation surface is a subset homeomorphic to
an annulus
which is the image of $I \times \R /c\Z$ for some $I$ and  $c$ as above, under a map
which is a local Euclidean isometry, and which is maximal in the sense
that the local isometry does not extend to $J \times \R/c\Z$ for an
interval $J$ properly containing $I$. The parameter $c$ is called the
{\em circumference} of the cylinder, and the image of $\{t\} \times
\R/c\Z$ for some $t \in \mathrm{int}( I)$ is called a {\em core
  curve}. In this case the two boundary components of the cylinder are
unions of saddle connections whose holonomies are all parallel to that
of the core curve. If a
translation surface $q$ can be
represented as a union of cylinders, which intersect along their
boundaries, then the directions of the holonomies of the
core curves of the cylinders are all the same, and we say that this
direction is {\em completely periodic} and that  $q$ {\em has
a cylinder decomposition} in that direction. 

If $C$ is a cylinder with an oriented core curve on a translation surface modeled on $S$,
we use $C^\ast$ to denote the associated cohomology class in $H^1(S,\Sigma; \Z)$ defined
by
\eq{eq:cohomology class of cylinder}{
C^\ast(\gamma)=\gamma \cap C \quad \text{for $\gamma \in H_1(S,\Sigma; \Z)$,}
}
where $\cap$ denotes the algebraic intersection pairing between $\gamma$
and the oriented core curve. Recall that the intersection pairing is
nondegenerate as a bilinear form on $H_1(S, \Sigma) \times H_1( S \sm
\Sigma).$ 

\subsection{Interval exchange maps, suspensions, and pseudo-Anosov maps}
\name{sect:IETs}
We define interval exchange tranformations (IETs) on the circle $\TTT =
\R/\Z$ rather
than the interval, because the original Arnoux-Yoccoz IET was defined
on the circle. An {\em interval exchange transformation} of the circle
is a  bijective piecewise rotation 
$$\IE:\TTT \to \TTT; \quad \IE(x)=x+t_j \quad \text{for $x \in I_j$},$$
where $\big\{I_j=[a_j,b_j)+\Z\big\}$ is a partition of $\TTT$ into
finitely many half open intervals, 
each $t_j$ is in $\TTT$, and addition is considered mod $ \Z$.
Then each restriction $\IE|_{I_j}$ rotates $I_j$ by $t_j$. We call the endpoints of the intervals $I_j$
the {\em discontinuities} of $\IE$.

A {\em suspension} of an interval exchange map $\IE:\TTT \to \TTT$ is a pair $(q,\gamma)$,
where $q$ is a translation surface and $\gamma:\TTT \to q
\smallsetminus \Sigma$ is a parameterized simple closed  
curve which transversely passes rightward over the vertical foliation
of $q$ such that (informally) $\gamma$ conjugates $\IE$ to the return
map of the vertical straight-line flow to $\gamma(\TTT)$. Formally,
we insist: 
\begin{itemize}
\item If $x \in \TTT$ is not a discontinuity, then the vertical straight-line trajectory
of $\gamma(x)$ hits $\gamma \circ \IE(x)$ before hitting any
singularities and before hitting any other point on $\gamma(\TTT)$. 
\item If $x \in \TTT$ is a discontinuity, then $\gamma(x)$ hits a
  singularity under the vertical straight-line flow. 
\end{itemize}

We call such a suspension {\em measure preserving} if $\gamma$ sends
Lebesgue measure 
on $\TTT$ to the measure on $\gamma(\TTT)$ induced by the Lebesgue
transverse measure to the vertical foliation on $q$. 
There are many possible suspensions of an IET but there is a natural
class of suspensions which gives the simplest possible
surfaces. Namely the
following was shown in 
\cite[\S 6]{Veech - Gauss}, \cite[\S 3]{Masur} (see also the discussion in \cite[\S 1.2]{KZ}):

\begin{prop}
\name{prop:unique stratum}
There is a unique stratum $\HH$ of translation surfaces containing
suspensions of $\IE$ so that if $(q,\gamma)$ is a suspension then $q
\in \HH$ if and only if every vertical separatrix on $q$ intersects
$\gamma(\TTT)$. Every such suspension lies in the same connected component of $\HH$.
\end{prop}

Given $\IE$, 
we will call this stratum the {\em stratum of minimal complexity}
corresponding to $\IE$. 
Recall that our strata have labeled singularities. It will be
convenient to specify 
a labeling of singular points of a suspension, which depends on $\IE$
but is independent of the suspension. To this end note that  
if $(q,\gamma)$ is a suspension of $\IE$, then the
vertical straight-line flow of the image under $\gamma$ of a
discontinuity of $\IE$ must hit a singularity. Now suppose $q$ is in
the minimal complexity stratum $\HH$ of Proposition \ref{prop:unique
  stratum}. Then a singularity of order $r$  
is associated to $r+1$ discontinuities of $\IE$. The choice of
$q$ thus determines a partition $P$
of the discontinuities of $\IE$. Conversely, it was shown by Veech
that the partition $P$ can be recovered from the combinatorics  
of  $\IE$ (as orbits of the auxiliary permutation $\sigma$
of \cite[\S 2]{Veech - Gauss}) and is independent of the choice of the
suspension. 
So if $(q,\gamma)$ is a suspension with $q \in \HH$ then there is a
bijection between the singularities of 
$q$ and elements of
the partition $P$, where the number of partition elements is the number of
singularities and a singularity of cone angle $2\pi (k+1)$ corresponds
to a block of size $k+1$ in $P$. We label the
singularities of surfaces in $\HH$ by elements of $P$, and given $\IE$, we 
will say that the suspension $(q,\gamma)$ is {\em compatibly labeled}
if for any partition element $[d] \in P$ and any discontinuity $d' \in
[d]$, the vertical straight-line flow applied to $\gamma(d')$
hits the singularity labeled by $[d]$. 
 
Fix $\IE$ and determine $\HH$ using the Proposition above. Let $(q_1,
\gamma_1)$ and $(q_2, \gamma_2)$ 
be two compatibly labeled suspensions of $\IE$ with $q_1, q_2 \in
\HH$. Corresponding to $(q_1, \gamma_1)$ and $(q_2,
\gamma_2)$, there is a {\em canonical isotopy class of
  homeomorphisms} $\varphi: q_1 \to q_2$. Homeomorphisms in this class
preserve the labeling of singularities 
and the isotopies are taken to fix the singularities. 
This class contains a homeomorphism $\varphi:q_1 \to q_2$ such that
\begin{enumerate}
\item $\gamma_2(t)=\varphi \circ \gamma_1(t)$ for all $t \in \TTT$.
\item $\varphi$ maps vertical leaves to vertical leaves, preserving
  their orientation.
\end{enumerate}
This information does not completely specify $\varphi$, but any two
such maps differ by precomposition 
by a homeomorphism of $q_1$ which preserves the vertical foliation and
preserves $\gamma_1(\TTT)$ pointwise.
All such maps are isotopic, so the isotopy class
$[\varphi]$ is well-defined and 
depends only on the suspensions.

\subsection{Rel, Real-rel, Imaginary-rel}
We describe the rel foliation as a foliation on 
$\HH_{\mathrm{m}}(\mathbf{r})$ which descends to a well-defined foliation on
$\HH(\mathbf{r})$. 
We view our cohomology classes as equivalence classes of linear maps
from the associated homology spaces. 
Observe there is a restriction map 
$$\mathrm{Res}:H^1(S, \Sigma ; \R^2) \to H^1(S; \R^2)$$
which is obtained by mapping a cochain $H_1(S, \Sigma; \R) \to \R^2$
to its restriction to the `absolute periods'
$H_1(S; \R) \subset H_1(S, \Sigma; \R)$. This restriction map is part
of the exact sequence in cohomology,  
\eq{eq: defn Res}{
H^0(S) \to H^0(\Sigma ) \to H^1(S, \Sigma )
\stackrel{\mathrm{Res}}{\to} H^1(S) \to \{0\}
}
(coefficients in $\R^2$), 
and we obtain a natural subspace 
\eq{eq: natural subspace}{
\mathfrak{R} =\ker \mathrm{Res} \subset H^1(S, \Sigma;
\R^2),
}
consisting of the cohomology classes which
vanish on $H_1(S) \subset
H_1(S, \Sigma).$ 
Since
the sequence \equ{eq: defn Res} is invariant under homeomorphisms in
$\Mod(S, \Sigma)$, the subspace 
$\mathfrak{R}$ 
is $\Mod(S, \Sigma)$-invariant. 
Since hol is equivariant with respect to the action of the group $\Mod(S,
\Sigma)$ on $\HH_{\mathrm{m}}(\mathbf{r})$ and $H^1(S, \Sigma;
\R^2)$, 
the foliation of $H^1(S, \Sigma;
\R^2)$ by cosets of the subspace $\mathfrak{R}$ induces by pullback a
foliation of $\HH_{\mathrm{m}}(\mathbf{r})$, and descends to a well-defined
foliation on $\HH(\mathbf{r}) = \HH_{\mathrm{m}}(\mathbf{r})/\Mod(S,
\Sigma)$. The area of a translation surface can 
be computed using the cup product pairing in absolute cohomology and
hence the foliation preserves the area of surfaces, and in particular
we obtain a foliation of a fixed  area  sublocus $\HH$ (see
\cite{BSW} for more details). This foliation is called the {\em rel} 
foliation.  Two nearby translation
surfaces $q$ and $q'$ are in the same plaque if the integrals of the flat structures along all 
closed curves are the same on $q$ and $q'$. 
Intuitively, $q'$ is obtained from $q$ by fixing one singularity as a
reference point and moving the other singularity.
Recall our convention that singularities are labeled, that is $\Mod(S,
\Sigma)$ does not permute the singular points. Using this one can show
that $\Mod(S, \Sigma)$ acts trivially on $\mathfrak{R} \cong H^0(\Sigma; \R)/H^0(S, \R)$ and hence 
the leaves of the rel foliation
are equipped with a natural translation structure, modeled on $\mathfrak{R}$. 
 The leaves of the rel foliation have (real)
dimension $2(k-1)$ 
(where $k=|\Sigma|$). In this paper we will focus on the case $k=2$, so that
rel leaves are 2-dimensional. We can integrate a cocycle $c  \in
\mathfrak{R}$ on any path joining distinct singularities and the
resulting 
vector in $\R^2$ will be independent of the path, since any two paths
differ by an element of $H_1(S)$. Thus in the case $k=2$ we obtain an  identification of $\mathfrak{R}$ with
$\R^2$ by the map $u \mapsto u(\delta)$ for any path joining the
singularities. Our convention for this identification will be that we
take a path $\delta$  oriented from  $\xi_1$ to $\xi_2$.

The existence of a translation structure on rel leaves implies that
any vector $u \in \mathfrak{R}$ determines
an everywhere-defined vector field on $\HH$. We can apply standard
facts about ordinary differential equations to integrate this vector
field. This gives rise to paths $\psi(t) = \psi_{q,u}(t)$ such that
$\psi(0) = q$ and $\frac{d}{dt} \psi(t) \equiv u.$ We will denote the maximal domain
of definition of $\psi_{q,u}$ by $I_{q,u}$. When  $1 \in
I_{q,u}$ we will say that $\rel^uq$ is defined and write
$\psi_{q,u}(1) = \rel^uq$. Also, in the case $k=2$  we will write 
$$
\rel^{(h)}_rq = \rel^uq \text{ when } u = (r,0),$$
and 
$$\rel^{(v)}_sq = \rel^uq \text{ when } u = (0,s). 
$$
These trajectories are called respectively the {\em real-rel} and {\em
  imaginary-rel} trajectories. 
We will use identical notations for $\bq \in \HH_{\mathrm{m}}(\mathbf{r})$, noting that
since $\pi: \HH_{\mathrm{m}} (\mathbf{r})\to \HH(\mathbf{r})$ is an
orbifold covering map, $I_{\bq, u} = 
I_{q,u}$ and $\pi(\rel^u \bq) = \rel^u q$. 

Note that the trajectories need not be defined for all time, i.e. $I_{q,u}$
need not coincide with $\R$. For instance this will happen when 
a saddle connection
on $q$ is made to have length zero, i.e. if `singularities
collide'. It was shown in \cite{MW} that this is the only
obstruction to completeness of leaves. 
Namely, in the case $k=2$, the following holds: 

\begin{prop}\name{prop: real rel main}
Let $\HH$ be a stratum with two singular points, let $\bq \in
\HH_{\mathrm{m}}$, and let $u \in \mathfrak{R}$. Then 
the following are equivalent: 
\begin{itemize}
\item 
$\rel^u\bq $ is defined. 
\item
For all saddle connections $\delta$ on $\bq$, and all $s \in [0,1]$,
$$\hol(\bq, \delta) + s \cdot u( \delta)  \neq 0.$$
\end{itemize}
\end{prop}

\begin{cor}\name{cor: real rel main}
If $q$ has two singular points and no horizontal (resp. vertical) saddle connections joining
distinct singularities, then $\rel^{(h)}_rq$ (resp. $\rel^{(v)}_sq$)
is defined for all $r,s\in \R$. 
\end{cor}

From standard results about ordinary differential equations we have
that 
the map $(q,u) \mapsto \rel^u q$ is continuous on its domain of
definition, and 
$$
\rel^{(h)}_{r_1}(\rel^{(h)}_{r_2} q) = \rel^{(h)}_{r_1+r_2 } (q), \ \
\  \rel^{(v)}_{s_1}(\rel^{(v)}_{s_2} q) = \rel^{(v)}_{s_1+s_2 } (q)
$$
(where defined). 
On the other hand we caution the reader that the rel
plane field need not integrate as a group action, i.e. it is easy to
find examples for which 
$$
\rel^{(h)}_{r} \left(\rel^{(v)}_{s} q \right) \neq \rel^{(v)}_{s}
\left(\rel^{(h)}_r q \right).
$$
We let $G$ act on the stratum $\HH$ in the
usual way and also let $G$ act on $\R^2$ by its standard linear
action. The action of $G$ is equivariant for the map $\hol$ used to
define the translation structure on rel leaves, and so this induces an
action of $G$ on the subspace $\mathfrak{R}$, and this
leads to the following result (see \cite{BSW}
for more details):
\begin{prop}\name{prop: rel and G commutation}
Let $x \in \HH$ and let $u \in
\mathfrak{R}$. If $\rel^u(x)$ is defined and $g\in G$ then
$\rel^{gu}(gx)$ is defined 
and $g(\rel^u(x))=\rel^{gu}(gx)$. In particular, if $q$ has two
singular points and has no
horizontal saddle connections joining distinct singularities,  then for all $r,s,t \in \R$, 
\eq{eq: rel commutation}{
g_t \rel^{(h)}_r q = \rel^{(h)}_{e^tr} g_t q \ \text{ and } \ g_t \rel^{(v)}_s q = \rel^{(v)}_{e^{-t}s} g_t q. 
}
\end{prop} 

\subsection{Rel and suspensions}\name{subsec: rel and suspension}
For any $\bq$ in any stratum $\HH_{\mathrm{m}}$ of marked surfaces,
one can pull back the foliation of $\R^2$ by vertical lines, via
planar charts, to obtain a 
topological singular foliation $\GG(\bq)$ on the surface $S$. Now suppose
surfaces in $\HH_{\mathrm{m}}$ have two singular points, and
$\bq'=\rel^{(v)}_s \bq$ for some $s \in \R$. Then the pulled back
foliation $\GG(\bq')$ is well-defined and topologically isomorphic to
$\GG$. That is, there is a homeomorphism $\rho: \bq \to \bq'$, which
maps each point to another point in its vertical leaf, and preserves
the measure transverse to vertical leaves. The
homeomorphism $\rho$ depends on certain choices but different choices give
rise to homeomorphisms which are isotopic to $\rho$ via an isotopy moving points
along vertical leaves. See \cite[Theorems 1.2 and 12.2]{MW}. 

Since $\rho$ only moves points inside their vertical leaf, and the
definition of a suspension involves only motion inside vertical
leaves, the following is an immediate consequence. 
\begin{prop}\name{prop: rel and suspension}
Let $q'=\rel^{(v)}_s q$ and $\rho: q \to q'$ be as above. 
If $(q, \gamma)$ is a suspension of the
$\IE: \TTT \to \TTT$ in the sense of \S \ref{sect:IETs} then $(q',
\rho \circ \gamma)$ is also a suspension of $\IE$. 
\end{prop}

\section{The Arnoux-Yoccoz IET, Suspensions, and Rel Rays}
\name{sect: rel ray}
Arnoux and Yoccoz defined the surface $x_0$ by first introducing an interval exchange
transformation of which this surface is a suspension. We will employ a
similar strategy; in fact all surfaces in the imaginary-rel ray we are
interested in will be suspensions of this IET. 

\subsection{Renormalization of the IET}
\name{sect:renormalization}
Fix an integer $\genus \geq 2$.
Let $\alpha$ be the unique positive real solution to the polynomial equation
\eq{eq: alpha}{
\alpha+\alpha^2+ \cdots + \alpha^\genus=1.}
It is known by work of \cite{AY} that the multiplicative inverse
$\alpha^{-1}$ of $\alpha$, is a Pisot number (i.e., the 
algebraic conjugates of $\alpha$ all lie in the interior of the unit
circle) and that for $\genus \geq 3$ the field $\Q(\alpha)$ is of
degree $\genus$ and is not totally real.
The value of $\alpha$ is decreasing as $\genus \to \infty$, dropping from the
reciprocal of the golden mean (when $\genus=2$) to $\frac{1}{2}$ in
the limit as $\genus \to \infty$.  

Associated to $\genus$ is an IET
$\IE:\TTT \to \TTT$. When
specifying coordinates we identify $\TTT$ with the half-open interval
$[0,1)$.  
The interval exchange $\IE$ is a composition of two involutions,
$\IE=I_2 \circ I_1$. We define $I_1:[0,1) \to [0,1)$ by partitioning
$[0,1)$ into $\genus$ subintervals: 
\eq{eq:intervals}{
J_1=[0,\alpha), ~ J_2=[\alpha, \alpha+\alpha^2), ~ \ldots, ~ J_\genus=[1-\alpha^\genus,1),}
where interval $J_k$ has length $\alpha^k$.  
Each interval $J_k$ is invariant under the map $I_1$, and the action of $I_1$ on $J_k$ satisfies
$$I_1 |_{J_k}:J_k \to J_k; \quad t \mapsto t+\frac{\alpha^k}{2} \pmod{\alpha^k}.$$
That is, $I_1$ treats each $J_k$ as a circle $\R/\alpha^k \Z$ and rotates this circle halfway around. 
The second involution does the same to the full interval $[0,1)$;
$$I_2:[0,1) \to [0,1); \quad t \mapsto t+\frac{1}{2} \pmod{1}.$$

For us, the most important aspect of the IET $\IE$ is that it
is renormalizable in the strongest sense: there is a return map to a
subinterval which is essentially the same as the original map (up to
uniform scaling and rotation).

\begin{thm}[Arnoux-Yoccoz \cite{AY}]
\name{thm:renormalization}
Let $\widehat \IE:\R/\alpha \Z \to \R/\alpha \Z$ be the first return
map of $\IE$ to $[0,\alpha)$, with $[0,\alpha)$ viewed as a
fundamental domain for $\R/\alpha\Z$. Define 
$$\psi:\R/\alpha \Z \to \TTT; \quad s \mod \alpha \Z ~ \mapsto ~ \alpha^{-1}
s+\frac{\alpha^{-1}-1}{2} \mod\Z.$$
(This map scales $\R/\alpha \Z$ to become $\R/\Z$ and then rotates by $\frac{\alpha^{-1}-1}{2}$.)
Then $\psi \circ \widehat \IE=\IE \circ \psi$.
\end{thm}

Arnoux and Yoccoz stated this for $\genus \geq 3$  and Bowman observed
that it also holds for $\genus=2$ (see \cite[\S 2]{AY} and
\cite[Appendix]{Bowman}). We give the proof for completeness
below. Before doing this, we will prove the following lemma  
which contains information on the interplay between $\IE$ and the map
$\psi$ defined above. This lemma will also be used for later arguments
involving rel rays. 

\begin{lem} Let $s \in [0,\alpha)$. Viewing $\IE$ as a map $[0,1) \to [0,1)$, we
  have: 
\label{lem:renormalization}
\begin{enumerate}
\item If $\IE(s) \geq \alpha$ then $\psi(s)=\alpha^{-1} \IE(s)-1$.
\item If $\IE(s) \leq \alpha$ then $\psi(s) \in J_\genus$ and $\IE \circ \psi(s)=\psi \circ \IE(s)$. 
\end{enumerate}
\end{lem}

\begin{proof}
Under the assumption that $s  \in [0,\alpha)$, by definition of $\IE$, we have
\eq{eq: IE(s)}{
\IE(s)=\begin{cases}
s+\frac{\alpha+1}{2} & \text{if $s<\frac{1-\alpha}{2}$,}\\
s-\frac{1-\alpha}{2} & \text{if $\frac{1-\alpha}{2}\leq s < \frac{\alpha}{2}$,}\\
s+\frac{1-\alpha}{2} & \text{if $\frac{\alpha}{2} \leq s.$}\end{cases}}
These values all lie in $[0,1)$. 

Statement (1) involves the first and third cases of (\ref{eq: IE(s)}). 
Here
$$\alpha^{-1} \IE(s)-1=\begin{cases}
\alpha^{-1} s+\frac{-1+\alpha^{-1}}{2} & \text{if $s<\frac{1-\alpha}{2}$,}\\
\alpha^{-1} s+\frac{\alpha^{-1}-3}{2} & \text{if $\frac{\alpha}{2} \leq s.$}\end{cases}$$
In either case we have $\alpha^{-1}\IE(s)-1=\alpha^{-1} s+\frac{\alpha^{-1}-1}{2} \pmod{1}$,
which is the definition of $\psi$. 

To see (2), we observe by computing with the second case of (\ref{eq: IE(s)})
that $\psi \circ \IE(s)=\alpha^{-1}s$. Also since $s \in
\left[\frac{1-\alpha}{2}, \frac{\alpha}{2} \right)$, 
we can compute
$$\psi(s)=\alpha^{-1} s +\frac{\alpha^{-1}-1}{2} \in \left
  [\alpha^{-1}-1, \frac{\alpha^{-1}}{2} \right).$$
Using the identity $\alpha^{-1}=2-\alpha^\genus$, we see that
$\psi(s) \in \left[1-\alpha^\genus, 1-\frac{\alpha^\genus}{2} \right)$, which is
the left half of the interval $J_\genus$. Using the same identity, 
$$\IE \circ \psi(s)=\psi(s)+\frac{\alpha^\genus}{2}-\frac{1}{2}=\alpha^{-1}s.$$
\end{proof}

\begin{proof}[Proof of Theorem \ref{thm:renormalization}]
Let $s \in [0,\alpha)$. We consider two cases. First suppose that $\IE(s) < \alpha$. 
Then by definition of the first return map, we have $\widehat \IE(s)=\IE(s)$.
So, the conclusion that $\psi \circ \widehat \IE(s)=\IE \circ \psi(s)$ follows from statement (2)
of Lemma \ref{lem:renormalization}.

Now suppose that $\IE(s) \geq \alpha$. Observe that $\widehat
\IE(s)=\IE^2(s)$. Using statement (1) of Lemma \ref{lem:renormalization}, we have 
$$\IE \circ \psi(s)=\IE\big(\alpha^{-1} \IE(s)-1\big).$$
Observe that the map $\zeta:x \mapsto \alpha^{-1} x-1$ sends the
interval $J_k$ affinely onto $J_{k-1}$ for 
$k \in \{2,\ldots, \genus\}$. Therefore, we observe that it commutes
with the first involution $I_1$. So, 
$$\IE \circ \psi(s)=I_2 \circ \zeta \circ I_1 \circ \IE(s).$$
If we can show that $I_2 \circ \zeta(x)=\psi \circ I_2(x)$ for $x=I_1 \circ \IE(s)$,
then we will complete the proof that $\psi \circ \widehat \IE(s)=\IE \circ \psi(s)$.
One side is
$$I_2 \circ \zeta(x)=I_2(\alpha^{-1} x-1)=\alpha^{-1} x+\frac{1}{2}+\Z.$$
Now note that $\IE(s) \geq \alpha$, so $x=I_1 \circ \IE(s) \geq \alpha$ and
therefore:
$$\psi \circ I_2(x)=\alpha^{-1} \left(x-\frac{1}{2}
\right)+\frac{\alpha^{-1}-1}{2} \mod  \Z=
\alpha^{-1} x - \frac{1}{2} \mod \Z.$$
We have verified $I_2 \circ \zeta(x)$ and $\psi \circ I_2(x)$ are equal in $\TTT$ as required.
\end{proof}

\subsection{Suspensions}
\name{sect: suspensions}
Fix $\genus \geq 2$. Consider the Arnoux-Yoccoz IET $\IE:\TTT \to
\TTT$. We work out the associated stratum of minimal complexity, and
its connected component containing suspensions of $\IE$. We recall
from \cite{KZ} that the stratum $\HH(1,1)$ is connected, and that $\HH(\genus-1, \genus-1)$ has three
connected components $\HH^{\mathrm{hyp}}(\genus-1, \genus-1), \HH ^{\mathrm{odd}} (\genus-1,
\genus-1), \HH^{\mathrm{even}}(\genus-1, \genus-1)$ when $\genus$ is odd and two
components $ \HH^{\mathrm{hyp}}(\genus-1, \genus-1), \HH^{\mathrm{non-hyp}}(\genus-1,
\genus-1)$ when $\genus \geq 4$ is even. 
\begin{prop}
\name{prop:stratum} 
The stratum of minimal complexity for $\IE$ is 
$\HH(\genus-1, \genus-1)$. If $\genus \geq 3$ is odd, then
the connected component 
is $\HH^{\mathrm{odd}}(\genus-1, \genus-1)$, and if  $\genus \geq 4$
is even then it is 
$\HH^{\mathrm{non-hyp}} (\genus-1, \genus-1)$. 
\end{prop}

\begin{figure}
   \includegraphics[width=2.4in]{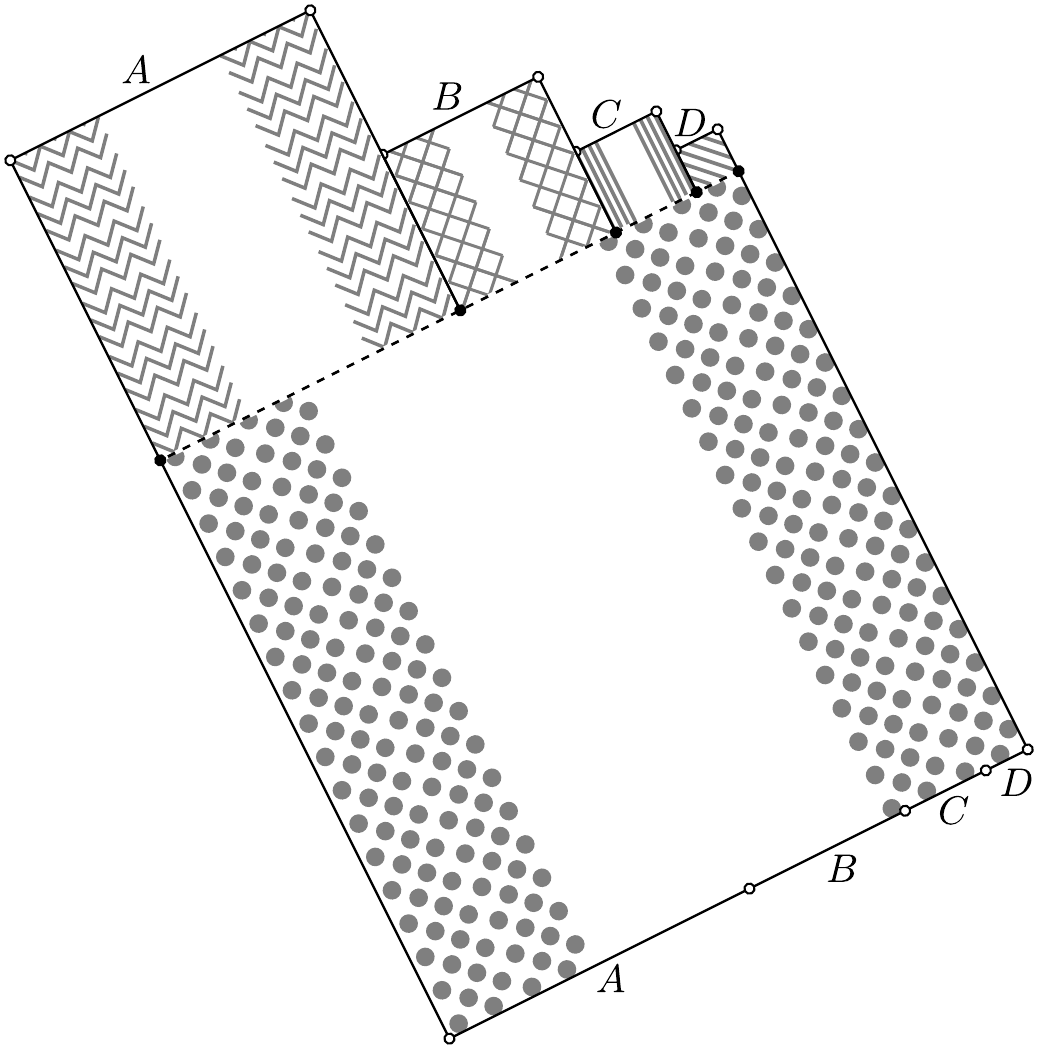}
   \includegraphics[width=2.4in]{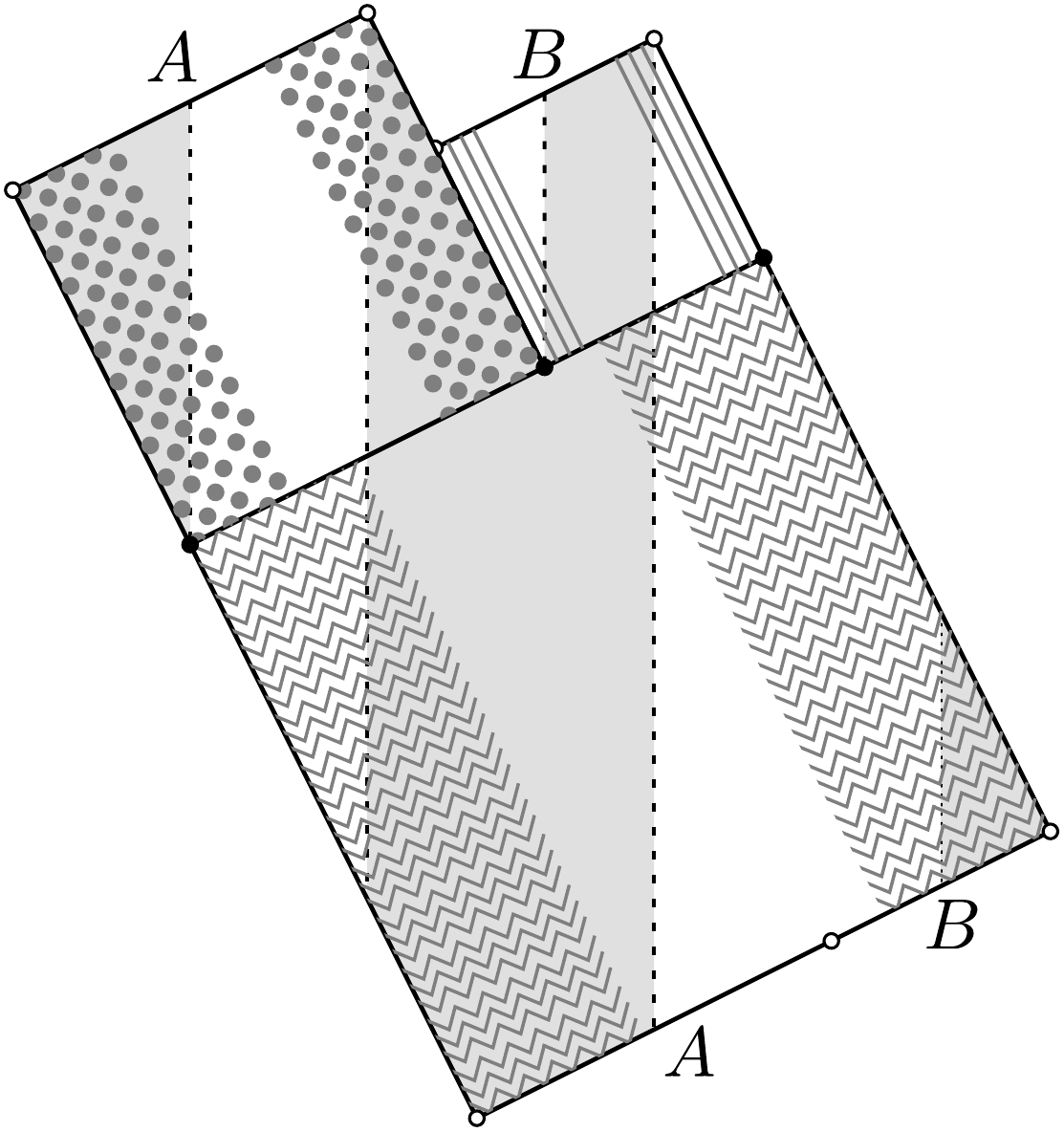}
  \caption{The surface $q_\star$ from the proof of Proposition
\ref{prop:stratum} with $\genus=4$ (left) and $\genus=2$ (right). Side
identifications for edges of slope $-2$ are indicated by the hatching. }
  \label{fig:suspension}
  \end{figure}

\begin{proof}
We will exhibit an explicit suspension $(q_\star,\gamma_\star)$ of $\IE$ with the property
that every vertical leaf of $q_\star$ intersects $\gamma_\star$.  
By Proposition \ref{prop:unique stratum} (following Veech), $q_\star$
must then lie in the stratum of minimal complexity. We will then
determine the connected component of the corresponding stratum. 

The suspension
$q_\star$ will be made of $\genus+1$ squares of dimensions $1 \times 1$, $\alpha
\times \alpha$, \ldots, $\alpha^\genus \times \alpha^\genus$. See
Figure \ref{fig:suspension} for examples. These squares are tilted so
that the edges have slope $\frac{1}{2}$ and $-2$. On
each square, opposite sides of
slope $-2$ are identified, to form a cylinder. The choice of slopes guarantees that a half
rotation is performed when following the vertical straight-line flow
from the bottom to the top of each of these cylinders, where 
the rotation is measured relative to the vertices of the square. The
slope $\frac{1}{2}$ edges of the smaller squares are glued to the
slope $\frac{1}{2}$ edge of the largest square. As we move from right
to left over either slope $\frac{1}{2}$ edge of the largest square we
glue to the other squares in order of largest to smallest. This
defines $q_\star$.  

The gluing rules imply that all the bottom corners of the small
squares are identified with each other and all the top corners of the
small squares are identified with each other. Thus there are two cone angles each of cone angle
$2 \pi \genus$; one lies on the top of the large square, and the other
lies on the bottom. Therefore the surface $q_\star$ lies in the stratum
$\HH(\genus-1,\genus-1)$.  

Let $\gamma_\star:\TTT \to q_\star$ be a constant speed parametrization of a geodesic core curve 
of the cylinder formed by the largest square, which we parameterize in
such a way that $\gamma_\star(0)$ hits the top corner of this largest square
under vertical straight-line flow. Based on the definition of $\IE$ as
a product of half rotations, we see that  
$\gamma_\star$ conjugates $\IE$ to the return map of the vertical
straight-line flow to $\gamma_\star(\TTT)$. That is, $(q_\star, \gamma_\star)$ is a
suspension of $\IE$. It also can be explicitly checked that every
vertical separatrix intersects 
$\gamma_\star$, so $q_\star$ lies in the connected component of the minimal
complexity stratum containing suspensions of $\IE$. 

The connected component of the stratum can we worked out following 
\cite{KZ}. When $\genus \geq 3$, the surface $q_\star$ is not hyperelliptic.
(If $q_\star$ admitted an affine automorphism with derivative $-\mathrm{Id}$,
it would have to preserve all but the largest square, because the
smaller squares are part of the Delaunay decomposition of
$q_\star$. But because of the way these squares are glued to the
largest square, rotating each of these squares halfway around does not
extend to an automorphism of the full surface.) 
Using the classification of connected components 
\cite[Theorem 1]{KZ}, 
we see that this determines the connected component containing these
suspensions when $\genus $ is even. 
When $\genus \geq 3$ is odd, we need to consider the spin
invariant. This can be computed using a symplectic basis 
for $H_1(S;\Z)$ where $S$ is the topological surface underlying the
translation structure provided by $q_\star$.  
A symplectic basis is given by core curves of the $\genus$ cylinders
of slope $-2$ (each of which pass through one of the smaller squares and the
largest square) together with core curves of the smallest $\genus$
cylinders of slope $\frac{1}{2}$ (all but the one through the largest
square). Since these are cylinders, the index of the field tangent to
each curve is zero; see \cite[Lemma 2]{KZ}. Thus the parity of the
spin structure on $q_\star$ is given by  
$\genus \pmod{2}$; see \cite[equation (4)]{KZ}. So when $\genus \geq
3$ is odd, $q_\star \in \HH^{\mathrm{odd}}(\genus-1,\genus-1)$.  
\end{proof} 

We will use $\bullet$ and $\circ$ to label the singularities of
surfaces in $\HH(\genus-1, \genus-1)$. 
A suspension $(q, \gamma)$ with $q \in \HH(\genus-1, \genus-1)$ is
{\em compatibly labeled} if 
\begin{enumerate}
\item[(a)] Whenever $t$ is an endpoint of $J_k$ from
  (\ref{eq:intervals}) for some $k$, the vertical straight-line flow
  of $\gamma(t)$ 
hits the singularity labeled $\bullet$. 
\item[(b)] Whenever $t'$ is a midpoint of an interval $J_k$, the
  vertical straight-line flow applied to $\gamma(t')$ 
hits the singularity labeled $\circ$. 
\end{enumerate}
To see that this definition gives a well-defined labeling (i.e., that
the label of a singularity is independent of
the choice of $k$), we just need to check on one surface. The reader
may 
consider the surfaces from the proof above. Statements (a) and (b)
above define a partition of the discontinuities of $\IE$, and
the reader can verify that the labeling is compatible in the sense 
of \S \ref{sect:IETs}, justifying our terminology.

The renormalizing map $\psi$ from Theorem \ref{thm:renormalization}
interacts with the suspensions of $\IE$. 
Recall that $\psi:\R/\alpha \Z \to \TTT$ scales uniformly and rotates. 
Let $(q, \gamma)$ be a suspension of $\IE$ with $q \in \HH(\genus-1,
\genus-1)$, and consider the map $\gamma \circ \psi^{-1}:\TTT \to q$
(where we treat $\psi^{-1}$ as a map $\TTT \to [0,\alpha)$). This 
map is not continuous, because $\gamma(0) \neq \gamma(\alpha)$. We
will fix this by pushing $\gamma \circ \psi^{-1}$ along the leaves of
the vertical foliation.
Observe that both $0$ and $\alpha$ are discontinuities for the map $\IE$, and
$$\IE(0)=\lim_{t \to \alpha-} \IE(t)=\frac{1+\alpha}{2}.$$
Let $F_s: q \to q$ be the vertical straight-line flow. From the above, we have
$$\gamma \left(\frac{1+\alpha}{2}\right)=\lim_{t \to 0+}
F_{r(t)} \gamma(t)=\lim_{t \to \alpha-} F_{r(t)}
\gamma(t),$$ 
where $r(t)$ denotes the return time of $\gamma(t)$ to $\gamma(\TTT)$
under the flow $\{F_t\}$. 
Choose a continuous function $s:[0,\alpha] \to \R$
satisfying $s(0)=r(0)$, $s(\frac{\alpha}{2})=0$, $s(\alpha)=r(\alpha)$  
and $0 \leq s(t) \leq r(t)$ for all $t \in (0, \alpha)$. (The value
$\frac{\alpha}{2}$ is singled out because it is the only discontinuity
of $\IE$ in the open interval $(0,\alpha)$). These choices make the
following a well-defined continuous curve: 
\eq{eq: hat gamma}{
\hat \gamma:\TTT \to q; \quad t \mapsto 
\begin{cases}
\gamma(\frac{1+\alpha}{2}) & \text{if $\psi^{-1}(t)=0$}, \\
F_{s(t)} \circ \gamma \circ \psi^{-1}(t) & \text{otherwise}.
\end{cases}
}
This new curve is well-defined because when $\psi^{-1}(t) \neq 0$, 
the vertical straight-line flow of $\gamma \circ \psi^{-1}(t)$ does not hit a singularity
up to time $s(t)$. Although $\hat \gamma$
depends on the choice of the function $s(t)$, different choices will
only change $\hat \gamma$ by an isotopy along vertical leaves. As this
will not affect our arguments we suppress the dependence on $s(t)$
from the notation. We denote by $\gamma \mapsto \hat \gamma$ this operation which produces
a new curve $\hat \gamma$ from $(q,\gamma)$.

We have the following
corollary to Theorem \ref{thm:renormalization}. The proof is a simple calculation which we leave to the reader.

\begin{cor}
\name{cor:second suspension}
If $(q,\gamma)$ is a compatibly labeled suspension of $\IE$ with $q
\in \HH(\genus-1, \genus-1)$, then so is $(q,\hat \gamma)$. 
\end{cor}

\begin{dfn}
\name{def: renormalization mapping class}
In the context of the corollary, 
let $\varphi:q \to q$ be a homeomorphism from the canonical isotopy
class corresponding to the suspensions  $(q,\gamma)$ and $(q, \hat
\gamma)$ (see \S \ref{sect:IETs}). 
We call $[\varphi]$ the {\em renormalization mapping class} of the
suspension $(q, \gamma)$.  
\end{dfn}

\subsection{Completely periodic rel rays}
\name{sect: rel rays}
In this section, we will describe an imaginary-rel ray in $\HH(\genus-1,\genus-1)$
so that every surface on the ray is horizontally completely
periodic. In this section we fix $\genus \geq 2$. The argument is 
 the same for each $\genus$. 

\begin{figure}
   \includegraphics[width=4in]{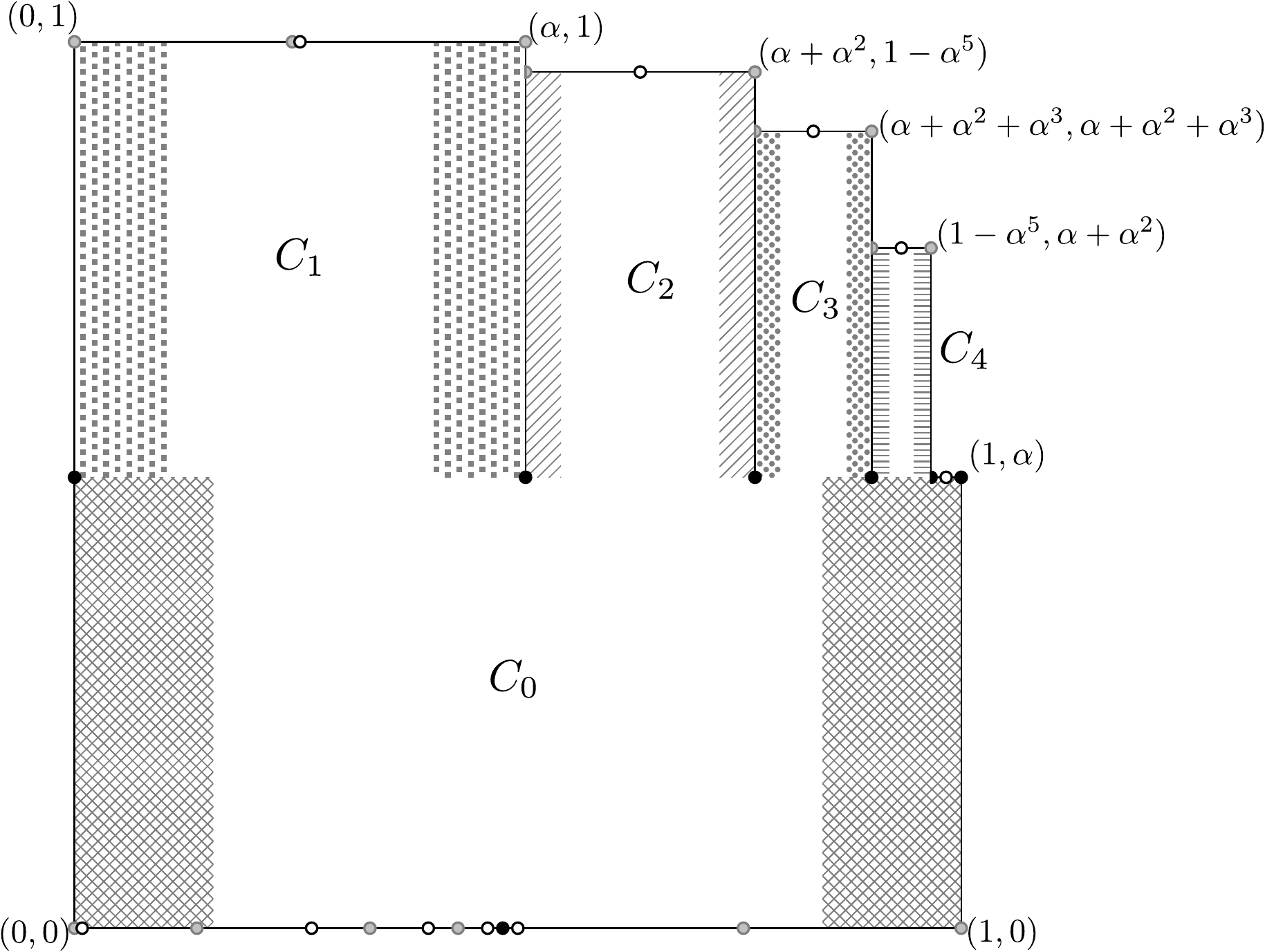}
  \caption{The surface $q_0$ with genus $\genus=5$, is made up of five horizontal cylinders. A point $(x,y)$ on a top boundary is glued to the point $(\IE(x),0)$.}
  \name{fig:start}
  \end{figure}

We will describe a surface $q_0 \in \HH(\genus-1,
\genus-1)$ which is horizontally completely periodic and is a
suspension of $\IE$. We present $q_0$ as a disjoint union of rectangles
with edge identifications. 
We define the rectangles
\eq{eq:old rectangles}{
\begin{array}{c}
\displaystyle R_0=[0,1] \times [0,\alpha],\\
\displaystyle 
R_k=\overline{J_k} \times \left[\alpha,\sum_{j=1}^{\genus-k+1} \alpha^{j}\right]\quad \text{for $k \in \{1,\ldots, \genus-1\}$}
\end{array}}
and with $\overline{J_k}$ the closure of the interval $J_k$ from
(\ref{eq:intervals}).
The edge identifications follow:
\begin{itemize}
\item We identify opposite vertical sides of $R_k$ for $k \in \{0,
  \ldots, \genus-1\}$ by a horizontal translation, 
so that each $R_k$ forms a horizontal cylinder $C_k$ on $q_0$. 
\item We glue the bottom edge of $R_k$ for $k \in \{1,\ldots,
  \genus-1\}$ with the top edge of $C_0$ by the identity map (the
  bottom edge of each $R_k$ is a subset of the top edge of $R_0$,
  viewed as subsets of $\R^2$).  
\item If $(x,y)$ is in the top edge of $R_k$ for $k \in \{1,\ldots,
  \genus-1\}$ or if $(x,y)$ is in the top edge of $C_0$ with $x \in
  J_\genus$, then we identify $(x,y)$ with $(\IE(x),0)$ in the bottom of $C_0$.
\end{itemize}
We label the singularities so that $q_0$ is a compatibly labeled suspension of $\IE$.
Figure \ref{fig:start} shows an example of the surface $q_0$. 

Note that $q_0$ is naturally a measure preserving suspension of
$\IE$. To see this define $\gamma_0$ to be a core curve of $C_0$ via 
$$\gamma_0:\TTT \to q_0; \quad x \mapsto \left(x, \frac{\alpha}{2}\right) \in C_0.$$
It should be clear that because of the identifications above,
$\gamma_0$ conjugates $\IE$ to the first return  
map of the vertical straight-line flow to $\gamma_0(\TTT)$. Recall
that Corollary \ref{cor:second suspension} 
described a second suspension $(q_0, \hat \gamma_0)$ constructible
from $(q_0, \gamma_0)$. Reviewing the definitions, we see that:  
\begin{prop}
\name{prop: gamma hat}
Up to isotopy along vertical leaves, the curve
$\hat \gamma_0$ on $q_0$ is defined by  
\eq{eq: parametrization of hat gamma}{
\hat \gamma_0: \TTT \to q_0; \quad x \mapsto \left(\psi^{-1}(x),
  \alpha+\frac{\alpha}{2}\right)}
which is a parameterized core curve of the cylinder $C_1$. 
\end{prop}

We will be investigating the imaginary rel leaf through $q_0$. Define  $q_s=\rel^{(v)}_s q_0$
whenever this definition makes sense. By Proposition \ref{prop: rel
  and suspension}, 
we can push $\gamma_0$ through the rel leaf obtaining parameterized
curves $\gamma_s$ in $q_s$ 
taken from the same homotopy class (up to this deformation) which make
$(q_s,\gamma_s)$ also 
a compatibly labeled measure preserving suspension of $\IE$. 

The following is our main result on the imaginary rel leaf through $q_0$. 

\begin{thm}
\name{thm:completely periodic ray}
Fix $\genus \geq 2$, and let $q_0$ be the surface defined above. 
Define 
\eq{eq: beta}{
\beta=\frac{\alpha^2}{1-\alpha}=\sum_{j=-\infty}^{-2}
\alpha^{-j}=\alpha^2+\alpha^3+\alpha^4+\cdots.
}
Then the surface $q_s=\rel^{(v)}_s q_0$ is well-defined for every $s
> -\beta$. 
Define $x_{t}=q_{t-\beta}$ for all $t>0$. Then:
\begin{enumerate}
\item 
For every $t>0$, we have 
\eq{eq: crucial property}{
\til g x_{\alpha^{-1}t} = x_t,
\quad \text{where} \quad
\til g=\left(\begin{array}{rr}
\alpha^{-1} & 0 \\
0 & \alpha \end{array}\right).}
Thus there is an affine homeomorphism $\varphi_t:x_{\alpha^{-1} t} \to
x_{t}$ with derivative $\til g$. Let $\rho_t:x_t \to x_{\alpha^{-1}
  t}$ be a homeomorphism obtained by the rel deformation  as in 
\S \ref{subsec: rel and suspension}. Then the composition $\varphi_t \circ \rho_t:x_t \to x_t$
lies in the renormalization mapping class of $(x_t,
\gamma_{t-\beta})$.
\item Every surface $x_t$ with $t>0$ is horizontally completely
  periodic.  Concretely, let $k$ be the integer so that 
\eq{eq: alpha interval}{
\sum_{j=-\infty}^k \alpha^{-j} \leq t < \sum_{j=-\infty}^{k+1} \alpha^{-j}.}
There are two distinct combinatorial cases:
\begin{enumerate}
\item If equality is attained on the left of (\ref{eq: alpha
    interval}) then $x_t$ has a horizontal cylinder decomposition with
  $\genus$ cylinders with circumferences $\alpha^{k+2}$,
  $\alpha^{k+3}$, $\alpha^{k+4}$, \ldots, $\alpha^{k+\genus+1}$. 
\item Otherwise, $x_t$ has a horizontal cylinder decomposition with
  $\genus+1$ cylinders with circumferences $\alpha^{k+2}$,
  $\alpha^{k+3}$, \ldots, $\alpha^{k+\genus+2}$. 
In this case, each cylinder has only one singularity in each of its
boundary components, where the cylinder of greatest circumference has 
$\bullet$ appearing on its top component, and $\circ$ appearing
on its bottom component, and the other cylinders have $\circ$ on their
top component and $\bullet$ on their bottom component. 
\end{enumerate}
\end{enumerate}
\end{thm}

The proof will show that the heights of the cylinders in $x_t$ can
also be explicitly described, see Figure \ref{fig:cylinder height
  graph}.

\begin{figure}
   \includegraphics[width=3.5in]{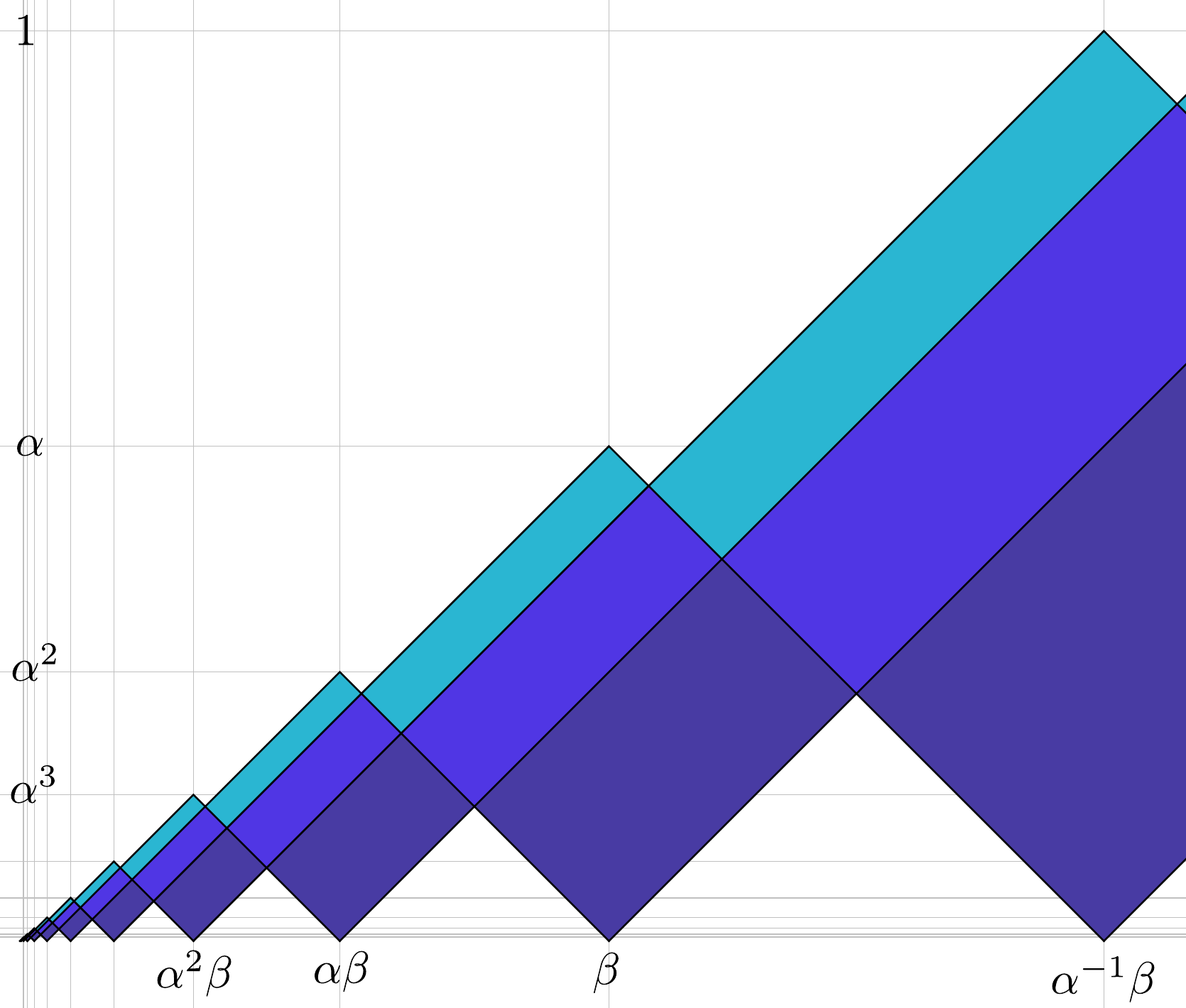}
  \caption{Black lines 
    represent heights of cylinders
    in $x_t$ as a function of $t$ (with $\genus=3$).}
  \name{fig:cylinder height graph}
  \end{figure}


The proof of Theorem \ref{thm:completely periodic ray} occupies the
rest of this subsection.
Consider the imaginary-rel flow applied to $q_0$ for distance
$s$, by which we defined the surface $q_s$.
This flow moves the white singularity $\circ$ upward, or equivalently,
moves the black singularity $\bullet$ downward. We will describe an
explicit construction of $\rel^{(v)}_s q_0$ (following
\cite{McMullen-twists}). Since $\bullet$ has cone angle $2 \pi 
\genus$, there are $\genus$ downward pointing prongs starting at this
singularity, and we form downward pointing slits starting at these
prongs of length $s$. This construction works as long as $\bullet$ is
not the top point of a vertical saddle connection of length at most $s$.
Define a new translation surface by
using the same 
polygonal presentation, keeping all gluing maps the same, except for
the gluing along the slits. Along the slits we reglue so that the former location
of the black singularity is replaced by $\genus$ regular points. This
determines the gluings of the slits, and the bottom endpoints of the
slits are forced to be identified into a $2 \pi \genus$ cone
singularity which we label $\bullet$ in the new surface. 
By examining the effect on holonomies of paths, it can be
observed that the resulting surface is indeed $\rel^{(v)}_s q_0$; that
is for every $\gamma$ joining a singular point to itself, its holonomy
is unchanged by this surgery, and for any directed $\gamma$ from the
black singularity to the white singularity, its holonomy changes by
adding $(0,s)$. We will refer to this presentation of the imaginary-rel surgery
as the {\em slit construction.}  

\begin{figure}
   \includegraphics[width=4in]{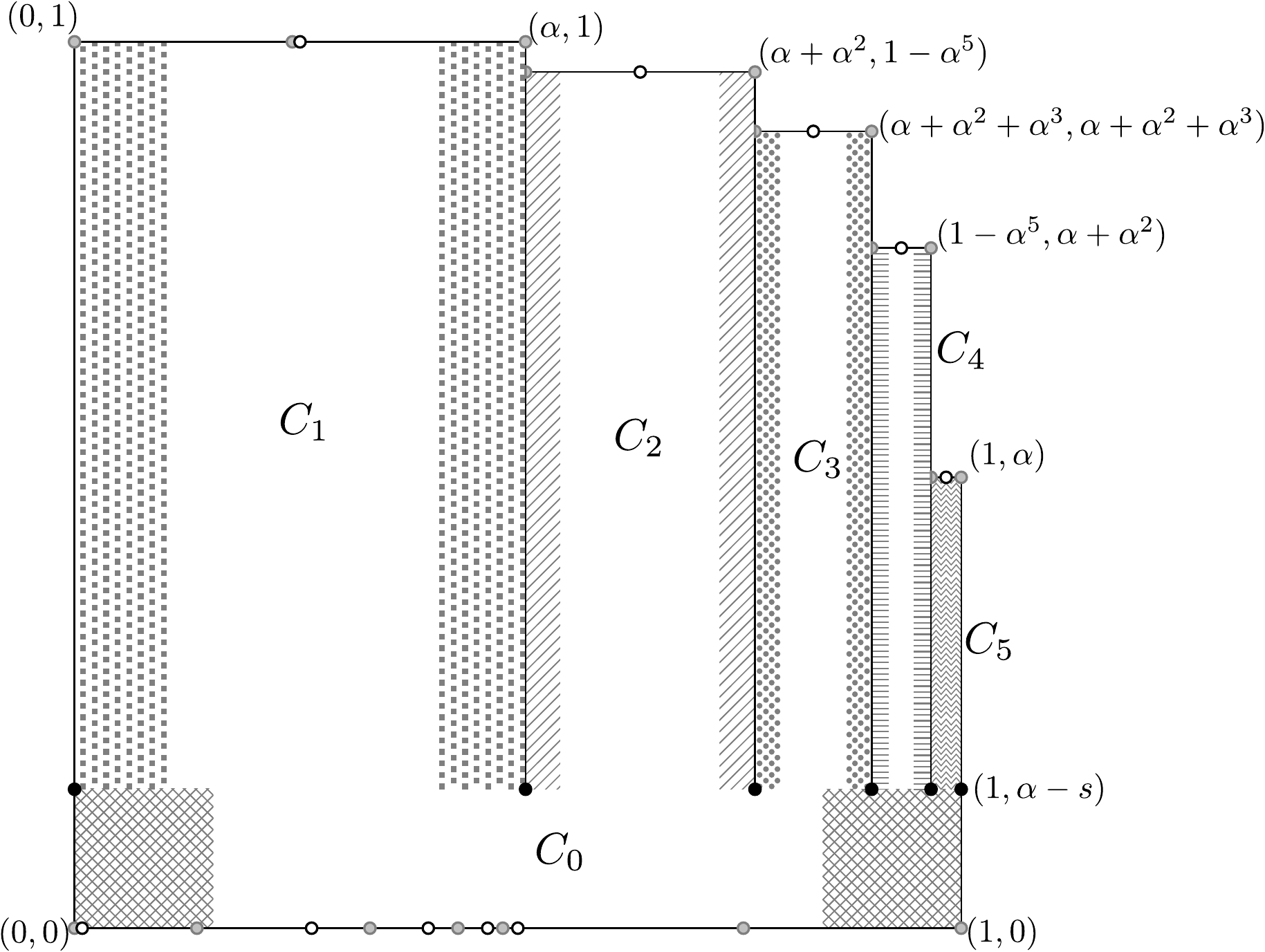}
  \caption{The surface $q_s$ with $0<s<\alpha$ when
    $\genus=5$. Again, a point $(x,y)$ on a top boundary is glued to
    the point $(\IE(x),0)$. } 
  \name{fig:partial rel}
  \end{figure}

First consider the slit construction applied to construct $q_s=\rel^{(v)}_s q_0$
for $0<s<\alpha$. The
$\genus$ downward slits cut into but not entirely  
through the cylinder $C_0$. Regluing the slits has the effect of
shortening the cylinder $C_0$ by $s$ and  
heightening the cylinders $C_1,\ldots,C_{\genus-1}$ by $s$.
This surgery also creates a new cylinder which we call $C_\genus$.
This cylinder emerges from the interval 
$J_{\genus} \times \{\alpha\}$, where the top boundary of $C_0$ touches
the bottom boundary of $C_0$. The segment $J_{\genus} \times \{\alpha\}$ has
$\circ$ as its midpoint and $\bullet$ as its endpoints. The slits
we cut move downward from these two endpoints. Since the
other edges of the slits have been accounted for, we must be gluing
the slits on either side of this interval together to make a new
cylinder. A surface of this form is shown in Figure \ref{fig:partial rel}.

\begin{prop}
\name{prop:complete periodicity}
For $0 < s<\alpha$, the surface $q_{s}$ admits a horizontal
cylinder decomposition with $\genus+1$ cylinders $C_0$, \ldots, $C_\genus$. The
circumferences of cylinder $C_k$ is $\alpha^k$, while the height $h_k$
of the cylinder $C_k$ is described by 
$h_0=\alpha-s$ and 
$$h_k=s+\sum_{j=2}^{\genus-k+1} \alpha^j= s+\beta-\alpha^{\genus-k}
\beta \quad \text{for $k=1,\ldots,\genus$.}$$ 
The homology classes $[C_k]$ of core curves of these cylinders oriented rightward 
are linearly independent as elements of $H_1(S \smallsetminus \Sigma,
\Q)$.
\end{prop}

\begin{proof}
All but the last sentence should be clear from the discussion above.
Consider that the holonomy map $\hol:H_1(S \smallsetminus \Sigma, \Q) \to \R^2$
is $\Q$-linear. Observe $\hol_x([C_i])=\alpha^i$.
Letting $W=\spa_\Q \{[C_0], \ldots, [C_{\genus-1}]\}$ (omitting $[C_\genus]$) we see
$\hol_x(W)=\Q(\alpha)$. Since $\alpha$ is an algebraic number of
degree $\genus$ \cite[Lemma 3]{AY}, we know that $\hol|_W$ is
injective and $\dim_\Q(W)=\genus$. 
Using the description of $q_s$ provided above we can see that 
$[C_0]-[C_1]-[C_2]-\cdots-[C_\genus]$ can be represented by a
loop which wraps once around $\bullet$ in the counterclockwise direction.
This nontrivial class is not in $W$ since $\hol|_W$ is injective, and hence
$[C_\genus] \notin W$. Thus  the homology classes of these $\genus+1$ cylinders are linearly
independent over $\Q$.
\end{proof}

Now consider using the slit construction to produce
$q_{\alpha}$. We think of the slit construction as
manipulating the presentation of $q_0$. Slits of length $\alpha$ cut
entirely across the rectangle $R_0$, and after the identifications 
of slits each of the rectangles $R_k$ for $k \in \{1,\ldots,
\genus-1\}$ 
has been made taller by $\alpha$. Also, a new rectangle is formed
below the interval $J_\genus \times \{\alpha\}$.  
We see that the surface $q_\alpha$ can be presented as a
disjoint union of rectangles with edge identifications. These
rectangles are  
\eq{eq:new rectangles}{R_k'=\overline{J_k} \times
  \left[0,\sum_{j=1}^{\genus-k+1} \alpha^{j}\right]\quad \text{for $k
    \in \{1,\ldots, \genus\}$}.} 

Edge identifications are given by:
\begin{itemize}
\item Opposite vertical edges of each $R_k'$ are glued by translation
  to form cylinders $C_k'$ for $k \in \{1,\ldots, \genus\}$. 
\item If $(x,y)$ is a point in the interior of a top edge of some
  $R_k$, we identify it to $(\IE(x),0)$ in the bottom edge of one of
  the rectangles. We extend this identification continuously to the
  corners. 
\end{itemize}
Figure \ref{fig:reled} shows an example presentation.

\begin{figure}
   \includegraphics[width=4in]{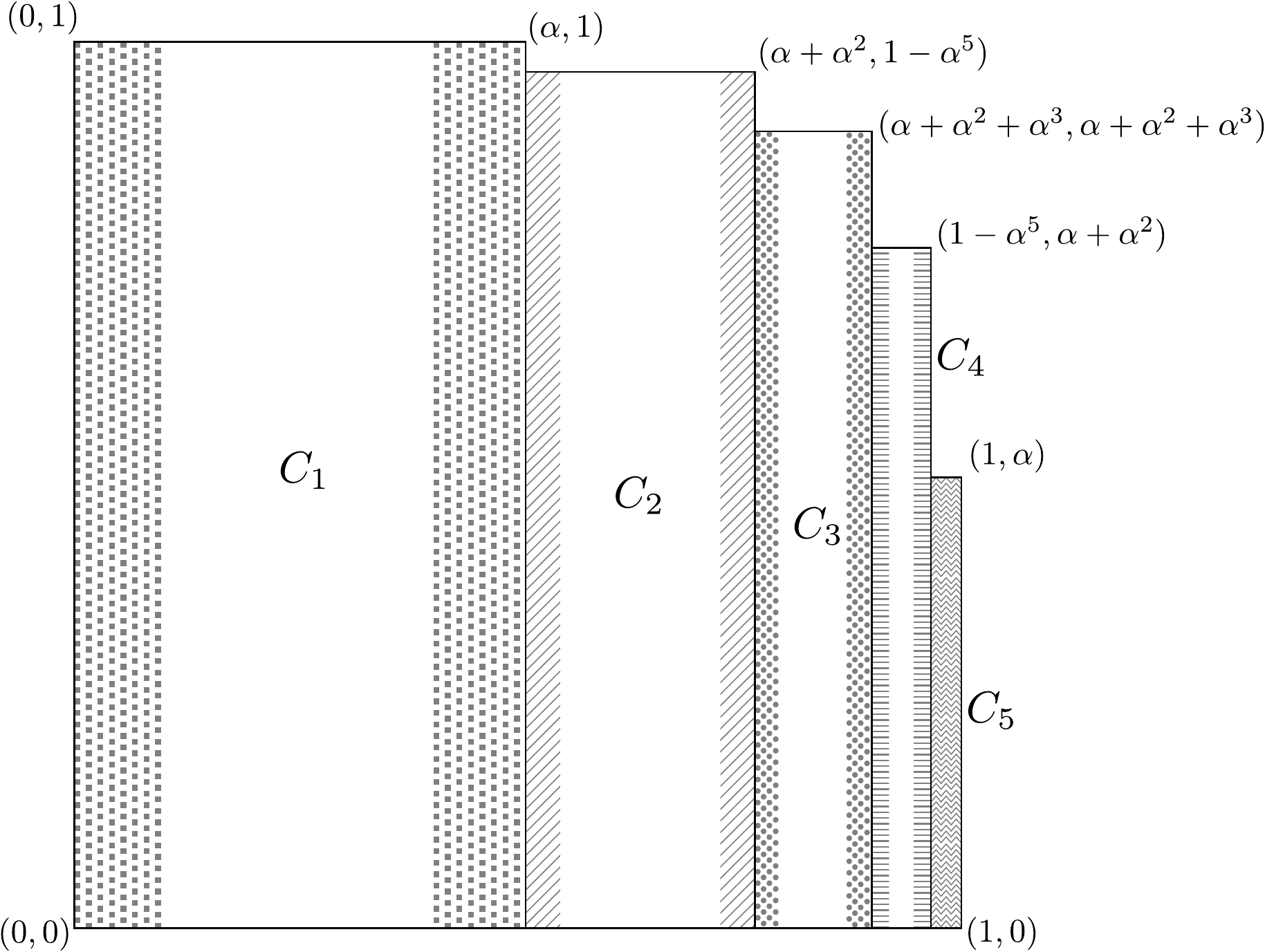}
  \caption{The surface $q_\alpha$ when
    $\genus=5$. Again, a point $(x,y)$ on a top boundary is glued to
    the point $(\IE(x),0)$.} 
  \label{fig:reled}
  \end{figure}

The following is the main ingredient in the proof of Theorem
\ref{thm:completely periodic ray}. 

\begin{lem}
\label{lem:affine homeomorphism} 
There is an affine homeomorphism $\varphi:q_\alpha \to q_0$ with
derivative $\til g$ (as in \eqref{eq: crucial property}), which
carries each cylinder $C_{i+1}'$ of $q_\alpha$ to cylinder $C_i$ of
$q_0$. Let $\rho:q_0 \to q_\alpha$ 
be a homeomorphism moving points in their vertical leaf, obtained from
the deformation through the imaginary 
rel leaf (see \S \ref{subsec: rel and suspension}). Then $\varphi \circ \rho$ 
is a representative of the renormalization mapping class of the
suspension $(q_0, \gamma_0)$. 
\end{lem}
\begin{proof}
We will explicitly define an affine homeomorphism $\varphi:q_\alpha
\to q_0$ with derivative $\til g$. We begin by defining $\varphi$ on
the cylinder $C_1'$. Recall that we think of $C_1'$ as $R_1'$ with
vertical edges identified by horizontal translation. 
Consider the map 
$$\varphi_1:R_1' \to C_0; \quad (x,y) \mapsto \Big(
\psi(x),\alpha y\Big),$$
where $\psi$ is the renormalizing map from Theorem \ref{thm:renormalization}. 
Observe that this map descends to a well-defined affine map $C_1' \to C_0$ with derivative $\til g$.
For $k \in \{2, \ldots \genus\}$, we define affine bijections
$$\varphi_k:R_k' \to R_{k-1}; \quad (x,y) \mapsto \til g (x,y)+(-1,\alpha).$$
The check that this is really an affine bijection is left to the reader.
These maps descend to well-defined maps $C_k'\to C_{k-1}$ for $k \geq 2$. 

We define $\varphi:q_{\alpha} \to q_0$ by requiring that its
restriction to each $C'_k$ is $\varphi_k$, for $k \in \{1, \ldots,
\genus\}$. For this to be well-defined, we need to show that the maps
$\varphi_k$ agree  
on the horizontal boundaries of the cylinders. We'll write $(x_1,y_1)_j \sim (x_2,y_2)_k$ to denote
that $(x_1,y_1) \in R_j$ is identified with $(x_2,y_2) \in R_k$ along
a horizontal edge in $q_0$. We'll use similar notation 
but with the equivalence relation $\approx$ to denote this equivalence for $q_\alpha$. 

First consider a point $(x,1)$ taken from $C_1'$ such that $\IE(x) \geq \alpha$. Then we have 
$$(x,1)_1 \approx \big(\IE(x),0\big)_k \quad \text{for some  $k \in \{2, \ldots, \genus\}$.}$$
We need to check that the images of these points are identified, i.e., that
$$\varphi_1 (x,1)_0 \sim \varphi_k\big(\IE(x),0\big)_{k-1}.$$
We compute:
$$\varphi_1 (x,1)=\big(\psi(x),\alpha\big) \quad \text{and} \quad
\varphi_k\big(\IE(x),0\big)=\big(\alpha^{-1} \IE(x)-1,\alpha).$$
Equality of these values points is given by (1) of Lemma \ref{lem:renormalization} so these points are identified $x_\alpha$. 

Now suppose that $(x,1)$ is taken from $C_1'$ and $\IE(x)<\alpha$. Then, 
$$(x,1)_1 \approx \big(\IE(x),0\big)_1.$$
We need to check that 
\eq{eq:check2}{
\varphi_1 (x,1)_0 \sim \varphi_1\big(\IE(x),0\big)_{0}.}
We have $\varphi_1(x,1)=\big(\psi(x),\alpha\big)$. By statement (2) of Lemma \ref{lem:renormalization},
we know that $\psi(x) \in J_\genus$. So from our description of $x_\alpha$,
$$\big(\psi(x),\alpha\big)_1 \sim \big(\IE \circ \psi(x),0\big)_1.$$
On the other hand, we have 
$$\varphi_1\big(\IE(x),0\big)=\big(\psi \circ \IE(x),0\big)=\big(\IE \circ \psi(x),0\big)$$
by the remainder of statement (2). We have shown (\ref{eq:check2}) as desired.

We have shown that our map is well defined on the top edge of $C_1'$. Now consider a point $(x,y)$
in the top of $C_k'$ for some $k \in \{2, \ldots, \genus\}$. Observe that $\IE(x) \in J_1$ because $x \geq \alpha$. 
From our description of $x_\alpha$, we have
$$(x,y)_k \approx \big(\IE(x),0\big)_1.$$
So, we need to check that 
\eq{eq:check3}{
\varphi_k (x,y)_{k-1} \sim \varphi_1\big(\IE(x),0\big)_{0}.}
Let $x_0=\IE^{-1}(x)$, which is in the interval $J_1$ by definition of $\IE$. Then,
$$\varphi_k(x,y)_{k-1}=\big(\alpha^{-1} \IE(x_0)-1,\alpha(y+1)\big)=\big(\psi(x_0), \alpha(y+1)\big),$$
where we use (2) of Lemma \ref{lem:renormalization}. The $y$-coordinate must lie on a top edge by prior work. 
Using the gluing rules in $q_\alpha$, we have
$$\varphi_k(x,y)_{k-1} \sim \big(\IE \circ \psi(x_0), 0 \big)_0.$$
Letting $\widehat \IE$ be the return map to $J_1$ as in Theorem \ref{thm:renormalization}.
Note that $\widehat \IE(x_0)=\IE^2(x_0)$, and using this theorem, we see
$$\IE \circ \psi(x_0)=\psi \circ \widehat \IE(x_0)=\psi \circ \IE^2(x_0).$$
Since $\varphi_1\big(\IE(x),0\big)=\big(\psi \circ \IE^2(x_0),0\big),$
we have verified (\ref{eq:check3}).

Now we will discuss the remaining statements. We can realize $\rho$ as
in the Lemma by a homeomorphism $\rho:q_0 \to q_\alpha$ as stated in
the Lemma which carries each cylinder $C_k$ to $C_k'$ for $k=1,
\ldots, \genus-1$. 
By Proposition \ref{prop: gamma hat},
the curve $\hat \gamma_0$ can be taken to be a core curve of $C_1$
with the provided parametrization (\ref{eq: parametrization of hat gamma}). 
Then up to homotopy along the leaves, the image under $\rho$ of $\hat
\gamma_0$ can be taken to be a core curve of $C_1'$ with the
$x$-coordinate parameterized as in (\ref{eq: parametrization of hat gamma}). We observe that
our definition of $\varphi$ 
carries $\rho \circ \hat \gamma_0$ to $\gamma_0$ up to homotopy along
the leaves. Since the composition 
$\varphi \circ \rho$ also respects the vertical foliation, it is in
the renormalization mapping class 
of the suspension $(q_0, \gamma_0)$. 
\end{proof}

\begin{proof}[Proof of Theorem \ref{thm:completely periodic ray}]
Define $\beta=\sum_{j=-\infty}^{-2} \alpha^{-j}$ and $x_t=q_{t-\beta}$ as in the
statement of the theorem. By definition $x_{t}=\rel^{(v)}_{t-\beta}
x_\beta$.  

We will begin by arguing that $x_t$ is well defined for $t>0$.
Let $\til g$ be as in the statement of the theorem. Through the use of
(\ref{eq:  rel commutation}), we have 
\eq{eq: commutation 2}{
\rel^{(v)}_{\alpha^{-k} s} \circ \til g^{-k}(x) =\til g^{-k} \circ
\rel^{(v)}_{s}(x), }
for any $x \in \HH$ and $k \in \Z$ for which at least one side of this equation is defined.
Lemma \ref{lem:affine homeomorphism} tells us that
\eq{eq: one step}{
\rel^{(v)}_\alpha x_{\beta}=\til g^{-1} x_\beta.}
Define $r(k)=\sum_{j=-\infty}^{k-2} \alpha^{-j}$. We claim
\eq{eq: inductive claim}{
x_{r(k)} =\til g^{-k} x_\beta \quad \text{for each $k \in \Z$.}
}
Recalling that $\beta=\sum_{j=-\infty}^{-2} \alpha^{-j}=r(0)$, we see this is tautologically true when $k=0$
and true when $k=1$ by (\ref{eq: one step}). We can verify this for
all $k \in \Z$ by an induction using (\ref{eq: commutation 2}) which
we leave to the reader. 
%
Observe that $r(k)=\alpha^{-k} \beta$, so we can rewrite (\ref{eq: inductive claim}) as
\eq{eq: proven by induction}{
x_{\alpha^{-k}\beta}=\til g^{-k} x_\beta \quad \text{for all $k \in \Z$}.
}
Let $s \in [0,\alpha)$. Then, we have
\eq{eq: after a bit more rel}{
\begin{array}{rcl}
x_{\alpha^{-k}(\beta+s)} & = & \rel^{(v)}_{\alpha^{-k} s} x_{\alpha^{-k}\beta}=
\rel^{(v)}_{\alpha^{-k} s} \til g^{-k} x_\beta \\
& = &
\til g^{-k} \rel^{(v)}_{s} x_\beta=
\til g^{-k} x_{\beta+s}=\til g^{-k} q_s.
\end{array}}
This shows that $x_t$ is defined for $t>0$. 

%

We will now prove statement (1) of the theorem. Recall that Lemma \ref{lem:affine homeomorphism}
provided us with an affine homeomorphism $\varphi:x_{\alpha^{-1}\beta}
\to x_\beta$ with derivative $\til g$. Recall from \S \ref{subsec:
  rel and suspension} that there is a natural
isotopy class of homeomorphisms between surfaces on the same rel leaf,
obtained by deforming
along the leaf. Abusing notation we will choose a homeomorphism from
this isotopy class and denote it by
$\rel^{(v)}_s$. For all $t>0$ define an isotopy class of 
homeomorphisms $[\varphi_t]$ from $x_{\alpha^{-1}t}$ to $x_t$ by the
composition of maps 
\eq{eq:chain}{x_{\alpha^{-1} t}
  \xrightarrow{\rel^{(v)}_{\alpha^{-1}(\beta-t)}} x_{\alpha^{-1}
    \beta} \xrightarrow{\varphi} x_{\beta}
  \xrightarrow{\rel^{(v)}_{t-\beta}} x_{t}.} 
We will check that there is an affine homeomorphism
$\varphi_t:x_{\alpha^{-1}t} \to x_t$ with derivative $\til g$ in this
isotopy class.  
It suffices to check that 
\eq{eq: holonomy check}{
\hol\big(\varphi_t(\gamma), x_t\big)=\til g \hol(\gamma, x_{\alpha^{-1} t}).}
for any relative cycle $\gamma \in H_1(S, \Sigma; \Z)$. Clearly this holds for an 
absolute cycle since the only thing in the composition which changes
holonomy of an absolute cycle is the action of $\varphi$, which has
derivative $\til g$. Now consider a path $\gamma$ joining $\bullet$ to
$\circ$ in 
$x_{\alpha^{-1} t}$. The action of
$\rel^{(v)}_{\alpha^{-1}(\beta-t)}$ changes its holonomy by adding $\big(0, \alpha^{-1}(\beta-t)\big)$,
then the action of $\varphi$ scales by $\til g$, then the action of
$\rel^{(v)}_{t-\beta}$ changes holonomy by adding $(0, t-\beta)$.  
Thus
$$\begin{array}{rcl}
\hol\big(\varphi_t(\gamma), x_t\big) & = & \til g \Big(\hol(\gamma, x_t) + \big(0, \alpha^{-1}(\beta-t)\big)\Big)+
(0, t-\beta)\\
& = & \til g \hol(\gamma, x_t)+(0,\beta-t)+(0, t-\beta)=\til g \hol(\gamma, x_t).\end{array}$$
This verifies (\ref{eq: holonomy check}), and implies the existence of
an
affine homeomorphism $\varphi_t$ in the class $[\varphi_t]$. Now
consider the mapping class $\varphi_t \circ \rho_t$, where $\rho_t:x_t
\to x_{\alpha^{-1} t}$ is obtained by deforming through the imaginary
rel leaf. We may consider the mapping class of the composition
$\varphi_t \circ \rho_t$  
as determined by deforming $x_t$ through a loop in moduli
space, namely a loop which follows the vertical rel leaf and then returns along
the geodesic given by the $A$-action. Correspondingly, there is an
element $\gamma_t \in \Mod(S, \Sigma)$ obtained by lifting this loop
to $\HH_{\mathrm{m}}$. The map $t \mapsto \gamma_t$ depends
continuously on $t$, but since $\Mod(S, \Sigma)$ is discrete, it must
be constant. By Lemma \ref{lem:affine
  homeomorphism}, in the case $t = \beta$ we find that $\gamma_\beta$ lies in the
renormalization isotopy class of $x_\beta=q_0$, and hence 
$\varphi_t \circ \rho_t$ lies in
the renormalization isotopy class of $x_t$.  
This proves statement (1) of the theorem.


We will now prove statement (2). Consider statement (a) where
$t=\sum_{j=-\infty}^k \alpha^{-j}=\alpha^{k+2} \beta$. By \eqref{eq:
  proven by induction} we have
$x_t=\til g^{k+2} x_\beta$. Based on our presentation of $q_0$, we see
that  
$x_t$ has $\genus$ cylinders whose circumferences
are $\alpha^{k+2}$ times as large as those of $q_0$. Now consider
statement (b). Here 
$t=\alpha^{k+2}(\beta+s)$ for some $s \in (0,\alpha)$. Using (\ref{eq:
  after a bit more rel}), we see that $x_t=\til g^{k+2} q_{s}$. A
cylinder decomposition of $q_s$ is described by Proposition
\ref{prop:complete periodicity} including the dimensions of the
cylinders. The circumferences of $x_t$ must be $\alpha^{k+2}$ times as
large. 
The positions of the singularities can also be determined by
considering $q_s$, see the discussion preceding Proposition \ref{prop:complete periodicity}.
\end{proof}

\subsection{Behavior of cylinders}
\label{sect:cylinders}
We will now investigate the behavior of homology classes of cylinders
as we move through the imaginary rel leaf $\{x_t\}_{t>0}$. We may
identify topological objects associated to the family of surfaces
$\{x_t\}$ using homeomorphisms $\rho$ 
obtained by moving through the rel leaf as in \S \ref{subsec: rel and suspension}.
With a view to the applications considered in \S \ref{subsec:
  deformations etc}, we will assume in this section that 
\eq{eq:assumption}{
t(1-\alpha)  \text{ is not an integral power of $\alpha$,}}
which guarantees that $x_t$ has $\genus+1$ horizontal cylinders; see
Theorem \ref{thm:completely periodic ray}(2). 

For any such $t$, let 
$$C_0(t), C_1(t), \ldots, C_{\genus}(t) \in H_1(S \smallsetminus \Sigma; \Z)$$ 
denote the homology classes of core curves of the horizontal cylinders
of $x_t$ oriented rightward and ordered
from longest to shortest circumference. From Theorem
\ref{thm:completely periodic ray}(1), we know 
\eq{eq:phi action on cylinders}{
\varphi_\ast \big(C_i(\alpha^{-1} t)\big)=C_i(t) \quad \text{for all
  such $t$ and all $i=0,\ldots, \genus$,}} 
where $\varphi: S \to S$ denotes the (common) renormalization mapping
class of the suspensions.  

Recall from (\ref{eq:cohomology class of cylinder}) that if $C_i$ is a cylinder,
$C_i^\ast \in H^1(S,\Sigma; \Z)$ denotes the cohomology class where
$C_i^\ast(\gamma)=\gamma \cap C_i$ for all $\gamma \in H_1(S,\Sigma; \Q)$.
The span of these classes in $H^1(S,\Sigma; \Z)$ will be important in
\S \ref{sec: minimal sets}, when we study twist coordinates for
analyzing the horizontal rel orbits of $x_t$. This span is closely connected to 
vertical holonomy for the family of surfaces $\{x_t\}$ as a linear
map. Let $\Q(\alpha)+\Q t$ denote the $\Q$-vector space of expressions
of the form $a+bt$ where $t$ is a free variable, $a \in \Q(\alpha)$
and $b \in \Q$. 
Then we can consider
the {\em vertical holonomy map of the family} to be 
\eq{eq: holonomy of family}{
\widetilde \hol_{y}: H^1(S,\Sigma; \Q) \to \Q(\alpha)+\Q t; \quad
\gamma \mapsto \hol_y(\gamma; x_t).} 

\begin{lem} \name{lem: span of cylinders}
For any $t_0>0$ satisfying (\ref{eq:assumption}),
the span over all horizontal cylinders $C$ of $x_{t_0}$ of the cohomology classes
$C^\ast \in H^1(S,\Sigma; \Q)$ is the $\genus+1$ dimensional
$\Q$-subspace
\eq{eq: P2}{
\{L \circ \widetilde \hol_{y}~:~ L \text{ is a
  $\Q$-linear map } \Q(\alpha)+\Q t \to \Q\}\subset H^1(S,\Sigma; \Q).} 
\end{lem}
\begin{proof}
Let $C_0^\ast,\ldots, C_{\genus}^\ast$ denote the cohomology classes
to the cylinders of $x_{t_0}$ ordered from longest to shortest
circumference as above.  
Let $V=\spa_\Q \{C_0^\ast, \ldots, C_\genus^\ast\}$,
and let $W$ be the space defined in (\ref{eq: P2}). We will first show that
$W \subset V$. Let $L: \Q(\alpha)+\Q t \to \Q$ be $\Q$-linear. Because
$t_0$ satisfies (\ref{eq:assumption}), 
the cylinders survive for $t$ in a neighborhood of $t_0$.  
For any $t$ in this neighborhood and 
any $\gamma \in H_1(S,\Sigma; \Q)$, we have
$$L \circ \widetilde \hol_{y}(\gamma)=- \sum_{i=0}^\genus L(h_i) C_i^\ast(\gamma),$$
where $h_i \in \Q(\alpha)+\Q t$ denotes the height of $C_i$, which varies linearly with $t$.
This shows $W \subset V$. To see $V = W$ it suffices to show that they
have the same $\Q$-dimension. Observe that $W$ has $\Q$-dimension
$\genus+1$ because its elements 
are in bijection with the choice of $L$ from the dual space to $\Q(\alpha)+\Q t$.
The space $V$ has $\Q$-dimension $\genus+1$ when $t \in
(\beta,\beta+\alpha)$ (corresponding to $q_s$ for $s \in (0,\alpha)$)
by Proposition \ref{prop:complete periodicity}, 
and has this same dimension for $t_0$ by considering that we can find
a $k \in \Z$ so that $\alpha^{-k} t_0 \in (\beta,\beta+\alpha)$ and
applying (\ref{eq:phi action on cylinders}). 
\end{proof}
\section{Recognizing the Arnoux-Yoccoz surfaces}
\name{sec: AY surface}
In \S \ref{sect: rel ray} we produced an imaginary rel ray
$\{x_r\}_{r>0}$ consisting of horizontally completely periodic
surfaces for each genus $\genus \geq 2$. 
In this section we explain that $x_r$ limits on the Arnoux-Yoccoz
surface of genus $\genus$ as $r \to 0$ and consider consequences.

\subsection{The Arnoux-Yoccoz surfaces}\name{subsec: AY surfaces
  general genus}
Here we define the Arnoux-Yoccoz surfaces in genus $\genus \geq 3$. 
We will follow Arnoux and Yoccoz, who describe this family of surfaces as the
canonical suspension of the interval  exchange $\IE$ of \S
\ref{sect:renormalization}. This relies on work of Levitt, and on
Thurston's classification of surface homeomorphisms
\cite{Thurston, FLP}.   

We recall a construction of \S \ref{sect: suspensions}: if
$(q, \gamma)$ 
is a canonically labeled suspension of $\IE$ with $q \in \HH(\genus-1,
\genus-1)$, then so is $(q, \hat \gamma)$, and there is a
corresponding renormalization mapping class $[\varphi]$ where $\varphi: q \to q$. 
See Corollary
\ref{cor:second suspension} and Definition \ref{def: renormalization
  mapping class}. We need the following:  
\begin{lem}
\name{lem:mapping class 1}
Let $\genus \geq 3$, let $(q,\gamma)$ be a suspension of $\IE$ with
$q \in \HH(\genus-1,\genus-1)$, and let $[\varphi]$ be the
renormalization mapping class.  
Then $[\varphi]$ is pseudo-Anosov.
Furthermore if $q$ is a measure preserving suspension then the action of 
$[\varphi]$ scales the vertical measured foliation of $q$ by
$\alpha^{-1}$. 
\end{lem}

\ignore{
To prove $[\varphi]$ is pseudo-Anosov it suffices to show that it acts neither periodically
nor reducibly (periodically on an isotopy class of simple closed
curves) \cite[Theorem 9.16]{FLP}. This was argued 
in \cite{AY}; see in particular Lemma 2. The statement involving the vertical measured foliation
of a measure preserving suspension follows
from the definition of 
$\psi$ in Theorem \ref{thm:renormalization} and its relation to the
renormalization mapping class. We leave a formal proof to the reader.
}

We include the argument here for completeness and to indicate the differences between
the case of $\genus \geq 3$ and $\genus=2$ which is handled in the next subsection.
\begin{proof}
It suffices to take $(q, \gamma)$ to be a measure preserving suspension of $\IE$. 
Let $\varphi:q \to q$ be a homeomorphism which satisfies
$\gamma=\varphi \circ \hat \gamma$ and which carries vertical leaves
to vertical leaves. By definition $\varphi$ is in the renormalization
mapping class.  

Let $({\mathcal F}^{(v)},\mu)$ be the vertical foliation on $q$ together with its Lebesgue transverse measure.
Observe that $\varphi({\mathcal F}^{(v)},\mu)=({\mathcal F}^{(v)},\alpha^{-1} \mu)$.
This is because $\psi$ (of Theorem \ref{thm:renormalization}, used in
definition of $\hat \gamma$) scales measure uniformly by a factor of
$\alpha^{-1}$.  
This measured foliation has a representative homology class in
absolute homology $H_1(q; \R)$, and thus the action  
$\varphi_\ast:H_1(q; \R) \to H_1(q; \R)$ has an eigenvalue of
$\alpha^{-1}$. We will see that this allows us to classify $[\varphi]$ as pseudo-Anosov
in the sense of Thurston's classification of surface homeomorphisms
\cite[Theorem 9.16]{FLP}.

According to 
Thurston's classification of surface homeomorphisms, 
$[\varphi]$ is either periodic, reducible (acts periodically on an
isotopy class of simple closed curves), or pseudo-Anosov.  

First observe $[\varphi]$ can not be periodic, since the eigenvalue
$\alpha^{-1}$ of $\varphi_\ast$ is not a root of unity.  

Now suppose $[\varphi]$ was reducible. We will replicate the argument of \cite[Lemma 2]{AY}.
By definition of reducibility, a positive power $[\varphi^k]$ leaves
invariant a subsurface $S \subset q$ of lower genus.  
The induced action on absolute homology of the subsurface has an eigenvalue of $\alpha^{-k}$, since it preserves the intersection of the vertical measured foliation with the subsurface.
We note that because $\alpha^{-1}$ is Pisot, $\alpha^{-k}$ also must have degree $\genus$.
Observe that the rank of $H_1(S;\R)$ is twice the genus of $S$, which is strictly less than $2g$.
Since the action $\varphi^k_\ast|_S$ of $[\varphi^k]$ on $H_1(S;\R)$ is by a integer matrix, we see that $\alpha^{-k}$ and all its $\genus$ algebraic conjugates must be eigenvalues. Since
this action is symplectic, $\alpha^{k}$ and all its $\genus$ conjugates must appear as eigenvalues
of $\varphi^k_\ast|_S$ as well. Because the rank of $H_1(S;\R)$ is
less than $2g$, some of these values must coincide, meaning 
that $\alpha^{-k}$ is conjugate to $\alpha^k$. Observe that $\alpha^{-k}$ is Pisot,
and the product of all its conjugates is $(-1)^k \alpha^k$.
When the degree $\genus \geq 3$, this coincidence is impossible. This
contradicts the irreducibility of $[\varphi]$ and shows that
$[\varphi]$ must be pseudo-Anosov. (This argument fails when
$\genus=2$, because $\alpha^{-2}$ is conjugate to $\alpha^2$.)  
\end{proof}

We will now use Lemma \ref{lem:mapping class 1} to construct the
Arnoux-Yoccoz surface when $\genus  \geq 3$. 
Fix a measure preserving suspension $(q,\gamma)$ and let 
$\varphi$ be a representative of the renormalization mapping class as
in the Lemma. Let 
$({\mathcal F}^{(v)},\mu)$ denote the vertical measured foliation on
$q$, so that $\varphi$ preserves $\mathcal{F}^{(v)}$, scaling $\mu$ by
$\alpha^{-1}$. 
Let $S$ be the topological
surface underlying the translation surface $q$. We identify $q$ with
$S$ topologically.
Since $\varphi$ is pseudo-Anosov it
leaves invariant a unique second projective class of measured
foliations on $S$, say the class 
$$\{({\mathcal F}^{(h)}, c \nu):~c>0\}.$$  
Integrating the pair of measures $(\mu,c \nu)$ defines a new
translation structure on $S$. That is, in the charts of this alternate
translation structure, the transverse
measures $\mu, c\nu$ correspond respectively to the 1-forms
$\hol_{\mathrm{x}}, \hol_{\mathrm{y}}.$ 
The area of this translation surface varies linearly with $c$, and we
define the Arnoux-Yoccoz surface in genus $\genus$ 
to be the surface obtained by choosing $c$ so that the area
is 
\eq{eq:A}{
\alpha^\genus \cdot
\alpha+\alpha^{\genus-1}\cdot (\alpha+\alpha^2)+\alpha^{\genus-2}\cdot (\alpha+\alpha^2+\alpha^3)+\cdots+\alpha \cdot 1.} We
denote this surface by $x_{\mathrm{AY}}$.
We choose this normalization of the area in order to be consistent with \cite{AY}
and \cite{Bowman}. Note that the surfaces $q_s$ and $x_t$ of \S \ref{sect: rel rays} have this area.

The two translation surfaces $q$ and $x_{\mathrm{AY}}$ are both marked by the
same underlying surface $S$, and the map $q \to x_{\mathrm{AY}}$ maps vertical
leaves to vertical leaves, preserving the transverse measure
$\hol_{\mathrm{x}}$. Using this topological identification, the curve
$\gamma$ on $q$ is mapped to a curve on $x_{\mathrm{AY}}$ which we also denote
by $\gamma$, and we see that $(x_{\mathrm{AY}}, \gamma)$ is also a canonically
labeled suspension of $\IE$.  
Since $[\varphi]$ scales $({\mathcal F}^{(v)},\mu)$ and $({\mathcal F}^{(h)}, \nu)$,
there is a canonical affine automorphism with diagonal derivative,
$\varphi_{\mathrm{AY}}:x_{\mathrm{AY}} \to x_{\mathrm{AY}}$, which is taken from the mapping
class $[\varphi]$. Since $\varphi_{\mathrm{AY}}$ scales $({\mathcal
  F}^{(v)},\mu)$ by $\alpha^{-1}$ and preserves area, it must scale
$({\mathcal F}^{(h)},\nu)$ by $\alpha$. Thus, $D \varphi_{\mathrm{AY}}=\til g$
with $\til g$ as in (\ref{eq: crucial property}). Also note that because
$\varphi_{\mathrm{AY}}$ is taken from the mapping class $[\varphi]$, it is in
the renormalization isotopy class of $x_{\mathrm{AY}}$.

As a consequence of these observations we find:
\begin{cor}\name{cor: above}
If $(q,\gamma)$ is a suspension of $\IE$ with $q \in
\HH(\genus-1,\genus-1)$, with $\genus \geq 3$, then $q$ has no vertical  saddle connections.
\end{cor}
\begin{proof}
Observe that $x_{\mathrm{AY}}$ has no vertical saddle connections because it
has an affine automorphism with derivative $\til g$. As we have seen, 
there is a homeomorphism $q \to
x_{\mathrm{AY}}$ which carries vertical leaves to vertical leaves. So $q$ also
does not have any vertical saddle connections. 
\end{proof}

Now recall the vertical rel leaf considered in \S \ref{sect: rel
  rays}. We observe that $q_0=x_\beta \in  \HH(\genus-1,\genus-1)$ is
a measure preserving suspension of $\IE$ with area given by \eqref{eq:A}. By 
Corollary \ref{cor: above}, $q_0$ has no vertical saddle connections. Therefore by
Proposition  \ref{prop: real rel main}, the imaginary rel leaf $\{x_t\}$ is defined for all $t$. 

We will show that $x_0=x_{\mathrm{AY}}$.  
Recall from Theorem \ref{thm:completely periodic ray}
that for $t>0$ we have an affine homeomorphism
$\varphi_t:x_{\alpha^{-1} t} \to x_t$ 
with derivative $\til g$. Furthermore, if $\rho_t:x_t \to
x_{\alpha^{-1} t}$ is obtained by imaginary rel  deformation, then
$\varphi_t \circ \rho_t$ is in the renormalization isotopy class. Now
let $t \to 0^+$. Since $\rho_t$ is the map induced by
$\rel^{(v)}_{\alpha^{-1}t-t}$, the limit of $\rho_t$ is the identity
map on $x_0$. The affine
homeomorphisms converge 
to an affine automorphism $\varphi_0:x_0 \to x_0$ with derivative
$\til g$. Furthermore by discreteness of the mapping class group, we know that $[\varphi_0]$ is in the renormalization isotopy class of $x_0$. Because the derivative of $\varphi_0$ is $\til g$, the mapping class preserves the horizontal and vertical foliations of $x_0$. 
Up to the canonical identification between suspensions, the mapping
classes $[\varphi_0]$ of $x_0$ and  $[\varphi_{\mathrm{AY}}]$ of $x_{\mathrm{AY}}$ are
the same. By the uniqueness of the expanding and contracting measured
foliations, we see that $x_0$ and $x_{\mathrm{AY}}$ have the same horizontal
and vertical measured foliations. Therefore, these two surfaces
represent the same translation surface. We have shown: 

\begin{thm}\name{thm: identifying AY}
Let $\genus \geq 3$. Let $\{x_t:t>0\}$ be the imaginary rel ray
discussed in Theorem \ref{thm:completely  periodic ray}. Then this rel
ray extends to an imaginary rel line $\{x_t: t\in \R\}$ and the
surface $x_0$ is the Arnoux-Yoccoz surface of genus $\genus$.  
\end{thm}

Figure \ref{fig:AY} shows the Arnoux-Yoccoz surface $x_{\mathrm{AY}}=x_0$ of genus
$\genus=3$. 
The surfaces of higher genus have similar presentations (see \cite{Bowman}).

\begin{proof}[Proof of Theorem \ref{thm: real rel}]
In Theorem \ref{thm:completely periodic ray}, for each $\genus \geq 2$
we exhibited an imaginary rel-ray $\{x_t : t>0\}$ consisting of
horizontally periodic surfaces. Since $\alpha <1$, statement (2) of
Theorem \ref{thm:completely periodic ray} shows that as $t \to
+\infty$, the circumferences of the horizontal cylinders on $x_t$ tend
uniformly to zero. Theorem
\ref{thm: identifying AY} shows that for $\genus \geq 3$, $x_t = \rel^{(v)}_t x_0$ for
$t>0$, concluding the proof.  
\end{proof}

\subsection{The noded surface in genus two}
\name{sect: genus two}
When $\genus=2$, there is also a ``canonical suspension'' of $\IE$ as a pair of tori attached at a node.
These observations on the genus $2$ version of the Arnoux-Yoccoz
surface are due to Bowman 
\cite[Appendix]{Bowman},  so we will call this surface the {\em
  Arnoux-Yoccoz-Bowman} noded surface. 
We construct this surface below and then connect it to the imaginary rel
ray in $\HH(1,1)$ described in \S \ref{sect: rel rays}. 

\begin{lem}~
\name{lem:mapping class 2}
Suppose $\genus=2$ and $(q, \gamma)$ is a suspension of $\IE$ with $q \in \HH(1,1)$. 
Then there are two vertical saddle connections on $q$ each of which
pass through one of the images under $\gamma$ of the midpoints of the
intervals $J_1$ and $J_2$ of (\ref{eq:intervals}). These saddle
connections  separate $q$ into two tori (glued along a slit), the
action of $[\varphi^2]$ swaps the saddle connections and the tori. The
action of $[\varphi^2]$ restricted to either torus is pseudo-Anosov. 
\end{lem}
\begin{proof}
First consider the suspension of $\IE$ shown in the
right hand side of Figure \ref{fig:suspension}, where we take $\gamma$
to be a unit speed parametrization of a straightline of slope
$\frac{1}{2}$ passing through the cylinder represented by the large
square. 
The surface can be described as two tori glued along a vertical slit
to form a genus $2$ surface. The
reglued slits form the two vertical saddle connections on the surface
(shown as dotted lines in the Figure \ref{fig:suspension}). 
The suspension is a measure preserving suspension, and the
midpoints of the bottom edges (of slope $\frac{1}{2}$) of the smaller squares lie on the vertical leaves passing through the images under $\gamma$ of the midpoints of $J_1$  and
$J_2$. 
We have verified the first assertion of the lemma holds for one choice of $(q,\gamma)$ 
with $q \in \HH(1,1)$. This assertion follows for any choice, because
the canonical homeomorphism between two 
suspensions sends saddle  connections to saddle connections. See \S \ref{sect:IETs}.

It can be observed that the saddle connections and hence the tori
are swapped by the renormalizing  homeomorphism $\varphi:q \to q$.  
Fix one of the tori. To see that the restriction of $\varphi^2$ to
this torus is pseudo-Anosov we argue as follows. 
The action of $\varphi^2$ on each torus expands the vertical measured
foliation by the quadratic irrational $\alpha^{-2}$. This is not a
root of unity so the action is not periodic. 
Since first absolute homology is $2$-dimensional, the action cannot
be reducible. Thus $\varphi^2$ restricted to the torus is
pseudo-Anosov. 
\end{proof}

We will now construct $x_{\mathrm{AYB}}$. Let $(q, \gamma)$ be a measure preserving suspension 
of $\IE$ with $q \in \HH(1,1)$. Let $S$ be the topological surface
underlying the translation structure on $q$.  
Let $({\mathcal F}^{(v)}, \mu)$ be the vertical measured foliation. By the above result, there are
vertical saddle connections in $q$ which 
separate $q$ into two tori. The action of $[\varphi^2]$ on either of
the tori is pseudo-Anosov. The action of $[\varphi^2]$ scales $({\mathcal F}^{(v)}, \mu)$
by $\alpha^{-2}$. Because the action $[\varphi^2]$ on each torus is
pseudo-Anosov, there is a projective class of measured foliations  
on each torus which is contracted by a factor of $\alpha^2$. 
Choose the measured foliations $({\mathcal F}_1^{(h)},\nu_1)$
 and $({\mathcal F}_2^{(h)},\nu_2)$
from these classes so that the areas of the tori as measured by $\mu
\times \nu_i$ is half of the value in (\ref{eq:A}). 
Then the action of $[\varphi]$ satisfies
$$\varphi_\ast({\mathcal F}^{(h)}_1,\nu_1)=({\mathcal F}^{(h)}_2,\alpha \nu_2)
\quad \text{and} \quad \varphi_\ast({\mathcal F}^{(h)}_2,\nu_2)=({\mathcal
  F}^{(h)}_1,\alpha \nu_1).$$ 
Let $\sigma$ be one of the vertical saddle connections on $q$ in the
common boundary of the two tori. Then for each $i$,
$\varphi^2(\sigma)=\sigma$ and $\varphi^2_\ast(\nu_i)(\sigma)=\alpha^2
\nu_i(\sigma)$, 
so $\nu_i(\sigma)=0$. Let 
$$({\mathcal F}^{(h)},\nu)=({\mathcal F}^{(h)}_1\cup {\mathcal F}^{(h)}_2,\nu_1+\nu_2).$$
Then the structure on $q$ induced by the product of this horizontal transverse measure
with $({\mathcal F}^{(v)}, \mu)$ gives $S$ a noded translation structure,
which we call the {\em Arnoux-Yoccoz-Bowman noded surface} $x_{\mathrm{AYB}}$.
It is noded in the sense that the simple closed curve consisting of
the union of the two vertical saddle connections has been collapsed to
a point (it gets no transverse measure with respect to either $\mu$ or
$\nu$). 

We now summarize some results regarding a bordification of
$\HH(1,1)$. See \cite[\S 5]{Bainbridge} for more details, and proofs
of the following facts. 
The stratum $\HH(1,1)$ can be viewed as a subset of the bundle of holomorphic 1-forms $(X, \omega)$
where $X$ is a Riemann surface of genus 2 (because of our conventions
that singularities are labeled, in order to make this
precise, one should pass to a twofold cover of this bundle). Deligne
and Mumford defined a compactification of the
moduli space $\mathcal{M}_2$ of genus 2 surfaces by adjoining moduli
spaces of stable curves as boundary components of
$\mathcal{M}_2$, and the bundle of holomorphic 1-forms on $\mathcal{M}_2$ can be
extended to a bundle of `stable forms' over the
Deligne-Mumford compactification. An example of a stable curve
in the boundary of $\mathcal{M}_2$ is two 
tori attached at a node, and a stable form over
such a surface is a  pair of nonzero holomorphic
1-forms on a torus with one marked point. Thus, in our notation, the
union 
$$
\HH(1,1) \cup \left[ (\HH(0) \sm \{0\}) \times (\HH(0) \sm \{0\})
\right ]
$$
inherits a topology from the bundle of stable 1-forms over the
Deligne-Mumford compactification of $\mathcal{M}_2$. This topology is
compatible with polygonal presentations of translation surfaces, in
the following sense. Suppose $\{x_t: t>0\}$ is a collection of translation surfaces in
$\HH(1,1)$ which are a connected sum of two tori glued
along a slit, such that the length of the slit
goes to zero as $t \to 0+$, and the polygonal representation of  each
of the connected components of the complement of 
the slit has a Hausdorff limit of positive area. Then $\lim_{t \to
  0+} x_t$ exists as a stable 1-form on two tori attached at a
node, and is represented in $(\HH(0) \sm \{0\}) \times (\HH(0) \sm \{0\})$
 by the Hausdorff limit of the polygonal presentations of the $x_t$. 

Using this terminology we have:
\begin{thm}\name{thm:
  this terminology}
Let $\genus = 2$. Let $\{x_t:t>0\}$ be the imaginary rel ray discussed
in Theorem \ref{thm:completely  periodic ray}
and defined by $x_t=\rel^{(v)}_{t-\beta} q_0$. Then $x_{\mathrm{AYB}}=\lim_{t \to
  0+} x_t$. 
\end{thm}
\begin{proof}
We fix $\genus =2$ and define $\alpha, \beta, q_0$ via
\eqref{eq: alpha}, \eqref{eq: beta}, and \S\ref{sect: rel rays} respectively. 
We get $\alpha=\frac{\sqrt{5}-1}{2}$ and $\beta=1$, and 
the surface $q_0=x_1$ 
is shown in Figure \ref{fig: genus 2 step 0}.
The surface breaks into two tori (depicted as white and gray regions)
along vertical saddle connections of length $1$. 
We change the presentation by cutting along the dotted lines and regluing using edge identifications.
The resulting surface is shown on the left side of Figure \ref{fig:
  genus 2 more}. The two regions shaded white and gray are two tori
glued along a slit, and as we apply $\rel^{(v)}_{t}$ and let $t \to
-1$, the length of the slit goes to zero. By the preceding discussion we can apply
$\rel^{(v)}_{-1}$ to this presentation and obtain the limit of the
surfaces $x_t$ as $t \to -1$. This collapses the vertical saddle connections
resulting in the noded surface shown on the right side of Figure
\ref{fig: genus 2 more}. A curve $\gamma$ representing the measure
preserving suspension of $\IE$ is shown as a dotted line in the
Figure. (This curve is in the same homotopy class as the core curve of
the cylinder $C_0$ in the description of $q_0$ provided in \S
\ref{sect: rel rays}.) 
To see that $x_0$ is $x_{\mathrm{AYB}}$, it suffices to check that the areas
match and that the renormalization mapping  class of $(x_0,\gamma)$ is
realized by an affine automorphism with derivative $\til g$. 
We leave these details to the reader.
\end{proof}

\begin{figure}
  \hfill
  \begin{minipage}[c]{0.65\textwidth}
   \vspace{0pt}\raggedright
   \includegraphics[scale=0.60]{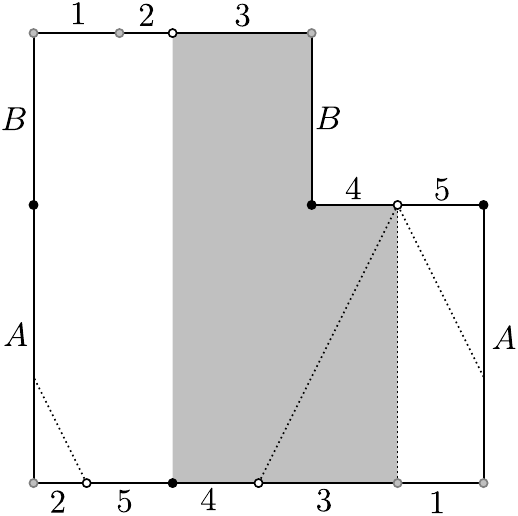}
  \end{minipage}
   \hfill
  \begin{minipage}[c]{0.3\textwidth}
\vspace{0pt}\raggedright
\begin{tabular}{ll}
Label & Edge length \\
\hline
\vspace{-1em} \\
$1$, $4$, $5$ & $\frac{3-\sqrt{5}}{4}$ \\
$2$ & $\frac{\sqrt{5}-2}{2}$ \\
$3$ & $\frac{\sqrt{5}-1}{4}$ \\
$A$ & $\frac{\sqrt{5}-1}{2}$ \\
$B$ & $\frac{3-\sqrt{5}}{2}$ \\
\end{tabular}  
  \end{minipage}
  \caption{The surface $q_0$ with $\genus=2$ as presented in \S \ref{sect: rel rays}.}
  \name{fig: genus 2 step 0}
  \end{figure}

\begin{figure}
  \begin{center}
  \begin{minipage}{0.495\textwidth}
  \begin{center}
   \includegraphics[scale=0.49]{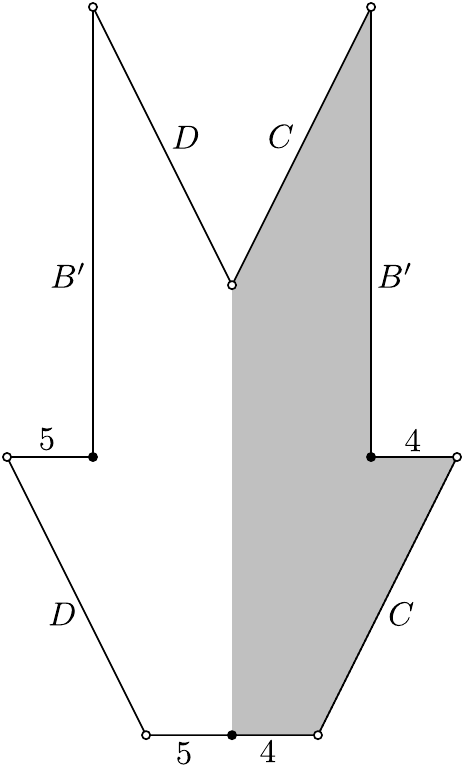}
   \end{center}
  \end{minipage}
  \begin{minipage}{0.495\textwidth}
   \begin{center}
   \includegraphics[scale=0.50]{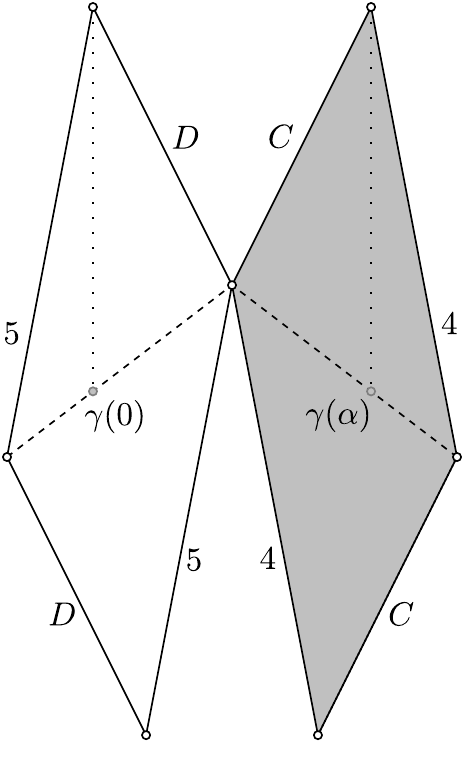}
   \end{center}
  \end{minipage}
  \end{center}
  \caption{Left: The surface $q_0$ with an alternate presentation. Right: The noded surface $x_0=\rel^{(v)}_{-1} q_0$
  is $x_{\mathrm{AYB}}$. 
  Edges labeled $C$ and $D$ have holonomy $(\pm \frac{\sqrt{5}-1}{4},\frac{\sqrt{5}-1}{2})$
  and edges on the right labeled $4$ and $5$ have holonomy $(\pm \frac{3-\sqrt{5}}{4},1)$.}
  \name{fig: genus 2 more}
  \end{figure}

\begin{proof}[Proof of Theorem \ref{thm: g=2}]
Statements (i) and (ii) follow from Theorem \ref{thm:completely
  periodic ray}, and statement (iii) follows from Theorem \ref{thm:
  this terminology}. 
Statement (iv) was proved by McMullen in
\cite[Theorem 9.3]{McMullen-foliations}. 
\end{proof}

\subsection{Genus three}
\name{sect: genus three}
In this subsection we fix $\genus=3$. Let $x_0$ be the Arnoux-Yoccoz surface of genus $3$.
This surface admits some special symmetries:
\begin{itemize}
\item The surface $x_0$ is hyperelliptic, and the hyperelliptic
  involution preserves the surface's singularities \cite[Lemma 2.1]{HLM} \cite[\S
  1.2]{Bowman10}. 
\item The surface $x_0$ admits an orientation-reversing involution
  preserving singularities with derivative  
$$h=\left[\begin{array}{rr}
0 & \alpha \\
\alpha^{-1} & 0 
\end{array}\right].$$
(This involution was described as $\rho'_1$ acting on an affine image
of $x_{0}$ in \cite[\S 2.3]{Bowman10}.) 
\end{itemize}

\begin{proof}[Proof of Corollary \ref{cor: real rel}]
The involutions above preserve the singularities and thus interact
with the rel deformation as described by Proposition \ref{prop: rel
  and G commutation}. In particular, 
for all $t \in \R$, 
$$-I \big ( \rel^{(v)}_t x_0\big)=\rel^{(v)}_{-t} x_0
\quad \text{and} \quad
h \big( \rel^{(v)}_t x_0\big)=\rel^{(h)}_{\alpha t} x_0.$$
In particular, the group generated by the matrices $-I$ and $h$ acts
transitively on the four horizontal and vertical rel rays leaving
$x_0$. The conclusion of the Corollary follows. 
\end{proof}

These symmetries are a special feature of the case $\genus =3$: 

\begin{prop}\name{prop: no involutions}
For $\genus \geq 2$, the surface $x_0$ admits no nontrivial
translation automorphisms. For $\genus \geq 4$, it admits no 
involutions with derivative
$-\mathrm{Id}$, and no orientation-reversing involutions with
derivative $(x,y) \mapsto (x, -y)$ fixing the singular points. 

\end{prop}

\begin{proof}
The proof relies on work of Bowman \cite{Bowman}, and we will freely
use Bowman's notation and results. To see that the
Arnoux-Yoccoz surfaces have no translation automorphisms, note that
there is no saddle connection parallel to $P_{\genus -1}Q_{\genus -1}$
and of equal length. 
The proof of \cite[Cor. 1.6]{Bowman} shows that for $\genus \geq 4$,
the surface has no involutions
with derivative $-\mathrm{Id}$. On the other hand it does admit an
involution swapping singular points with derivative $(x, y) \mapsto
(x, -y)$. If it had a second involution with the same derivative,
fixing singular points, the composition of these two maps would be a
nontrivial translation automorphism. 
\end{proof}

We will now explain Figure \ref{fig: sequence}, which describes the
algebraic dynamics of the rel-deformed Arnoux-Yoccoz IET. Let
$x'_r=\rel^{(h)}_r x_0$. For $0 \leq r < \frac{\alpha^3}{2}$, the
surface $x'_r$ can be considered to be the surface obtained from the
presentation of $x_0$ in Figure \ref{fig:AY} by moving the white
singularity rightward by $r$. As such $x'_r$ is a suspension of an
interval exchange $\IE_r$ defined using the same permutation as the
Arnoux-Yoccoz IET viewed as an interval exchange on $[0,1)$, i.e. the permutation
$$\left(\begin{array}{rrrrrrr}
1 & 2 & 3 & 4 & 5 & 6 & 7 \\
2 & 5 & 4 & 7 & 6 & 3 & 1 
\end{array}\right)$$
and with the vector of lengths
$$\left(\frac{1-\alpha}{2},\alpha-\frac{1}{2}+r,\frac{a}{2}-r,\frac{\alpha^2}{2}+r,
  \frac{\alpha^2}{2}-r, \frac{\alpha^3}{2}+r,
  \frac{\alpha^3}{2}-r\right).$$ 

Observe that for any such $r$ and any point $p \in [0,1)$, the difference
$\IE_r(p)-p$ takes one of the following six values modulo one:
$$\pm \frac{1 - \alpha }{2}, \quad \pm \frac{1 - \alpha^2}{2}, \quad \pm \frac{1 - \alpha^3}{2}.$$
Let $H$ be the group of rotations of $\TTT=\R/\Z$ generated by these six
numbers, and let $\Gamma$ denote the Cayley graph of $H$ with respect
to these six
generators. As a graph $\Gamma$ is isomorphic to the graph
obtained from the edges of the tiling of the plane by equilateral triangles. A 
choice of $p \in [0,1)$ and $r$ as above determines an embedded curve
in $\Gamma$, i.e., a sequence of elements $h_n \in H$ so that 
$$\IE_r^n(p)=p + h_n \mod \Z,$$
where each $h_n$ is viewed as a rotation of $\TTT$. We call $\{h_n\}$ the {\em arithmetic orbit} of $p$.  

Since each $x'_r$ for $r>0$ has a vertical cylinder decomposition, the
corresponding interval exchanges $\IE_r$ are periodic. Based on
Theorem \ref{thm:completely periodic ray}, the periods decrease at
discrete times as $r$ increases. Through an understanding of $x_r'$
when $r$ is close to but slightly smaller than $\frac{\alpha^3}{2}$, it can
be observed that the simplest such periodic orbit has period three,
visiting the sequence of intervals with indices $1$ then $6$ then
$4$. (This periodic orbit is realized by starting near the right
endpoint of interval $1$.) Call the cyclically ordered list $164$ 
the {\em orbit type} of this periodic orbit. Every point of $[0,1)$
has such an {\em orbit type} under $\IE_r$ for $r$ as above: it is the
cyclic sequence of interval indices hit by the orbit. For each $p$ and each $r$
for which the orbit $\IE_r^\Z(p)$ is defined, there is an open set of
values $(p',r') \in \R^2$ which give rise to the same orbit
type. The orbit-type 
determines the arithmetic orbit, because
the translations applied to each interval do not vary with
$r$. The
arithmetic 
orbit corresponding to this simplest periodic orbit is a triangle,
depicted in the top left of Figure \ref{fig: sequence}. 

It follows from the discussion above and  Theorem \ref{thm:completely
  periodic ray}(1) that $\til g x'_{\alpha r}=x'_{r}$ for all $r$. Identifying
surfaces in the real-rel leaf through 
rel deformations, we see that
the affine homeomorphism $x'_{\alpha r} \to x'_{r}$ with derivative $\til g$
is in the renormalization mapping class $[\varphi]$ of $x_0$. Then
the image under $\varphi^{-1}$ of a (homotopy class of a) cylinder $C$
in $x_r'$ yields a (homotopy class of a) cylinder $C'$ in $x_{\alpha
  r}'$. Restrict to the case when $0<r<\frac{\alpha^3}{2}$. Then a
cylinder $C$ in $x_r'$ can be associated to an orbit type $O$ of
$\IE_r$, and the cylinder $C'$ in $x_{\alpha r}'$ is associated to a
different orbit type $O'$. 
Using a natural Markov partition for the action of $\varphi^{-1}$ on
$x_0$, it can be observed that $O'$ can be derived from $O$ by
applying the following ``substitution'' (which makes sense for cyclic
words): 
\eq{eq:substitution}{
\begin{array}{c}
1 \mapsto 34, 2 \mapsto 3 4, 4 \mapsto 1 6, 5 \mapsto 1 7, 6 \mapsto 2, 7 \mapsto 3 \\
3 \mapsto \begin{cases}
3 5 & \text{if the number preceding $3$ was $4$ or $7$},\\
1 5 & \text{if the number preceding $3$ was $3$ or $6$}.
\end{cases}\end{array}}
It follows from Theorem \ref{thm:completely periodic ray}
that the orbit of $164$ under this substitution of gives all orbit types
of periodic orbits of IETs in the family $\{\IE_r: 0 \leq r <
\frac{\alpha^3}{2}\}$. 
So the simplest orbit types of such periodic orbits are
$$164, ~ 34216, ~ 151634342 ~ \text{and} ~ 34173421516351634
$$
corresponding to the four arithmetic orbits depicted on the top left side of Figure \ref{fig: sequence}.

The substitution of (\ref{eq:substitution}) has the Tribonacci substitution
$$a \mapsto ab; \quad b \mapsto ac; \quad c \mapsto a$$
as a factor obtained by mapping $1$, $2$ and $3$ to $a$, mapping $4$ and $5$ to $b$
and mapping $6$ and $7$ to $c$. We speculate that after passing to a
subsequence, the arithmetic orbits (as depicted in Figure \ref{fig:
  sequence}) converge up to rescaling and a uniform affine coordinate
change to Rauzy's fractal \cite{Rauzy} in the Hausdorff topology. As
curves, these rescaled paths likely converge to the curve defined by
Arnoux's semi-conjugation \cite{Arnoux88} between the Arnoux-Yoccoz IET and the toral
rotation associated to Rauzy's fractal. We hope to investigate these
relations in future work.

\section{Results of Eskin-Mirzakhani-Mohammadi, 
  consequences and subsequent developments}
\name{subsec: breakthroughs}
Recent breakthrough results of Eskin, Mirzakhani and Mohammadi
\cite{EM, EMM2} give a wealth of information about orbit-closures for
the actions of $G$ and $P$ on strata of translation surfaces. The
following summarizes the results which we will need in this paper:
\begin{thm}\name{prop: EMM}
Let $x$ be a translation surface in a stratum $\HH$, and let $\HH_0$
be the subset of the connected component of $\HH$ containing $x$, and
consisting of surfaces whose area is the same 
area as the area of $x$. Then 
$$
\overline{Gx} = \overline{Px} = \MM \cap \HH_0,
$$
where $\MM$ is an immersed submanifold of $\HH$ of even (real)
dimension which is cut out by linear equations 
with respect to period coordinates, and $\MM \cap \HH_0$ is the support of a
finite smooth invariant measure $\mu$. Moreover 
\eq{eq: P action}{
\frac{1}{T} \int_0^T \int_0^1 (g_tu_s)_* \delta_x \, ds \, dt
\to_{T \to \infty} \mu,
}
where $\delta_x$ is the Dirac measure at $x$ and the convergence is weak-* convergence in
the space of probability measures on $\HH$. 
\end{thm}

\subsection{Surfaces whose horocycle orbit-closure contains their \texorpdfstring{$G$}{G}-orbit}
\label{sect:horocyle}
Recall that Proposition \ref{prop: V suffices} claimed that if $x$ periodic under $g_t$, then $\overline{Ux} = \overline{Vx}= \overline{Gx}$.

\begin{proof}[Proof of Proposition \ref{prop: V suffices}]
We prove for $U$, the proof for $V$ being similar. 
Suppose $g_{p_0}x=x,$ where $p_0>0$ is the period for the closed
geodesic. 
Let $\mu$ be the smooth
$G$-invariant measure with $\overline{Gx} = \supp \, \mu$. By \equ{eq:
  P action},  
\eq{eq: lhs}{
\begin{split}
\mu & = 
\lim_{m \to \infty}
\frac{1}{mp_0} \int_0^{mp_0} \int_0^1 (g_tu_s)_* \delta_x \, ds \, dt \\
 & = \lim_{m \to \infty} \frac{1}{mp_0} \int_0^{p_0} 
 \sum_{i=0}^{m-1} \int_0^1 (g_{ip_0+p}u_s)_* \delta_{x} \, ds \, dp. 
\end{split}
}
For each $p \in [0, p_0)$ we write 
\eq{eq: rhs}{
\nu_{p, m} = \frac{1}{m}\sum_{i=0}^{m-1} 
\int_0^1 (g_{ip_0+p} u_s)_* \delta_x \, ds = \frac{1}{m} \sum_{i=0}^{m-1}
\frac{1}{e^{2(ip_0+p)}} \int_0^{e^{2(ip_0+p)}} (u_s g_{p})_*\delta_x
\, 
ds,}
where we have used the commutation relations $g_\tau u_s =
u_{e^{2\tau}s} g_\tau$ and the fact that $g_{ip_0}x =x $ for each $i$.  
Then the right hand-side of \equ{eq: lhs} is $\lim_{m \to \infty} \frac{1}{p_0} \int_0^{p_0}
\nu_{p, m} dp$. Let $\nu_{0,m}$ be the measure corresponding to
$p =0$, then for any $p$, $\nu_{p, m} = g_{p*}\nu_{0,m}$. Take
a subsequence $\{m_j\}$ along which $\nu_{0,m_j}$ converges to a measure
$\nu$ on $\HH$ (where $\nu(\HH) \leq 1$). Then $\nu_{p, m_j} \to_{j \to \infty} g_{p*}
\nu$. The right hand side of \equ{eq: rhs} shows that $\nu$ is
$U$-invariant and therefore so is each $g_{p*} \nu$, and by \equ{eq:
  lhs} we have $\mu = \frac{1}{p_0} \int_0^{p_0} g_{p*} \nu \, dp$
(and in particular $\nu(\HH)=1$). Since $\mu$ is
$G$-ergodic, by the Mautner property (see e.g. \cite{EW}) it is $U$-ergodic. This implies
that $g_{p*}\nu = \mu$ for almost every $p$, and (since $\mu$ is
$\{g_t\}$-invariant), $\nu=\mu$. These considerations are valid for
every convergent subsequence of the sequence $\nu_{0,m}$ and hence
$\nu_{0,m} \to_{m \to \infty} \mu$. Since $\nu_{0,m}$ is obtained by
averaging over the $U$-orbit of $x$,  the orbit $Ux$ is dense in $\supp \, \mu =
\overline{Gx}$, i.e. $\overline{Ux} = \overline{Gx}$.  
\end{proof} 

\subsection{Results of Hubert-Lanneau-M\"oller, and their extension}
\name{subsec: full rank}
In \cite{HLM}, Hubert, Lanneau and M\"oller determined
$\overline{Gx_0}$ for the Arnoux-Yoccoz surface $x_0$ in the
case $\genus =3$. The proof of \cite{HLM} preceded the recent
breakthroughs of \cite{EMM2} discussed above. 
In this section, using \cite{EMM2} and ideas of Alex Wright, we show how to extend the
results of \cite{HLM} to arbitrary $\genus \geq 3$.
For further background we refer the reader to the article of Wright
\cite{Wright survey} which surveys the theory of translation surfaces
taking into account the results of \cite{EMM2} and ensuing work.

A set $\MM$ as in Theorem \ref{prop: EMM}
will be called an {\em affine invariant manifold}. Note that we can
recover $\MM$ from $x$ as dilations of surfaces in $\overline{Gx}$, or
as the orbit-closure of $x$ under the group $\GL_2^\circ(\R)$ (the
connected component of the identity in the group of invertible $2
\times 2$ matrices). By choosing a
marking, lifting to $\HH_{\mathrm{m}}$, and using period coordinates, the tangent space
to $\MM$ at a point can be identified using period coordinates with a
subspace of $H^1(S, \Sigma; \R^2)$. Moreover if we identify $\R^2
\cong \C$ in the natural way, then this subspace is a $\C$-linear
subspace of $H^1(S, \Sigma; \C)$ (see \cite{EMM2, AEM}). 
The {\em field of definition} of $\MM$ is the
smallest field $k$ such that for any $x \in \MM$ and any identification of
$T_x(\MM)$ with a subspace $W \subset H^1(S, \Sigma; \R^2)$, using
period coordinates, $W$ can be defined by linear equations with
coefficients in $k$. 
The following summarizes results of Avila, Eskin,
Filip, M\"oller and Wright (see \cite{AEM, Wright_field, Filip} and the
references therein). 

\begin{thm}\name{thm: Wright Filip}
The field of definition of $\MM$ is the intersection of the holonomy
fields of all surfaces in $\MM$. It is a totally real field.
\end{thm}


Let $\MM$ be an affine invariant manifold, contained in a stratum
$\HH$ of surfaces of genus $\genus$, and let $\Res$ be the map as in
\equ{eq: defn Res}. We will say that {\em $\MM$ is of full
rank in $\HH$} if the restriction of $\Res$ to the tangent space to
$\MM$ is surjective, that is if $\dim_\C \Res(W) = 2\genus$, where $W
\subset H^1(S, \Sigma; \R^2)$ is
the subspace corresponding to $T_x(\MM)$ for some $x\in \MM$ using
some marking. Similarly we will say that $\overline{Gx_0}$ is of full
rank if the corresponding affine invariant manifold $\MM$ is of full
rank.

Note that if $\MM$ is of full rank then 
\eq{eq: dimension formula}{
\dim_{\C} \MM \geq \dim_{\C} \HH - \dim_{\C}
\mathfrak{R},}
 where $\mathfrak{R}$ is as in \eqref{eq: natural
  subspace}.  
We will prove the following: 
\begin{thm}\name{thm: full rank Wright idea}
Let $x$ be a translation surface of genus $\genus$ which 
admits a pseudo-Anosov map such that the holonomy field of $x$ is of
degree $\genus$ and does not contain
any totally real subfields other than $\Q$ (and in particular is not itself totally real).
Then $\overline{Gx}$ is of
full rank. 
\end{thm}

\begin{proof}[Proof (Alex Wright, cf. \cite
{Wright_field}, Cor. 8.1).]
By Theorem \ref{thm: Wright Filip}, the field of definition of $\MM$
is totally real and contained in the holonomy field of $x$. By
assumption this implies that the field of definition 
is $\Q$. 

Choosing $x_1 \in \MM$ and a marked surface $\x_1 \to x_1$, let $W \subset H^1(S, \Sigma; \R^2)$ be the
subspace tangent to $\pi^{-1}(\MM)$ at $\x_1$. Let $V = \Res(W)$, so
that when identifying
$H^1(S; \R^2) \cong H^1(S; \C)$, we have that $V$ is  a
$\C$-vector subspace of $H^1(S; \C)$. By
\cite[Thm. 5]{Wright_field}, $V$ is also defined over $\Q$. Since $V
\subset H^1(S;\C)$ and $\dim H^1(S;\C)=2\genus$, it suffices to show that
$\dim_{\C} V \geq 2\genus$. Let $H^1(S;
\C)^{\mathrm{(st)}}$ be the subspace of $H^1(S; \C)$ generated by the 1-forms
$\hol_{\mathrm{x}}(x), \hol_{\mathrm{y}}(x)$. This is a 2-dimensional complex linear 
subspace defined over the holonomy field $k$ of $x$. It is contained in
$V$ because $\MM$ is $G$-invariant. We can find a basis $v_1, v_2$ of $H^1(S;
\R)^{\mathrm{(st)}}$ consisting of vectors with coefficients in
$k$. The images of these vectors under the field 
embeddings $k \to \C$ are contained in $V$ because $V$ is defined over
$\Q$. According to \cite[Thm. 1.5]{Wright_field} (following arguments
developed by M\"oller \cite{Moller JAMS} for lattice surfaces), these vectors are
linearly independent, and so we find that $V$ contains $2\genus$ linearly
independent vectors. That is, $\dim_{\C} \MM \geq 2\genus = \dim_{\C} H^1(S;
\C)$, which shows that $\MM$ is of full rank.
%
\end{proof}

Recall Theorem \ref{thm: no real subfields}, proved in Appendix
\ref{appendix: BSZ} as Corollary \ref{NoTrCor}: 

\begin{thm}[Bary-Soroker,
Shusterman and Zannier]
\name{thm: new name}
Let $\genus \geq 3$ and let $\alpha$ be the unique real number in
$[0,1]$ satisfying $\alpha + 
\cdots + \alpha^\genus =1$. Then $\Q(\alpha)$ contains no totally real
subfields other than $\Q$.  
\end{thm}

\begin{cor}\name{cor: extending HLM}
Let $x_0$ be the Arnoux-Yoccoz surface in genus $\genus \geq 3$. Then
$\overline{Gx_0}$ is of full rank. 
\end{cor}

\begin{proof}
Let $k = \Q(\alpha)$. 
Since $x_0$ admits a pseudo-Anosov diffeomorphism
with derivative $\til g$, its holonomy field is
$\Q(\alpha+\alpha^{-1})$ \cite[Theorem 28]{KS}. In particular the
holonomy field of $x_0$ is contained in $k$. (These fields actually
coincide, that is 
$\Q(\alpha + \alpha^{-1}) = \Q(\alpha)$, but we will not need this fact.) 
%
According to 
Theorem \ref{thm: no real subfields}, $k$ has no totally real
subfields other than $\Q$ and hence neither does the holonomy field of
$x_0$. By Theorem \ref{thm: full rank
  Wright idea}, $\overline{Gx_0}$ is of
full rank.
%
%
\end{proof}

\begin{remark}
\label{rem: Mirzakhani Wright}
In \cite{HLM}, the orbit-closure $\overline{Gx_0}$ is determined
explicitly in case $\genus =3$: it is the subset of hyperelliptic
surfaces in $\HH^{\mathrm{odd}}(2,2)$. In work of Mirzakhani and Wright
\cite{Mirzakhani-Wright}, it is shown that the only full rank 
sub-loci of $\HH(\genus -1, \genus -1)$ for $\genus \geq 3$ arise from the
space of quadratic differentials $\mathcal{Q}\left(\genus -2, \genus
  -2, -1^{2\genus} \right)$ by a double cover
construction. In particular the only proper affine invariant
submanifold of full rank consists of surfaces with a hyper-elliptic
involution. As we 
have seen in Proposition \ref{prop: no involutions}, for $\genus \geq
4$ the Arnoux-Yoccoz surface $x_0$ has no 
involutions with derivative $-\mathrm{Id}$, and therefore Corollary
\ref{cor: extending HLM} and the result of 
\cite{Mirzakhani-Wright} imply that when $\genus
\geq 4$, $Gx_0$ is dense in its stratum component, proving Theorem
\ref{thm: orbit closure}. 
\end{remark}

We will need the following:
\begin{prop}\name{prop: HLM}
Suppose $\MM$ is an affine invariant manifold that is of full rank in
a stratum $\HH$. For any $x \in \HH$ let $\HH_0$ be the subset of the
connected component of $\HH$ containing $x$, consisting of surfaces
whose area coincides with that of $x$. Let $\MM_0 = \MM \cap \HH_0$. Then the  subset 
$$
\mathcal{U} = 
\{\rel^u(z): z \in \MM_0, \, u\in \mathfrak{R}, \, \rel^u(z) \text{ is
  defined} \}
$$
contains an open subset of $\HH_0$ and is dense in $\HH_0$. 
\end{prop}
\begin{proof}
Since $\MM$ is of full rank, $W \oplus \mathfrak{R} = H^1(S, \Sigma;
\R^2)$ (where as above, $W$ is the subspace above identified with the tangent
space to $\MM$). Since the map $(z, u) \mapsto \rel^u(z)$ is defined in
an open set in $\HH \times \mathfrak{R}$, for any $z \in \MM$ there
are open neighborhoods $U_1 \subset \MM$ of $z$ and $U_2 \subset
\mathfrak{R}$ of $0$ such that the map 
$$
U_1 \times U_2 \to \HH, \ \ (z',u) \mapsto \rel^u(z')
$$
is a homeomorphism onto its image. Since $\MM$ has full rank, the
image contains an open subset of $\HH$. Moreover the same is true if
we replace everywhere $\MM$ with $\MM_0$ and $\HH$ with $\HH_0$. In order to show that
$\mathcal{U}$ is dense, by ergodicity of the $G$-action on $\HH_0$ it suffices to
show that $\mathcal{U}$ is $G$-invariant. Let $z_2 = \rel^u(z_1) \in
\mathcal{U}$, where $u \in \mathfrak{R}$ and $z_1 \in \MM_0$, and let $g
\in G$. Then $gz_1 \in \MM_0$ since $\MM_0$ is $G$-invariant, and $gz_2 =
\rel^{gu} (gz_2) $ is defined and contained in $ \mathcal{U}$ by
Proposition \ref{prop: rel and G commutation}. This 
completes the proof. 
\end{proof}

As above $\HH_0$ denote the subset of the component of $\HH(\genus
-1, \genus -1)$ containing $x_0$ and consisting of surfaces with the
same area as $x_0$. 
Combining Proposition \ref{prop: HLM}, Corollary \ref{cor: extending HLM} and
Proposition \ref{prop: V suffices} we obtain: 
\begin{cor}\name{cor: corollary V}
For $\genus \geq 3$, 
$\overline{U x_0}$ is an affine invariant manifold of full rank, and
the only closed subset of $\HH_0$ containing $Ux_0$ and invariant under the rel
foliation is $\HH_0$.
\end{cor}

\section{Minimal tori for the horocycle flow and real-rel}\name{sec:
  minimal sets} 
In this section we give basic information about rel leaves of
periodic surfaces, and use this to complete the proof of Theorem
\ref{thm: rel}. 
While the results below are
elementary, they are  of independent interest and we formulate them in
greater generality than we need.

\subsection{Twist coordinates for completely periodic
  surfaces}
\label{sect: twist coordinates}
Let $\x$ be a marked translation surface with a non-empty labeled singular set $\Sigma$,
and suppose $\x$ is completely periodic in the horizontal direction.
Then $\x$ admits a decomposition into horizontal cylinders, and
the surface can be recovered by knowing some related combinatorial
data and some geometric parameters. Denote the horizontal cylinders by $C_1,
\ldots, C_m$.  
A {\em separatrix diagram } for a horizontal cylinder decomposition
consists of the ribbon graph of right-pointing horizontal
saddle connections forming the union of boundaries of cylinders, whose
vertices are given by $\Sigma$ (with orientation at vertices induced
by the translation surface structure on a neighborhood of the singular
point), and an indication of which pairs of
circles in the diagram bound each cylinder $C_i$. For more information
see \cite{KZ}, where this notion was introduced. We add more combinatorial information by selecting 
for each horizontal cylinder $C_i$ an upward-oriented saddle
connection $\sigma_i$ whose image is inside $C_i$ and which  joins the
bottom boundary component of $C_i$ to the top boundary component. The
union of the separatrix diagram and the saddle connections 
$\sigma_i$ still has a ribbon graph structure induced by the surface. 
The complete {\em combinatorial data} for our cylinder decomposition
consists of a labeling of cylinders, and the ribbon graph which is the
union of 
the separatrix diagram and the saddle connections $\sigma_i,\ i=1,
\ldots, m$.

Fixing the number and labels of cylinders and the combinatorial data
above, the marked translation surface 
structure on $\x$ is entirely determined by the following {\em parameters}:
the circumferences $c_1, \ldots, c_m \in \R_{>0}$ of the cylinders; 
the lengths $\ell_1, \ldots, \ell_n\in \R_{>0}$
of the horizontal saddle connections along the boundaries of the
cylinders; 
and the holonomies 
\eq{eq: if you like it}{
\hol(\sigma_i,\x)=(x_i, y_i) \in \R \times \R_{>0}
}
of each saddle connection $\sigma_i$ for $i=1, \ldots, m$. Observe
that each $y_i$ records the height of the cylinder $C_i$.

We call the numbers $x_i$ the {\em twist parameters}. Observe that we
can vary the twist parameters at will. That is, given $\x$ as above,
there is a map 
\eq{eq:til Phi}{
\til \Phi:\R^m \to \HH_{\mathrm{m}}; \quad (\hat x_1, \ldots, \hat x_m) \mapsto \hat \x,}
where $\hat \x$ is the surface built to have the same parameters as $\x$ except
with twist parameters given by $\hat x_1, \ldots, \hat x_m$,
and where $\HH_{\mathrm{m}}$ is the space of marked translation
surfaces structures modeled on the same surface 
and singularity set as $\x$.
Let $\pi:
\HH_{\mathrm{m}} \to \HH$ be the natural projection. The image of
$\pi \circ \til \Phi$ is the set of all translation surfaces which
have a horizontal cylinder decomposition with the same combinatorial
data as $x = \pi(\x)$, and the same parameters describing cylinder
circumferences, lengths of horizontal saddle connections, and heights of
cylinders. We refer to this set as the {\em horizontal twist space at
  $x$}, and denote it by $\mathcal{HT}_x$. We wish to explicitly parameterize this space. 

Let $x_1, \ldots, x_m$ denote the twist parameters for $\x$, and choose a second set of twist parameters 
$\hat x_1, \ldots, \hat x_m$. Then $\hat \x$ can be obtained from $\x$ by slicing each cylinder $C_i$ 
along a geodesic core curve and regluing so that the right side has moved right by $\hat x_i-x_i$. 
Thus
for each $\gamma \in H_1(S,\Sigma; \Z)$, 
\begin{equation}
\label{eq:Phi holonomy}
\hol(\gamma,\hat \x)=\hol(\gamma,\x)+\sum_{i=1}^m 
\big(
(\hat x_i-x_i) (\gamma \cap C_i),0\big)
\end{equation}
where $\cap$ from $\gamma \cap C_i$ denotes the algebraic intersection pairing.
Recall from (\ref{eq:cohomology class of cylinder}) that $C_i^\ast \in H^1(S,\Sigma; \Z)$ to denotes the cohomology class defined by 
$C_i^\ast(\gamma)=\gamma \cap C_i$. We have observed:
\begin{prop}
\label{prop:Phi derivative}
The map $\til \Phi$ is an affine map whose derivative is 
\eq{eq: satisfies}{
D \til \Phi\left(\frac{\partial}{\partial x_i} \right)=(C_i^\ast, 0) \in H^1(S,
\Sigma; \R^2), \ \ i=1, \ldots, m.
}
\end{prop}

Recall that $\til \Phi$ defined in (\ref{eq:til Phi}) maps vectors to elements of
$\HH_{\mathrm{m}}$, i.e.  {\em marked} surfaces. Since holonomies of
the saddle connections $\sigma_i$ distinguish surfaces in the image,
$\til \Phi$ is injective. However, the map 
$\pi \circ \til\Phi:\R^m \to \HH$
is certainly not injective.
In the space $\HH$ we consider translation surfaces equivalent if they
differ by the action of an element
of $\Mod(S,\Sigma)$. Let
$\operatorname{M}(\x) \subset \Mod(S,\Sigma)$ denote the subgroup
consisting of equivalence classes of orientation preserving
homeomorphisms of $\x$ such that:
\begin{itemize}
\item Each cylinder $C_i$ is mapped to a cylinder $C_j$ of the same
  circumference and height.
\item Each horizontal saddle connection is mapped to a horizontal saddle
  connection of the same length, respecting the orientation. 
\end{itemize}
Note that each element
of $\mathrm{M}(\mathbf{x})$ preserves the image of $\til \Phi$, and 
that distinct twist parameters yield the same surface in $\HH$ if and
only if they differ by an element of $\operatorname{M}(\x)$. 
In light of Proposition
\ref{prop:Phi derivative}, $\operatorname{M}(\x)$ pulls back to an
affine action on the twist parameters, and we obtain an affine homeomorphism
\eq{eq:Phi}{\Phi: \R^m / \operatorname{M}(\x) \to \mathcal{HT}_x
  \subset \HH.}

Some elements of the subgroup $\operatorname{M}(\x)$ are clear. Each
Dehn twist $\tau_i \in \Mod(S,\Sigma)$ 
in each horizontal cylinder $C_i$ lies in $\operatorname{M}(\x)$. The
action of $\tau_i$ on twist parameters just affects the twist
parameter $x_i$ of $C_i$ and has the effect of adding $c_i$, the
circumference of $C_i$. The {\em multi-twist subgroup} $\operatorname{M}_0(\x) = \langle
\tau_1, \ldots, \tau_m \rangle$ 
of $\operatorname{M}(\x)$ is
isomorphic to $\Z^m$. Moreover, 
$\mathrm{M}(\mathbf{x})$ acts by permutations on the horizontal saddle connections, and
$\operatorname{M}_0(\x)$ is the kernel of this permutation action, and
hence is normal and of finite index in
$\operatorname{M}(\x)$. Thus we  have  
the following short exact sequence of groups
\eq{eq:exact}{
\{1\} \to \operatorname{M}_0(\x) \to \operatorname{M}(\x) \to \Delta \to \{1\},}
where $\Delta$ is a subgroup of the group of permutations of the horizontal
saddle connections. 
We set $\mathbb{T} = \R^m /\operatorname{M}_0(\x) \cong \prod_{i=1}^m \R / c_i
\Z $ (an $m$-dimensional torus). By normality, the action of
$\operatorname{M}(\x)$ on $\R^m$ descends to an action 
on ${\mathbb{T}}$ which factors through an action of  $\Delta$ via
(\ref{eq:exact}). Thus we have the sequence of covers 
$$\R^m \to \mathbb{T}  \to  \R^m / \operatorname{M}(\x) \cong
\mathbb{T}/\Delta.$$
We see that $\mathcal{HT}_x$ is isomorphic to the quotient of a torus by
a finite group of linear automorphisms. 
When $\Delta$ is trivial, we actually have $\R^m /
\operatorname{M}(\x) = \mathbb{T}$. Thus the following holds:

\begin{prop}
\name{prop:Phi}
Suppose that the horizontal cylinders of $x$ have distinct
circumferences or distinct heights, and that each 
cylinder has a saddle connection on its boundary whose length is distinct from the
lengths of other saddle connections on the boundary. 
Then $\operatorname{M}(\x)= \operatorname{M}_0(\x)$
and therefore $\Phi: \mathbb{T} \to \mathcal{HT}_x$ is an isomorphism of affine manifolds.
\end{prop}

\begin{remark}
The above discussion equips the quotient $\mathcal{HT}_x
\cong \mathbb{T}/\Delta$ with the structure of an affine orbifold,
since it is the quotient of $\mathbb{T}$ by the action of a finite
group of affine automorphisms $\Delta$. For an example in which
$\Delta$ is nontrivial and $\mathcal{HT}_x $ is not a torus, let $x$
be the {\em Escher  
staircase surface} obtained by cyclically gluing $2m$ squares (see e.g.
\cite[Figure 3]{LW}). The surface has $m$ parallel cylinders, the
torus $\mathbb{T}$ is isomorphic to $(\R/\Z)^m$, and the
group $\Delta$ is a cyclic group generated by a homeomorphism of the
surface which goes one step up the staircase. The group $\Delta$ acts
on $(\R/\Z)^m$ by cyclic permutations of coordinates. 
\end{remark}

\begin{remark}
The maps \equ{eq:til Phi} and \equ{eq:Phi} were used in
\cite{calanque} in order to analyze the horocycle flow on completely
periodic surfaces, but the case in which $\Delta$ is
nontrivial was overlooked. Nevertheless \cite[Prop. 4]{calanque} is true as stated. Namely, in order
to prove this one can work in a finite cover $\HH'$ of $\HH$ in which $\mathbb{T}$ embeds,
analyzing the closure of the $U$-orbit in $\HH'$ via the arguments
used in 
\cite[Proof of Prop. 4]{calanque}. Then, using the fact that 
elements of $\operatorname{M}(\x)$ preserve circumferences and
heights of cylinders which they permute, it is not hard to
show that the validity of \cite[Prop. 4]{calanque} in $\HH'$ implies
that the same statement is valid in $\HH$. 
%
%
%
\end{remark}

\subsection{Real-rel flow in twist coordinates}
\name{sect: rel twist}
Now we will specialize to the setting where $\x \in \HH_{\mathrm{m}}$ has two singularities,
which we label by the symbols $\bullet$ and $\circ$. Figure \ref{fig:partial rel}
shows an example.
We continue to  suppose that the horizontal direction is completely periodic, but now we
will also assume that $\x$ admits no horizontal saddle connections joining distinct singularities.
That is, each boundary edge of each horizontal cylinder
contains only one of the two singularities. 
We will order the cylinders so that $C_1, \ldots, 
C_k$ have the singularity $\circ$ on their bottom and the 
singularity $\bullet$ on the top,
so that $C_{k +1}, \ldots, C_\ell$ have $\bullet$  on the
bottom and $\circ$ on the top, and $C_{\ell+1}, \ldots,
C_m$ have the same singularity on both boundary components. 

We observe that the real-rel flow applied to a surface in the
horizontal twist space can be viewed as only changing the twist
parameters. Concretely, $\rel_r^{(h)}$ decreases the twist parameters
of cylinders $C_1, \ldots, C_k$ by $r$, increases the twist parameters
of $C_{k+1}, \ldots, C_\ell$ by $r$, and does not change the twist
parameters of $C_{\ell+1}, \ldots, C_m$. Therefore we find:

\begin{prop}
\label{prop:conjugacy}
The horizontal twist space $\mathcal{HT}_x$ is invariant under $\rel^{(h)}$. Define 
$$\vec{w} \in \R^m, \ \ w_i=\begin{cases} -1 & \text{if $i \leq k$}\\
1 & \text{if $k < i \leq \ell,$}\\
0 & \text{if $i >\ell$}.
\end{cases}$$
Then the straightline flow 
$$F^r_{\vec{w}} : \vec{x} \mapsto \vec{x}+r \vec{w}$$ 
on $\R^m$ induces a well-defined straightline flow on $\R^m /
\operatorname{M}(\x)$  and $\Phi$ is a topological conjugacy from this
flow to the restriction of $\rel_r^{(h)}$ to $\mathcal{HT}_x$; that is
$\Phi \circ F^r_{\vec{w}} = \rel^{(h)}_r \circ \Phi$ for every $r$. 
\end{prop}

\begin{proof}
In the case $\operatorname{M}(\x) = \operatorname{M}_0(\x)$, the
horizontal twist space $\mathcal{HT}_x$ is isomorphic to the torus $\mathbb{T}$ and it is
straightforward to check that the effect of applying $\rel^{(h)}_r$ on
the twist coordinates is exactly $x_i \mapsto x_i + rw_i$,
giving the required conjugacy. In the general case we need to show
that the action of  $F^r_{\vec{w}}$ and of the group $\Delta$ on
$\mathbb{T}$ commute; indeed this will imply both that the action of
$F^r_{\vec{w}}$ on $\mathbb{T} /\Delta \cong \mathcal{HT}_x$
is well-defined, and that 
$\Phi$ intertwines the straightline flow $F^r_{\vec{w}}$ on
$\R^m/\operatorname{M}(\x)$ with the real-rel flow $\rel^{(h)}_r$ on
$\mathcal{HT}_x$. 
The definition of the $w_i$ implies that $F^r_{\vec{w}}$ and
$\Delta$ commute provided the permutation action of $\operatorname{M}(\x)$ on the horizontal 
cylinders preserves each of the three collections of cylinders $\{C_1, \ldots, C_k\}, \{C_{k+1}, \ldots,
C_\ell\}, \{C_{\ell+1}, \ldots, C_m\}$; this in turn follows from our assumption that
singularities are labeled, and the definition of $\operatorname{M}(\x).$
\end{proof}

Given a real vector space $V$, a {\em $\Q$-structure} on $V$ is a choice of
a $\Q$-linear subspace $V_0$ such that $V = V_0 \otimes_{\Q}
\R$ (i.e. there is a basis of $V_0$ as a vector space over $\Q$, which is a
basis of $V$ as a vector space over $\R$). The elements of $V_0$ are
then called {\em rational} points of 
$V$. If $V_1, V_2$ are vector spaces with $\Q$-structures, then a
linear transformation $T:
V_1 \to V_2$ is said to be {\em defined over $\Q$} if it maps rational
points to rational points. 
There is a natural $\Q$-structure on $H^1(S, \Sigma; \R^2)$, namely
$H^1(S, \Sigma; \Q^2)$. Since the action of $\Mod(S,
\Sigma)$ preserves $H_1(S, \Sigma; \Z)$ this induces a well-defined
$\Q$-structure on $\HH$. Moreover since $\HH$ is an affine manifold locally modeled on
$H^1(S, \Sigma ; \R^2)$, the tangent space to $\HH$ at any $x \in \HH$
inherits a $\Q$-structure. With respect to this $\Q$-structure, we obtain:

\begin{prop}\name{prop: real rel closure}
Retaining the notation above, let 
\eq{eq: retaining}{\mathcal{O}(x) = 
\overline{\{\rel^{(h)}_sx : s \in \R \}}\subset \mathcal{HT}_x.
}
Then ${\mathcal
  O}(x)$ is a $d$-dimensional affine sub-orbifold of $\HH$, where $d$ is the
dimension of the $\Q$-vector space 
\eq{eq: dim}{
\spa_{\Q} \left\{\frac{-1}{c_1}, \cdots, \frac{-1}{c_k},
  \frac{1}{c_{k+1}}, \ldots, \frac{1}{c_\ell} \right\} \subset \R.} 
\ul{Moreover, for every $x$, the tangent space to $\mathcal{O}(x)$ is a
$\Q$-subspace of $H^1(S, \Sigma; \R^2)$.} \footnote{Underlined statements are erroneous. See the Erratum in Appendix \ref{erratum}.}
\end{prop}
\begin{proof}
We will work with the
standard $m$-torus ${\mathbb T}^m=\R^m/\Z^m$, and 
define 
\eq{def:Psi}{
\Psi: {\mathbb T}^m \to \mathcal{HT}_x; \quad (t_1, \ldots, t_m) \mapsto
\Phi(c_1 t_1, \ldots, c_m t_m).
}
Then the conjugacy from Proposition \ref{prop:conjugacy} leads to a
semi-conjugacy from the straight-line flow on ${\mathbb T}^m$ in
direction  
$$\vec{v}=\left(\frac{-1}{c_1}, \cdots, \frac{-1}{c_k},
  \frac{1}{c_{k+1}}, 
  \cdots, \frac{1}{c_\ell},0, \ldots, 0\right).$$
The orbit closure of the origin of this straight-line flow is a
rational subtorus, and the tangent space to the origin is the smallest
real subspace $V$ of $\R^m$ defined over $\Q$ and containing the
vector $\vec{v}$. Every rational relation among the coordinates of $\vec{v}$
gives a linear equation with $\Q$-coefficients satisfied by $\vec{v}$, and
vice versa; this implies that the dimension of $V$ is the same as the rational
dimension of \equ{eq: dim}. 

Similarly to \equ{def:Psi}, let $\til \Psi: \R^m \to
\HH_{\mathrm{m}}$ be defined by $\til \Psi (t_1, \ldots, t_m) = \til
\Phi(c_1t_1, \ldots, c_mt_m)$. 
Then $\til \Psi$ intertwines
the action of $\Z^m$ on $\R^m$ by translations, with the action of
$\operatorname{M}_0(\x)$ on the image of $\til \Psi$, and we have $\mathcal{O}(x) =
\pi (\til \Psi(V))$.  
Let $\vec{v}_1, \ldots, \vec{v}_d$ be a basis of $V$ contained in $\Z^m$, and let
$\Gamma \subset \Z^m$ be the sub-lattice $ \langle \vec{v}_1, \ldots, \vec{v}_d \rangle.$ 
The subspace $V$ is fixed by the action
$\Gamma \subset \Z^m$ by translations, and $\til \Psi$ intertwines this
translation action with  a translation action of a subgroup of 
$\operatorname{M}_0(\x)$. Thus in order to prove that the tangent space to
$\mathcal{O}(x)$ is defined over $\Q$, it is enough to prove that any
element of $\operatorname{M}_0(\x)$ acts by translation by a rational 
vector. 

As we have seen (see \equ{eq:Phi holonomy}), the action of each $\tau_i$ on $H^1(S,
\Sigma; \R^2)$ is induced by its action on $H_1(S, \Sigma ; \Z)$ via 
$$\gamma \mapsto \gamma + (\gamma \cap C_i) C_i,$$ 
i.e. by
translating by a vector in $H_1(S, \Sigma; \Z)$. \ul{Now the assertion
follows from the definition of the $\Q$-structure on
$H^1(S, \Sigma ; \R^2)$.} \footnote{The displacement in period coordinates, $\tau_i(\x)-\x$, is rational if and only if $\hol(C_i; \x) \in \Q^2$.}
%
\end{proof}


\subsection{Real-rel flow on deformations of the Arnoux-Yoccoz
  surface}\name{subsec: deformations etc}
We now specialize further  to $x_r = \rel^{(v)}_r x_0$. Throughout
this section we assume that $r>0$ and $r(1-\alpha)$ is not an integral
power of $\alpha$.
Then $x_r$ admits a decomposition into $\genus  +1$ horizontal cylinders by
Theorem \ref{thm:completely periodic ray}(2)(b).
Define $\mathcal{O}_r = \mathcal{O}(x_r)$ via \equ{eq:
  retaining}, and denote the tangent space to $\mathcal{O}_r$ at $x_r$
by $T_r$. Then we have:  
\begin{lem}
\name{lem:tangent space}
With  the notation above,  
$T_r$ is a $\genus$-dimensional \ul{$\Q$-subspace} of $H^1(S,\Sigma;\R^2)$, 
and ${\mathcal
  O}_r$ 
is a $\genus$-dimensional affine torus. 
\end{lem}
\begin{proof}
By Theorem
\ref{thm:completely periodic ray}(2)(b) the circumferences of
the horizontal cylinders on $x_r$ are $\alpha^{k+2}, \ldots,
\alpha^{k+\genus+2}$ for some $k$ depending on $r$. Multiplying
through by $\alpha^{k+\genus+2}$ in formula \eqref{eq: dim}, we see that the
dimension of $\mathcal{O}_r$ is $\genus$. To see that the affine
sub-orbifold $\mathcal{O}_r$ is a torus (i.e. not its quotient by a
nontrivial finite group) we use Proposition \ref{prop:Phi} and the fact
that the circumferences of the cylinders are distinct. The other assertions of the
Lemma follow from Proposition \ref{prop: real rel closure}. 
\end{proof}

Define the {\em horizontal twist cohomology subspace} $P  $ 
to be the real subspace of $H^1(S,\Sigma;\R^2)$ spanned by the classes of
the form $(C^\ast, 0)$, where $C$ varies over the cylinders of
$x_r$. As we have seen, $P$ is the tangent space to
$\mathcal{HT}_r$, and by Proposition \ref{prop: real rel closure}
the subspace $T_r$ lies in $P$ for all $r$ as above. Furthermore we
have: 

\begin{cor} \name{cor:action on P}
The subspace $P \subset H^1(S,\Sigma;\R^2)$ is $\genus+1$ dimensional
and is independent of the choice of $r$ selected as above. The action of $\varphi^\ast$ on $H^1(S,\Sigma;\R^2)$
preserves $P$ and has $\big(\hol_{\mathrm y}(x_0), 0\big)$ as its dominant eigenvector with eigenvalue $\alpha^{-1}$. 
\end{cor}
\begin{proof}
Recall the holonomy map for the family of surfaces $\{x_r\}$ is 
$$\widetilde \hol_y: H_1(S, \Sigma; \Q) \to \Q(\alpha)+\Q r; \quad \gamma \mapsto 
\hol_y(\gamma; x_r),$$
where $r$ is viewed as a free variable. See (\ref{eq: holonomy of family}). Because
$x_r=\rel^{(v)}_r x_0$, 
$$\widetilde \hol_y(\gamma)=\hol_y(\gamma, x_0)+\delta(\gamma)r$$
where the boundary of $\gamma$ in $H_0(\Sigma; \Q)$ is the scalar $\delta(\gamma) \in \Q$
times the class $[\circ]-[\bullet] \in H_0(\Sigma; \Z)$. 

Lemma \ref{lem: span of cylinders} tells us that $P$ is $\genus+1$ dimensional
and when restricted to $H_1(S,\Sigma; \Q)$
is the set of maps of the form 
$(L \circ \widetilde \hol_y,0)$ where $L:\Q(\alpha)+\Q r \to \R$ is $\Q$-linear. 
Since here we take $r$ as a free variable, $P$ is independent of $r$. 

Let $\gamma \in H_1(S, \Sigma; \Q)$ be arbitrary.
The action of $\varphi^\ast$ on the first coordinate of an element of $P$ is given by 
$$
\begin{array}{rcl}
\big(\varphi^\ast(L \circ \widetilde \hol_y)\big)(\gamma) & = &
L \circ \widetilde \hol_y(\varphi^{-1} \gamma) \\ 
& = &
L \Big( \hol_y\big(\varphi^{-1} \gamma; x_0)+\delta(\varphi^{-1} \gamma) r\big)\Big) \\
& = &
L \big( \alpha^{-1} \hol_y\big(\gamma; x_0)+\delta(\gamma) r\big)\Big).
\end{array}$$
The class $\big(\hol_y(x_0),0\big) \in P$ corresponds to the choice of $L$
which acts trivially on $\Q(\alpha)$ and sends $\Q r$ to zero. From the above, we can see that this class is an eigenvector
with eigenvalue $\alpha^{-1}$. Since $\alpha^{-1}$ is Pisot and
$\varphi^\ast$ acts as an integer matrix 
on cohomology, it dominates the action of its $\genus-1$ algebraic
conjugates. The remaining eigenvector is given by $\gamma \mapsto
\delta(\gamma)$ as above (corresponding to an $L$ which annihilates
$\Q(\alpha)$ and acts as the identity on $r \Q$), which has eigenvalue
one.   
\end{proof}

\begin{thm}\name{thm: crucial}
For any $r$ as above, let $r_n = \alpha^n r \to 0$. Then 
$$
U x_0 \subset \overline{\,\bigcup_n {\mathcal O}_{r_n}\,}.
$$
\end{thm}

\begin{proof}
In the affine orbifold structure on $\HH$, the orbit $U x_0$ is a line,
and the sets $\mathcal{O}_{r_n}$ are linear submanifolds. Since
$x_{r_n} \to x_0$ it suffices to show that the set of accumulation
points of the tangent space $T_{r_n}$ contains the tangent direction
to $U$. By definition of the $U$-action, 
the derivative of the 
horocycle flow $\frac{d}{ds}[u_sx_0]$ 
(as an element of the tangent
space to $\HH$ at $x_0$, identified with $H^1(S, \Sigma; \R^2)$) is precisely 
$\big(\hol_{\mathrm y}(x_0), 0\big)$. By Corollary
\ref{cor:action on P},  $\big(\hol_{\mathrm y}(x_0), 0\big)$ is the
dominant eigenvector for the action of $\varphi^\ast$ on $P$. 

On the
other hand taking $\til g = D \varphi$ and using Proposition
\ref{prop: rel and G commutation} we
see 
$$\{\rel^{(h)}_s x_{r_{n+1}} : s \in \R\} = \{\rel^{(h)}_s
\til g x_{r_n} : s \in \R \} = \{\til g \rel^{(h)}_s  x_{r_n}: s \in
\R\},$$ and hence $\mathcal{O}_{r_{n+1}} = \til g
\mathcal{O}_{r_n}$. Let $P_1 = \spa \big\{ \big(\hol_{\mathrm{y}}(x_0),0\big) \big\} \subset P$. Since $P_1$
is generated by an eigenvector, there is a
$\varphi^\ast$-invariant complementary subspace $P_2$. We claim that
the subspace $T_r \subset P$ is not contained in $P_2$.  
Assume to the contrary that $T_r \subset P_2$. 
Since $P$ has dimension $\genus+1$, the subspace $P_2$ has dimension $\genus$. 
If $T_r$ was contained in $P_2$ then by Lemma \ref{lem:tangent space}
and dimension considerations $T_r=P_2$. 
But $P_2$ is not rational because $\big(\hol_{\mathrm{y}}(x_0),0\big)$ is not rational.
\ul{This contradicts the rationality of $T_r$; see Lemma {\ref{lem:tangent
  space}}.} \footnote{See the Erratum in Appendix \ref{erratum}.}

Since $T_r \not \subset P_2$ there is a vector $v \in T_r$ with a
nontrivial projection onto $P_1$, with respect to the
$\varphi$-invariant splitting $P= P_1 \oplus P_2$. Since 
$\mathcal{O}_{r_{n+1}} = \til g
\mathcal{O}_{r_n}$ and $\til g = D \varphi$, when applying iterates of
$\varphi$ to $v$ we get a sequence of vectors in $T_{r_n}$ whose span 
converges projectively to the direction of the dominant
eigenvector, as required.
%
%
\end{proof}

\begin{proof}[Proof of Theorem \ref{thm: rel}]
Let $\HH_0$ be as in the statement of the Theorem and let 
$\Omega \subset \HH_0$ denote the 
closure of the rel leaf of $x_0$. For any $r>0$ 
and $r$ not an integral power of $\alpha$, $\Omega$ contains
$\mathcal{O}_r = \mathcal{O}(x_r)$ (as in \equ{eq: retaining}).
Hence by Theorem \ref{thm: crucial}, $\Omega$ contains
the orbit $Ux_0$. Applying Corollary \ref{cor: corollary V} gives the
required conclusion. 
\ignore{
$\Omega$ contains the hyperelliptic locus $\LL$. 

Since $\Omega$ is saturated with respect to the Rel foliation, 
for any $z \in \LL$ and any $u \in \mathfrak{R}$ for which $\rel^u (z)$
is defined, $\Omega$ contains $\rel^u(z)$. Given $z \in \LL$ and $g
\in G$, $gz \in \LL$ since $\LL$ is $G$-invariant. Moreover, if $\rel^u(z)$
is defined, by Proposition \ref{prop: rel and G commutation},
$\rel^{gu}(gz) = g\rel^u(z)$ is also defined and contained in 
$\Omega$. The set $$\{\rel^u(z): z \in \LL, \, u \in
\mathfrak{R},  \rel^u(z) \text{ is defined} \}$$ 
contains an open subset $\mathcal{U}$ of
$\HH(2,2)$ by Proposition \ref{prop: HLM}. Since
$\mathcal{U}$ has positive measure, with respect to the natural flat measure on
$\HH$, we can apply ergodicity of the $G$-action, to find that
$\mathcal{U}$ contains a point $z_0$ for which $\overline{Gz_0} =
\HH$. Since $\overline{Gz_0} \subset \Omega$  we find that $\Omega =
\HH$. }
\end{proof}

\newpage
\appendix

\section{Totally real subextensions}\name{appendix: BSZ}
\markleft{\uppercase{Bary-Soroker, Shusterman and Zannier}}

\begin{center}\uppercase{Lior Bary-Soroker, Mark Shusterman, and Umberto Zannier}
\end{center}

\maketitle
}

For a tower of fields $E \leq K \leq \bar{E}$ we say that an algebraic extension $M/E$ is totally-$K$ if every embedding of it into $\bar{E}$ over $E$ lands in $K$. Given an algebraic extension $F/E$ we give a bound on its maximal totally-$K$ subextension in case that the absolute Galois group $G_K$ is contained in a Sylow subgroup of $G_E$ of order prime to $[F : E]$. We use this bound to show that the maximal totally real subfield of $\mathbb{Q}[X]/(X^{n} + X^{n-1} + \dots + X - 1)$ is $\mathbb{Q}$. This confirms a conjecture made (and empirically tested) by Hooper and Weiss, allowing them to establish the density of the rel leaf of the Arnoux-Yoccoz surface in its connected component in the moduli space of translation surfaces of fixed genus, combinatorial characteristics of singularities, and area.

\subsection{Preliminaries}
In order to state our main result, we need to fix some notation. 

Let $E$ be a field, let $\bar{E}$ be an algebraic closure of $E$, and let $E \leq K \leq \bar{E}$ be an intermediate subfield. We say that an algebraic extension $M/E$ is \textbf{totally-$K$} if for every $\tau \in \mathrm{Hom}_E(M,\bar{E})$ we have $\mathrm{Im}(\tau) \leq K$, or equivalently, if $\mathrm{Hom}_E(M,\bar{E}) = \mathrm{Hom}_E(M,K)$. It is easily verified that the compositum (inside a given field) of a family of totally-$K$ extensions is once again totally-$K$, and that $E/E$ is totally-$K$, so there exists a unique \textbf{maximal totally-$K$ extension} of $E$ inside $\bar{E}$ which we denote by $$E_{\mathrm{t},K}.$$ We also have a relative version of this notion: given an algebraic extension $F/E$, we denote by $$E^F_{\mathrm{t},K} \defeq E_{\mathrm{t},K} \cap F$$ the unique \textbf{maximal totally-$K$ subextension} of $F/E$. 

For a set $A$, we let $\#A$ be the cardinality of $A$ in case that the latter is finite, and $\infty$ otherwise. We set $[F : E]_s \defeq \# \mathrm{Hom}_E(F,\bar{E})$, and use $E \leq_f F$ to stress that $[F : E] < \infty$. The maximal purely inseparable algebraic extension of $K$ is denoted by $K^{\mathrm{ins}}$, which is just the compositum of all finite extensions $R/K$ with $[R : K]_s = 1$. A prime number $p$ is said to divide $[F : E]_s$ if there exists some $$E \leq_f F_0 \leq F$$ such that $p \mid [F_0 : E]_s$. We would like to remind that the automorphism group of an algebraic field extension (e.g. $\mathrm{Aut}(\bar{E}/K)$) carries a topology (in our example the open subgroups are $\mathrm{Aut}(\bar{E}/K')$ for finite extensions $K'/K$) with respect to which it is a profinite group.

Recall that a directed poset is a nonempty set $I$ partially ordered by a relation $\leq$ such that for every $\alpha, \beta \in I$ there exists some \mbox{$\omega \in I$} such that $\omega \geq \alpha, \beta$. For example, the family of all finite subextensions $E \leq_f F_0 \leq F$ which we denote by $I_{F/E}$ is a directed poset with respect to inclusion. A subset $J \subseteq I$ is called cofinal if for each $\omega \in I$ there exists some $\pi \in J$ such that $\pi \geq  \omega$. Note that a cofinal subset is itself a directed poset, and that a subset $J \subseteq I_{F/E}$ is cofinal if and only if $\cup J = F$. Given a function $f \colon I \to [0,\infty)$, we say that a point $P \in [0,\infty]$ (viewed as the one-point compactification of the nonnegative real ray $\mathbb{R}_{\geq 0}$) is a limit point of $f$ if for every neighborhood $V \subseteq [0,\infty]$ of $P$, the inverse image $f^{-1}(V)$ is cofinal. A routine compactness argument shows that $f$ has at least one limit point. We define $$\liminf_{I} f$$ to be the infimum of the set of limit points of $f$, and notice that it is a limit point of $f$. 

\subsection{Results}
Our main result is a bound on the maximal totally-$K$ \mbox{subextension $E^F_{\mathrm{t},K}$.}

\begin{theorem} \label{FResThm}

Let $E$ be a field, let $\bar{E}$ be an algebraic closure, let $E \leq K \leq \bar{E}$ be an intermediate subfield, and let $F/E$ be an algebraic extension. Suppose that there exists a prime $p$ not dividing $[F : E]_s$, such that $\mathrm{Aut}(\bar{E}/K)$ is a pro-$p$ group. Then $$ [E^F_{\mathrm{t},K} : E]_{s} \leq \liminf_{I_{F/E}} |\mathrm{Hom}_E(-,K^{\mathrm{ins}})|.$$

\end{theorem}

For simplified (and more concrete) versions of the theorem, where $E$ is assumed to be perfect, see \corref{PerfCor} and \corref{PrimCor}. From these corollaries we deduce the case of odd $n$ in \corref{NoTrCor}. The case of even $n$ follows from the work of Martin (see \cite{M} and the proof of \corref{NoTrCor}). 

Hooper and Weiss use \corref{NoTrCor} (see Theorem \ref{thm: no real subfields}) as a crucial step in their proof of Theorem \ref{thm: rel}.


\subsection{Preparatory claims}
Let $G$ be a group acting on a set $X$, and pick some $x \in X$. We denote by $O_x \defeq \{gx : g \in G\}$ the orbit of $x$, and set $G_x \defeq \{g \in G : gx = x\}$ for the stabilizer of $x$. We call $|O_x|$ the length of the orbit of $x$. Also, $G_X$ denotes the (normal) subgroup of all elements of $G$ which act trivially on $X$.

\begin{proposition} \label{ProFixProp}

Let $p$ be a prime number, and let $\Gamma$ be a pro-$p$ group acting continuously on a finite set $X$ of cardinality coprime to $p$ ($X$ carries the discrete topology). Then there is an $x \in X$ such that $\gamma x=x$ for each $\gamma \in \Gamma$.

\end{proposition}

\begin{proof}

Discreteness implies that for each $x \in X$ the subgroup $\Gamma_x$ is open in $\Gamma$, so $$\Gamma_X = \bigcap_{x \in X} \Gamma_x$$ is an open normal subgroup of $\Gamma$ since $X$ is finite. Hence, the action of $\Gamma$ on $X$ factors through the finite $p$-group $P \defeq \Gamma/\Gamma_X$. For each $x \in X$ we have $|O_x| = [P : P_x]$, and $X$ is a finite union of disjoint orbits, so it is impossible for all of the orbits to be of length divisible by $p$ as $p \nmid |X|$. Thus, there is an orbit of length prime to $p$. But, the length of an orbit is the index of a subgroup, so it divides the order of the $p$-group $P$. We conclude that there exists an orbit of length $1$.
\end{proof}

A map $\beta \colon I \to J$ between directed posets is called a morphism if for all $x,y \in I$ with $x \geq y$ we have $\beta(x) \geq \beta(y)$.

\begin{proposition} \label{OrdMorProp}

Let $\beta \colon I \to J$ be a surjective morphism of directed posets, let $S \subseteq I$ be a cofinal subset, and let $f \colon J \to [0,\infty)$. Then:

\begin{enumerate}

\item The subset $\beta(S) \subseteq J$ is cofinal.

\item Each limit point of $f \circ \beta$ is a limit point of $f$.

\item We have the inequality $$\liminf_J f \leq \liminf_I (f \circ \beta).$$

\end{enumerate}

\end{proposition}

\begin{proof}
In order to establish $(1)$, pick some $j \in J$. Surjectivity provides an $i \in I$ with $\beta(i) = j$. Since $S$ is cofinal, there exists some $s \in S$ such that $s \geq i$. We see that $\beta(s) \geq \beta(i) = j$ as required. For $(2)$ let $P \in [0,\infty]$ be a limit point of $f \circ \beta$, and let $V$ be a neighborhood of it. This time surjectivity implies that $$f^{-1}(V) =  \beta( \beta^{-1}(f^{-1}(V))) = \beta ((f \circ \beta)^{-1}(V))$$ where the rightmost set is cofinal by $(1)$. Therefore, $f^{-1}(V)$ is cofinal as required. The infimum of a subset (the set of limit points of $f \circ \beta$) is never smaller than that of the ambient set (the set of limit points of $f$) so $(3)$ follows. 
\end{proof} 

\begin{proposition} \label{InvLimProp}

Let $I$ be a directed poset, and let $\{A_i\}_{i \in I}$ be an inverse system of finite sets. Then $$ \# \varprojlim_{i \in I} A_i \leq \liminf_{i \in I} |A_i|.$$

\end{proposition}

\begin{proof}
We may assume that the right hand side takes a finite value $m$, for otherwise there is nothing to prove. Take $V$ to be any open interval of length less than $1$ containing $m$, and note that it contains at most one integer. Since $m$ is a limit point of the integer-valued function $i \mapsto |A_i|$, the subset $\{i \in I : |A_i| \in V \} \subseteq I$ is cofinal, and in particular nonempty, so $V$ contains an integer. We conclude that there exists an integer in every neighborhood of $m$, which means that $m$ is an integer - the unique integer in any such $V$. Therefore, $$J \defeq \{i \in I : |A_i| = m\} \subseteq I$$ is cofinal, so it is not difficult (see \cite[Lemma 1.1.9]{RZ}) to show that $$\# \varprojlim_{i \in I} A_i = \# \varprojlim_{j \in J} A_j.$$ It is therefore sufficient to demonstrate that $$X \defeq \varprojlim_{j \in J} A_j$$ contains at most $m$ distinct elements.

Suppose that there exist distinct $x_1, \dots, x_{m+1} \in X$, and for $j \in J$, denote by $\pi_j \colon X \to A_j$ the projections from the inverse limit. Also, for $s \geq t \in J$, we denote by $\tau_{s,t} \colon A_s \to A_t$ the compatible maps of the inverse system. For any $1 \leq k \neq \ell \leq m+1$ the fact that $x_k \neq x_\ell$ implies that there exists some $j_{k,\ell} \in J$ such that $\pi_{j_{k,\ell}}(x_k) \neq \pi_{j_{k,\ell}}(x_\ell)$. Since $J$ is directed, there exists some $r \in J$ with $r \geq j_{k,\ell}$ for all possible values of $k, \ell$. Let us now show that $\pi_r(x_1), \dots, \pi_r(x_{m+1})$ are distinct. Pick some $1 \leq k \neq \ell \leq m+1$ and suppose that $\pi_r(x_k) = \pi_r(x_\ell)$, so that $\tau_{r,j_{k,\ell}}(\pi_r(x_k)) = \tau_{r,j_{k,\ell}}(\pi_r(x_\ell))$ which means that $\pi_{j_{k,\ell}}(x_k) = \pi_{j_{k,\ell}}(x_\ell)$ contrary to our choice of $j_{k,\ell}$. We have thus shown that $|A_r| > m$ which is a contradiction to the definition of $J$. 
\end{proof}

\subsection{The proof}
For a profinite group $G$, and a subgroup $\Gamma \leq G$, we denote by $\langle \Gamma \rangle^G$ the minimal closed normal subgroup of $G$ containing $\Gamma$ (i.e. the closure in $G$ of the subgroup generated by the conjugates of $\Gamma$ in $G$). We are now ready to establish \thmref{FResThm}. Since we may replace $F$ by any other field isomorphic to it over $E$, we may well assume (by \cite[Theorem 2.8]{L}) that $F \leq \bar{E}$. Thus, we need to prove:

\begin{theorem} \label{BoundMaxKThm}

Let $E \leq F,K \leq \bar{E}$ be field extensions. Suppose that there exists a prime $p$ not dividing $[F : E]_s$, such that $\mathrm{Aut}(\bar{E}/K)$ is a pro-$p$ group. Then $$ [E^F_{\mathrm{t},K} : E]_{s} \leq \liminf_{I_{F/E}} |\mathrm{Hom}_E(-,K^{\mathrm{ins}})|.$$

\end{theorem}

\begin{proof}
Let $E \leq F_0 \leq F$ be a finite subextension. Set $$G \defeq \mathrm{Aut}(\bar{E}/E), \ L \defeq E^F_{\mathrm{t},K}, \ L_0 \defeq E^{F_0}_{\mathrm{t},K}$$ $$A \defeq \mathrm{Hom}_E(F_0,\bar{E}), \  \Gamma \defeq \mathrm{Aut}(\bar{E}/K), \ H \defeq \langle \Gamma \rangle^G.$$ First we claim that $H$ fixes $L_0$ pointwise. Since $H$ acts continuously on $\bar{E}$ (endowed with the discrete topology) it is sufficient to check that a topological set of generators of $H$ acts trivially on $L_0$. Taking $\gamma \in \Gamma, \ g \in G,  \ \ell_0 \in L_0$ we see that $g\ell_0 \in K$ so $$(g^{-1} \gamma g) \ell_0 = g^{-1}(\gamma g \ell_0) = g^{-1} (g \ell_0) = \ell_0$$ since $\gamma$ is trivial on $K$, as claimed.

Since $F_0 \leq \bar{E}$, the inclusion map $\mathfrak{i} \colon F_0 \hookrightarrow \bar{E}$ is in $A$, so $A$ is nonempty, and finite because $|A| \leq [F_0 : E] < \infty$ (see \cite[Theorem 4.1]{L}). Given $g \in G, \ a \in A$ we see that $g \circ a \in A$, so it is immediate that $G$ acts on $A$ from the left. To see that the action is transitive, pick an $a \in A$ and extend it (using \cite[Theorem 2.8]{L}) to some $g \in G$, so that $g\mathfrak{i} = g \circ \mathfrak{i} = g|_{F_0} = a$ as required for transitivity. In order to show that the action of $G$ on $A$ (viewed as a discrete space) is continuous (i.e. factors through an open normal subgroup), denote by $N$ the normal closure of $F_0$ in $\bar{E}$, and note that the action of $\mathrm{Aut}(\bar{E}/N) \lhd_o G$ on $A$ is trivial since the image of any map $a \in A$ is contained in $N$. Therefore, the action factors through a finite quotient of $G$ as claimed. 

Restricting the action to $H$, we get a decomposition of $A$ into disjoint $H$-orbits: $$A = \bigcup_{j = 1}^m Ha_j$$ for some $a_1, \dots, a_m \in A$, and $m \in \mathbb{N}$. For $1 \leq j,k \leq m$, take some $g \in G$ for which $ga_k = a_j$ and recall that $H \lhd G$. We thus have $$|Ha_j| = |Hga_k| = |g^{-1}Hga_k| = |Ha_k|$$ so all the orbits of $H$ share the same length which we denote by $r$. Consequently, 
\begin{equation} \label{RmEq}
[F_0 : E]_s = |A| = rm
\end{equation} 
so $r$ is coprime to $p$, since $|A|$ is coprime to $p$ by our assumption. As $H$ acts trivially on $L_0$, we have  $H\mathfrak{i} \subseteq \mathrm{Hom}_{L_0}(F_0,\bar{E})$, so 
\begin{equation} \label{REq}
r  = |H\mathfrak{i}| \leq |\mathrm{Hom}_{L_0}(F_0,\bar{E})| = [F_0 : L_0]_s.
\end{equation}

Since $H$ acts continuously on each of its orbits, so does $\Gamma$. As a result, we have $m$ continuous actions of the pro-$p$ group $\Gamma$ on sets of cardinality prime to $p$. From \propref{ProFixProp}, we conclude that there are at least $m$ fixed points for the action of $\Gamma$ on $A$ (at least one fixed point in each orbit of $H$). On the other hand, an embedding $\varphi \in A$ is a fixed point of $\Gamma$ if and only if its image is contained in the fixed field of $\Gamma$ inside $\bar{E}$ which is $K^{\mathrm{ins}}$ as shown in \cite[Proposition 6.11]{L}. Summarizing, we find that 
\begin{equation} \label{LongEq}
\begin{split}
|\mathrm{Hom}_E(F_0,K^{\mathrm{ins}})| &= |A^{\Gamma}| \geq m \stackrel{\ref{RmEq}}{=} \frac{[F_0 : E]_s}{r} \\ &\stackrel{\ref{REq}}{\geq} \frac{[F_0 : E]_s}{[F_0 : L_0]_s} = [L_0 : E]_s
\end{split}
\end{equation}
where the last equality follows from \cite[Theorem 4.1]{L}. 

Finally, denote by $\beta \colon I_{F/E} \to I_{L/E}$ the epimorphism of directed posets given by $\beta(F_0) \defeq L_0 = E^{F_0}_{\mathrm{t},K}$ and observe that  
\begin{equation*}
\begin{split}
\liminf_{F_0 \in I_{F/E}} |\mathrm{Hom}_E(F_0,K^{\mathrm{ins}})| &\stackrel{\ref{LongEq}}{\geq}
\liminf_{F_0 \in I_{F/E}} [\beta(F_0) : E]_s \\
&\stackrel{\ref{OrdMorProp}}{\geq} \liminf_{L_0 \in I_{L/E}} [L_0 : E]_s \\ 
&= \liminf_{L_0 \in I_{L/E}} |\mathrm{Hom}_E(L_0, \bar{E})| \\
&\stackrel{\ref{InvLimProp}}{\geq} \# \varprojlim_{\mathclap{L_0 \in I_{L/E}}} \mathrm{Hom}_E(L_0, \bar{E}) \\ 
&= \# \mathrm{Hom}_E(\varinjlim_{\mathclap{L_0 \in I_{L/E}}} L_0, \bar{E}) \\
&= \# \mathrm{Hom}_E(L,\bar{E}) 
= [L : E]_s 
= [E^{F}_{\mathrm{t},K}: E]_s.
\end{split}
\end{equation*}
\end{proof}

\subsection{Conclusions}
Let us now derive some consequences from our main result. For a field $K$, we denote by $G_K$ its absolute Galois group.


\begin{corollary} \label{PerfCor}

Let $E \leq K \leq \bar{E}$ be an extension of the perfect field $E$, and let $F/E$ be a finite extension. Suppose that there exists a prime $p$ not dividing $[F : E]$, such that $G_K$ is a pro-$p$ group. Then $$ [E^F_{\mathrm{t},K} : E] \leq |\mathrm{Hom}_E(F,K)|.$$

\end{corollary} 

\begin{proof}
This is just a special case of \thmref{FResThm}. Let $K^{\mathrm{sep}}$ be a separable closure of $K$, and note that the separability of $E$ implies that $[M : E]_s = [M : E]$   for every finite extension $M/E$. Also, $K$ is perfect since it is an algebraic extension of $E$ (see \cite[Corollary 6.12]{L}). Hence, $$K^{\mathrm{ins}} = K, \ \ G_K \cong \mathrm{Gal}(K^{\mathrm{sep}}/K) \cong \mathrm{Aut}(\bar{E}/K).$$ At last, the cofinal subsets of $I_{F/E}$ are those containing $F$, so the $\liminf$ boils down to the value \mbox{at $F$}. 
\end{proof}

In view of the primitive element theorem (see \cite[Theorem 4.6]{L}), it is natural to "restrict" our attention to field extensions of $E$ of the form $E[X]/(g)$ for some $g(X) \in E[X]$.

\begin{corollary} \label{PrimCor}

Let $E \leq K \leq \bar{E}$ be an extension of the perfect field $E$, and let $g(X) \in E[X]$ be an irreducible polynomial. Suppose that there exists a prime $p$ not dividing $\deg(g(X))$, such that $G_K$ is a pro-$p$ group. Then for $F \defeq E[X]/(g(X)) $ we have $$[E^F_{\mathrm{t},K} : E] \leq |\{x \in K : g(x) = 0 \} |.$$

\end{corollary}

\begin{proof}
Since $[F : E] = \deg(g(X))$, we may invoke \corref{PerfCor}. An embedding of $F$ into $K$ over $E$ is uniquely determined by its value at $X \in F$ which has to be a root of $g$. Conversely, any root of $g$ in $K$ gives rise to a unique $E$-embedding of $F$ into $K$. Therefore, $$|\mathrm{Hom}_E(F,K)| = |\{x \in K : g(x) = 0 \} |$$ and the corollary follows. 
\end{proof}

\begin{corollary} \label{NoTrCor}

Let $n \geq 3$, set $g(X) \defeq X^{n} + X^{n-1} + \dots + X - 1$, and let $L \leq \mathbb{Q}[X]/(g(X))$ be a totally real subfield. Then $L = \mathbb{Q}$.

\end{corollary}

\begin{proof}
Set
\begin{equation*}
\begin{split}
h(X) &\defeq -X^ng(\frac{1}{X}) = X^n - X^{n-1} - ... - X - 1, \\ F &\defeq \mathbb{Q}[X]/(g(X))
\end{split}
\end{equation*} 
and note that the splitting field of $g(X)$ over $\mathbb{Q}$ is the same as that of $h(X)$ since the roots of the former are just the inverses of the roots of the latter. By \cite[Corollary 2.2]{M}, $h(X)$ is irreducible over $\mathbb{Q}$, so $g(X)$ is irreducible over the rationals as well.

Suppose first that $n$ is even. By \cite[Theorem 2.6]{M}, the Galois group of $h(X)$ over $\mathbb{Q}$ is $S_n$, so the same holds for $g(X)$. By the Galois correspondence, $F$ corresponds to a stabilizer of a point in $S_n$ which can be naturally identified with $S_{n-1}$, and is readily seen to be a maximal subgroup of $S_n$. Since $L$ corresponds to a subgroup of $S_n$ containing the aforementioned stabilizer, it follows that either $L = F$ or $L = \mathbb{Q}$. By the arguments preceding \cite[Theorem 2.1]{M}, $h(X)$ has exactly two real roots, so as usual, $g(X)$ inherits this property. Since $n \geq 3$, $g(X)$ has at least one complex root not in $\mathbb{R}$. We infer that $F$ is not totally real, so necessarily $L = \mathbb{Q}$.

Suppose now that $n$ is odd, and set $E \defeq \mathbb{Q}, \ p = 2$. Let $\bar{E}$ be the set of all algebraic numbers in $\mathbb{C}$ (this is an algebraic closure of $E$) and put $K \defeq \bar{E} \cap \mathbb{R}$. Observe that $G_K \cong \mathbb{Z}/2\mathbb{Z}$ which is a pro-$p$ group. By \corref{PrimCor}, $[L : \mathbb{Q}]$ is bounded from above by the number of real roots of $g(X)$ which is $1$ as shown just before \cite[Theorem 2.1]{M}. It follows that $L = \mathbb{Q}$. 
\end{proof}

As we have already mentioned earlier, Theorem \ref{thm: no real subfields} reformulates \corref{NoTrCor}. In this theorem, the field $F$ is realized as the field obtained by adjoining to $\mathbb{Q}$ the unique root of $g(X)$ in $[0,1]$. Clearly, this field is isomorphic over $\mathbb{Q}$ to $\mathbb{Q}[X]/(g(X))$ (the field in \corref{NoTrCor}), so there exists a one to one correspondence between their totally real subfields. 


\subsection{Acknowledgments}

We are grateful to Barak Weiss for telling us about the problem of finding the maximal totally real subfields of number fields arising in dynamics. Special thanks go to Patrick Hooper whose computer verification of \corref{NoTrCor} for all $n \leq 1000$ greatly stimulated our work. 
We would also 
like to thank Moshe Jarden for his comments on drafts of this work. The first and second authors were partially supported by the Israel Science Foundation grant no. 952/14. The third author was partially supported by the ERC-Advanced grant “Diophantine problems" (grant agreement no. 267273).

\section{Erratum}
\name{erratum}
Our arguments in \S \ref{sec:  minimal sets} of this paper contain an error. In this note we
explain the error and how to fix it. We are grateful to Florent Ygouf
for both pointing out the mistake, and for indicating the correct
argument included below.
All of the results stated in the introduction of the paper
remain valid, as a consequence of the amended argument which will be
given below. 
In the recent preprint \cite{Y20}, Ygouf proves related results about
existence of dense rel leaves in other loci.  

\subsection{The error}
The second assertion of Proposition \ref{prop: real rel closure} is wrong. Namely, the tangent
space to $\mathcal{O}(x)$ is not a $\Q$-space. This affects the
validity of Lemma \ref{lem:tangent space}: while it is true that $T_r$ and $\mathcal{O}_r$ are $\mathbf{g}$-dimensional, it is not true that $T_r$ is a
$\Q$-subspace. The rationality of $T_r$ is used once in the paper, in
the proof of 
Theorem \ref{thm: crucial}, at the end of the second paragraph of the proof.

\subsection{The area form}
Let $\Theta$ be the $\R$-valued anti-symmetric bilinear form on
$H^1(S, \Sigma; \R)$ defined by 
$\Theta(\beta_1, \beta_2) = \int_S  \beta_1\wedge \beta_2$. This
bilinear form restricts to the intersection form on $H^1(S ; \R)$ and
can also be used to compute the flat area of a surface, namely the
area of the translation surface $x_0$ is $\Theta(\hol_x(x_0),
\hol_y(x_0))$. See \cite[\S
3.3]{FM} for more information. Since $\Theta$ is defined purely in
topological terms,  for any homeomorphism $\varphi: (S, \Sigma) \to (S, \Sigma)$ we have
$\Theta(\varphi^* \beta_1, \varphi^*\beta_2) = \Theta (\beta_1,
\beta_2)$. 

\subsection{Correcting the proof of Theorem \ref{thm: crucial}}
Retain the notation of \S \ref{sec:  minimal sets}. 
We need to justify the claim made in the second paragraph of the proof
of Theorem \ref{thm: crucial}, that $T_r \not \subset P_2$, where $P_1 = \spa
\{(\hol_y(x_0), 0)\}$ and $P_2$ is the subspace of $P$ which is
$\varphi^*$-invariant and complementary to $P_1$.
Since $\Theta(\hol_x(x_0), \hol_y(x_0))$
is equal to the area of the surface $x_0$, it is not 0, and thus
\eq{eq: for direct computation}{P_2 = \ker F
\quad \text{for} \quad 
F:P \to \R, \ \quad F(\beta,0) = \Theta (\hol_x(x_0), \beta).}
We have $\dim P =
\mathbf{g}+1, \, \dim P_2 = \mathbf{g},$ and by Lemma \ref{lem:tangent space},   $\dim T_r
= \mathbf{g}$. 

Our proof proceeds by contradiction. Suppose $T_r \subset P_2$. Then by dimension considerations, $T_r = P_2$, and in particular,
since $T_r$ is the tangent space to the torus ${\mathcal O}_r \subset {\mathcal H}$, 
$\Psi^{-1}({\mathcal O}_r)$ is a subtorus of $\R^{\mathbf{g}+1}/\Z^{\mathbf{g}+1}$, where $\Psi$ is 
the map defined in \eqref{def:Psi}. Let $\til \Psi$ be the map on p. \pageref{def:Psi}. The tangent space to $\Psi^{-1}({\mathcal O}_r)$
is then 
$D[\til \Psi^{-1}](T_r)$. Since this tangent space is parallel to the subtorus $\Psi^{-1}({\mathcal O}_r)$ of $\R^{\mathbf{g}+1}/\Z^{\mathbf{g}+1}$,
any normal vector to $D[\til \Psi^{-1}](T_r)$ must be proportional to a rational vector.

We claim that a normal vector to $D[\til \Psi^{-1}](T_r)$
is given
by $\left(c_0^2, c_1^2, \ldots, c_{\mathbf{g}}^2\right)$. To see this,
observe that the function $F \circ D \til \Psi$ vanishes on $D[\til
\Psi^{-1}](T_r)$.  
By Proposition \ref{prop:Phi derivative} and \eqref{def:Psi}, we have
$$D \til \Psi (t_0, \ldots, t_{\mathbf{g}}) =\left (t_0 c_0 C_0^\ast +
\ldots + t_\mathbf{g} c_\mathbf{g} C_\mathbf{g}^\ast,0 \right).$$ 
Thus,
$$F \circ D \til \Psi (t_0, \ldots, t_{\mathbf{g}})=
\sum_{i=0}^\mathbf{g} t_i c_i \,\Theta\big(\hol_x (x_0),
C_i^\ast\big).$$
For $r>0$, the surface  $x_r$ is made of $\mathbf{g}+1$ horizontal
cylinders as in the discussion in the first paragraph of \S \ref{sect: twist coordinates}. The
definition of 
$C_i^*$ then implies that $C_i^*(\sigma_j) = \delta_{ij}$ (Kronecker
delta) and $C_i^*$ assigns 0 to all other saddle connections on
boundaries of cylinders. The area of $x_r$ can be computed by adding
the areas of the cylinders, and thus, by the interpretation of $\Theta$
as an area, we see that $\Theta\big(\hol_x (x_r), C_i^\ast\big) =
c_i,$ the horizontal holonomy of the core curve of $C_i$ on
$x_r$. Since the surfaces $x_0$ and $x_r$ assign the same horizontal
holonomy to absolute periods, we also have 
$\Theta\big(\hol_x (x_0), C_i^\ast\big)=c_i,$ and therefore
$$F \circ D \til \Psi (t_0, \ldots, t_{\mathbf{g}})=
\sum_{i=0}^\mathbf{g} t_i c_i^2.$$
Since we have seen that $T_r = \ker F$, this proves the claim. 

We have shown that $\left(c_0^2, c_1^2, \ldots,
  c_{\mathbf{g}}^2\right)$ is proportional to a rational vector.  
On the other hand, by Theorem \ref{thm:completely periodic ray} we have 
$\frac{c_1}{c_0} = \alpha$, and so $\alpha^2 = \frac{c_1^2}{c_0^2} \in
\Q$. But, this contradicts that the fact that $\dim_\Q \Q(\alpha) =
\mathbf{g} \geq 
3.$

\markleft{\uppercase{Hooper and Weiss}}

\end{document}